\documentclass[a4paper,10pt,twoside,reqno]{amsart}

\usepackage[left=2cm,right=2cm,top=2cm,bottom=2cm]{geometry}
\usepackage{verbatim}
\usepackage[utf8]{inputenc}
\usepackage[OT1]{fontenc}
\usepackage[english]{babel}
\usepackage{amsfonts}
\usepackage{amsmath}
\usepackage{amssymb}
\usepackage{amsthm}
\usepackage{amsaddr}
\usepackage{enumerate}
\usepackage{color}
\usepackage{esint}
\usepackage{hyperref}
\usepackage{cite}
\usepackage{mathrsfs}
\usepackage{bbm}
\usepackage{graphicx}
\usepackage{cancel}
\usepackage{tikz-cd}

\makeatletter
\renewcommand{\section}{\@startsection
  {section}{1}{\z@}%
  {-3.5ex \@plus -1ex \@minus -.2ex}
  {1.5ex \@plus .2ex}
  {\centering\normalfont\Large\scshape\bfseries}} 
\makeatother

\usepackage{todonotes}

\newcommand{\del}{\partial} 

\newcommand{\ee}{\varepsilon}

\newcommand{\EE}{\mathcal{E}}

\newcommand{\ZZ}{\mathbb{Z}}
\newcommand{\R}{\mathbb{R}}
\newcommand{\NN}{\mathbb{N}}

\newcommand{\no}[2]{ \left\| #1 \right\|_{#2} }

\newcommand{\Ccal}{\mathcal{C}}

\newcommand{\Rcal}{\mathcal{R}}
\renewcommand{\Re}{\operatorname{Re}}
\renewcommand{\Im}{\operatorname{Im}}
\newcommand{\im}{{i}}

\newcommand{\erf}{\, \mathrm{erf}}

\usepackage{amsthm}
\usepackage{aliascnt}
\usepackage{enumitem}

\usepackage{cleveref}
\usepackage{hyperref}

\newtheorem{theorem}{Theorem}[section]

\newaliascnt{lemma}{theorem}
\newtheorem{lemma}[lemma]{Lemma}
\aliascntresetthe{lemma}
\crefname{lemma}{Lemma}{Lemmas}
\Crefname{lemma}{Lemma}{Lemmas}

\newaliascnt{prop}{theorem}
\newtheorem{prop}[prop]{Proposition}
\aliascntresetthe{prop}
\crefname{prop}{Proposition}{Propositions}
\Crefname{prop}{Proposition}{Propositions}

\newaliascnt{cor}{theorem}

\aliascntresetthe{cor}
\crefname{cor}{Corollary}{Corollaries}
\Crefname{cor}{Corollary}{Corollaries}

\newaliascnt{definition}{theorem}
\newtheorem{definition}[definition]{Definition}
\aliascntresetthe{definition}
\crefname{definition}{Definition}{Definitions}
\Crefname{definition}{Definition}{Definitions}

\newaliascnt{remark}{theorem}
\newtheorem{remark}[remark]{Remark}
\aliascntresetthe{remark}
\crefname{remark}{Remark}{Remarks}
\Crefname{remark}{Remark}{Remarks}

\newaliascnt{hypotheses}{theorem}
\newtheorem{hypotheses}[hypotheses]{Hypotheses}
\aliascntresetthe{hypotheses}
\crefname{hypotheses}{Hypotheses}{Hypotheses}
\Crefname{hypotheses}{Hypotheses}{Hypotheses}

 \DeclareFontFamily{OT1}{rsfs}{}
 \DeclareFontShape{OT1}{rsfs}{m}{n}{ <-7> rsfs5 <7-10> rsfs7 <10-> rsfs10}{}
 \DeclareMathAlphabet{\mycal}{OT1}{rsfs}{m}{n}

\newcommand\smallO{
  \mathchoice
    {{\scriptstyle\mathcal{O}}}
    {{\scriptstyle\mathcal{O}}}
    {{\scriptscriptstyle\mathcal{O}}}
    {\scalebox{.7}{$\scriptscriptstyle\mathcal{O}$}}
  }

\begin{document}

\title[ ]{\textbf{Hadamard Ill-Posedness  of the linearised  Prandtl Equations in Gevrey spaces 
}}

\author[
    F. De Anna\quad
    J. Kortum
]{
    Francesco De Anna$^1$ \quad
    Joshua Kortum$^2$
}

 \address{
 $\,^{1,\,2}$ Julius-Maximilians-Universität Würzburg, Institute of Mathematics,
Chair for Mathematics in the Sciences, Emil-Fischer-Stra\ss e 40, 97074 Würzburg, Germany
 \\
 \medskip
 $\,^1$francesco.deanna@uni-wuerzburg.de\\
 $\,^2$joshua.kortum@uni-wuerzburg.de
}

\begin{abstract}
 
We prove Hadamard ill-posedness for the non-autonomous linearised Prandtl equations around time-dependent shear-flow equilibria in function spaces up to Gevrey class 4. More precisely, we construct compactly supported smooth initial data, Gevrey class $4$ in the tangential variable, for which the system admits no weak solution for any positive lifespan. In this regard, we improve previous results by showing that classical semigroup-type instabilities do not, by themselves, 
imply Hadamard ill-posedness in the non-autonomous case when the initial time is not a variable of the system. 

\smallskip 
\noindent 
Our argument is based on a family of exact unstable modes whose $L^2$-norms grow, at tangential  frequency $k$, like $\exp(c \sqrt k t)$ up to times of order $t \sim k^{-1/4}$. Their construction relies on an inner–outer gluing scheme, in the spirit of matched asymptotic expansions, which combines unstable inner solutions near a non-degenerate critical point with an exact outer solution and yields exponentially small matching errors for short times. 
\end{abstract}

\maketitle	


\section{Introduction}
%
\noindent 
The Prandtl equations are a classical asymptotic model in fluid dynamics describing the leading-order dynamics of boundary layers in the vanishing viscosity limit of the Navier-Stokes equations. Among the various mathematical challenges of rigorously justifying their asymptotics \cite{MR3634071,MR3961300,MR4592099,MR4715239,MR1761409,MR3566199}, a central problem concerns the well-posedness of the Prandtl equations themselves.

\smallskip 
\noindent 
The well-posedness theory reflects, in part, the physics of boundary layer separation since it falls into two categories. Under a monotonicity assumption on the initial velocity, the equations are well-posed in Sobolev spaces  \cite{AWXY2015,MR3765768,oleinik1963prandtl,MW2015,LY2020}; without monotonicity, well-posedness requires high tangential regularity, typically analytic or so-called Gevrey classes \cite{MR4465902,MR3855356,MR3429469,lombardo2003well}, with Gevrey-2 appearing as a borderline case \cite{MR3925144,MR4773381}.  

\smallskip 
\noindent 
These regularity requirements are supported by the fact that the Prandtl equations, when linearised around suitable non-monotonic profiles, are ill-posed in Sobolev spaces  \cite{MR2952715,MR2849481}. To the best of our knowledge, this is the first ill-posedness result in any Gevrey class for the linearised Prandtl equations. In this work, we establish ill-posedness for the linearised equations in Gevrey-4 spaces and provide some clarifications concerning the Sobolev instabilities.

\smallskip 
\noindent 
We focus on the two-dimensional Prandtl model linearised around a time-dependent shear flow $(u,v) = (u_s,0)$. Here, the shear map $u_s= u_s (t,y)$ is a smooth solution of the heat equation
\begin{equation}\label{eq:heat-equation}
    \partial_t u_s  - \partial_y^2 u_s  = 0,\qquad \qquad 
    u_s |_{y = 0} = \lim\limits_{y \to + \infty} u_s = 0,\qquad \qquad
    u_s |_{t = 0} = U_s,
\end{equation}
with $t>0$, $y \in \mathbb R_+ = (0, \infty)$ and where  $U_s= U_s(y)$ is an appropriate  non-monotonic profile (cf.~\Cref{Hypotheses-US-us}). The shear flow $(u_s,0)$ is a smooth solution of the nonlinear Prandtl equations and the corresponding linearisation is non-autonomous, taking the form:
\begin{equation}\label{eq:lin-Prandtl}
   \begin{cases}
       \, \partial_t u  + u_s \, \partial_x u   + v \,\partial_y u_s 
       - \partial_y^2 u  = 0
       \qquad 
       &\!\!(t,x,y)\in(0,\delta) \times \mathbb T \times \mathbb R_+,
       \\
       \, \partial_x u + \partial_y v = 0
       \qquad 
       & 
       \phantom{(t,x,y) \in }(0,\delta) \times \mathbb T \times \mathbb R_+,
       \\
       (u,v)\big|_{y = 0} =(0,0),\quad \lim\limits_{y \to +  \infty} u = 0 
       \qquad 
       & 
       \phantom{(t,x,y) \in } (0,\delta) \times \mathbb T,
        \\
       \,u\,|_{t = 0} = u_{\rm in}
       \qquad 
       & 
       \phantom{\,(t,x,y) \in  (0,\delta) \times } \mathbb T \times \mathbb R_+,
   \end{cases} 
\end{equation}
for a velocity field $(u,v)$ and an arbitrary (small) lifespan $\delta >0$. 

\medskip  
\noindent 
When $u_s$ depends on time, \cite{MR2601044} established a seminal result regarding the ill-posedness of \eqref{eq:lin-Prandtl} in Sobolev spaces: in general, System \eqref{eq:lin-Prandtl} does not generate a uniformly continuous two-parameter semigroup acting in Sobolev regularities. More specifically, under sufficient regularity and decay of $u_s$, the equations are locally well-posed for $u_{\rm in}$ with Fourier modes in $x \in \mathbb T$ that are weighted Lebesgue/Sobolev in $y \in \mathbb R_+$ and whose norms decay at an analytic rate (cf.~Proposition~1 in \cite{MR2601044}). If one  denotes by $T(\tau,t)(u_{\rm in}) = u(t, \cdot)$ the underlying solution operator with $u|_{t = \tau} = u_{\rm in} $ imposed at a time $\tau \in [0, \delta)$, Theorem 1-$(ii)$ in \cite{MR2601044} shows that there exist non-monotonic profiles $u_s(t,y)$ for which $T$ does not extend uniformly as an operator in Sobolev regularities, since it satisfies:
\begin{equation}\label{eq:result-of-GV}
    \sup_{0 \leq \tau \leq t \leq \delta } 
    \big\| 
        \,e^{- \sigma(t-\tau) \sqrt{|\partial_x|}} T(\tau,t) \,
    \big\|_{\mathcal{L}(H^{m_1}, H^{m_2} )} = + \infty,
    \quad \forall m_1,\, m_2 \geq 0,\quad 
    \;  
    H^m = H^m (\mathbb T, L^\infty(\mathbb R_+, e^{\lambda y} dy )),
\end{equation}
for any $\sigma\in [0, \sigma_0)$, with given $\sigma_0>0$, and any $\delta>0$, $\lambda >0$. By this identity and the Banach–Steinhaus theorem, there exists a unit-norm initial datum $u_{\rm in}$ in $H^{m_1}$ such that $T(\tau_n,t_n)(u_{\rm in})$ diverges in $H^{m_2}$ for some sequence $(\tau_n, t_n)\to (0, 0)$ with $0\leq \tau_n \leq t_n$. In other words, continuity at $(\tau,t) = (0, 0)$ fails  in any Sobolev space, when the initial time  $\tau \in [0, \delta)$ is also considered a variable of the system.

\smallskip 
\noindent
Similar relations as in \eqref{eq:result-of-GV} were subsequently extended to three dimensions \cite{MR3458159}, MHD boundary layers \cite{MR3864769}, and to the case of lower decay of $u_s$ as $y \to \infty$ \cite{MR3670620}.

\smallskip 
\noindent 
There are, however, several examples of non-autonomous PDEs that define a global-in-time, two parameter semigroup $S = S(\tau,t)$ in Sobolev spaces, which is not uniformly continuous, yet satisfies the relations:
\begin{equation}\label{eq:intro-S-semigroup-relations}
    \sup_{0 \leq t < + \infty} 
    \big\| 
        \,S(0,t) \,
    \big\|_{\mathcal{L}(H^{m_1}, H^{m_1} )} \leq 1 ,\quad 
    \qquad 
     \sup_{0 \leq \tau \leq t \leq \delta } 
    \big\| 
        \,e^{- \sigma(t-\tau) \sqrt{|\partial_x|}} S(\tau,t) \,
    \big\|_{\mathcal{L}(H^{m_1}, H^{m_2} )} = + \infty, 
\end{equation}
for any $m_1,\, m_2\in \mathbb R$, $\sigma \in [0,1)$ and $\delta >0$. A simple toy-model is presented in \Cref{sec:toy-model}. In particular, if the initial time is not treated as a variable, for instance~$\tau = 0$ is fixed, these PDEs are well-posed in any Sobolev spaces in the sense of Hadamard. As a natural question: does System \eqref{eq:lin-Prandtl} fall into this category, or does it instead exhibit Hadamard ill-posedness even when the initial time is fixed? We determine that the latter holds.

\smallskip
\noindent
We emphasize that this issue is exclusive to the non-autonomous case of System~\eqref{eq:lin-Prandtl}. If $u_s(t,y) = U_s(y)$, indeed, Gérard-Varet and Nguyen~\cite{MR2952715} proved the existence of non-monotonic $U_s$ and smooth, fast decaying initial data $u_{\rm in}$ that do not generate weak solutions. However, as we discuss in \Cref{remark:GV-Nguyen}, their approach does not extend directly to the non-autonomous setting of~\eqref{eq:lin-Prandtl}.   



\subsection{Main Results}$\,$

\smallskip 
\noindent 
Our approach combines a refined construction of unstable quasi-eigenmodes with a superposition argument adapted to Gevrey regularity. 
Inspired by~\cite{MR2601044} and \cite{DeAnnaKortum2025}, we revisit a special class of unstable quasi-eigenmodes $u(t,x,y) = u_k(t,y)e^{i k x}$, at fixed tangential frequencies $k \in \mathbb N$. Our first main result (cf.~\Cref{thm:inflating-solutions}) establishes the existence of exact quasi-eigenmode solutions (as opposed to approximate ones as in \cite{MR2601044}) to  System \eqref{eq:lin-Prandtl} that exhibit growth of order $e^{c \sqrt{k}\,t}$ persisting at least up to times of order $t \sim k^{-\frac{1}{4}}$. By superposition,  our second result exploits their maximal growth $e^{c \sqrt[4]{k}}$ to construct smooth, compactly supported initial data which are Gevrey class $4$ in $x\in \mathbb T$, and for which no weak solution to \eqref{eq:lin-Prandtl} exists for any lifespan $\delta >0$ (cf.~\Cref{def:weak-solution-Prandtl} and \Cref{thm:non-existence-of-solutions}).

\smallskip 
\noindent 
Our results rely on structural assumptions on the initial profile $u_s|_{t =0} = U_s$. We first introduce  the weighted Sobolev spaces $H^m_\lambda(\mathbb R_+)$, with $\mathbb R_+ = (0, \infty)$, defined for $m \in \mathbb N$ and $\lambda >0$:
\begin{align*}
    H^{m}_\lambda(\mathbb{R}_+) := 
    H^m(\mathbb{R}_+,\, e^{\lambda y} dy ) =
    \Big\{ \, f \in H^m(\mathbb{R}_+, \mathbb R)\; \big|\; 
    f^{(j) } \in L^2_\lambda(\mathbb R_+) := L^2(\mathbb R_+, e^{\lambda y} dy),\; \forall j  = 0,\dots, m \Big\},
\end{align*}
endowed with its natural norm
\begin{equation*}
    \| \, f \, \|_{H^{m}_\lambda}
    := \sum_{j = 0}^m\| \,f^{(j)} \,\|_{L^2_\lambda}
    = \sum_{j = 0}^m \Bigg( \int_{\mathbb R_+} |f^{(j)}(y)|^2 \,e^{2\lambda y} dy \Bigg)^\frac{1}{2}.
\end{equation*}
The space $H^m_\lambda(\mathbb R_+, \mathbb C)$  for complex-valued functions is defined analogously.  
\begin{hypotheses}\label{Hypotheses-US-us} We assume 
the following conditions for the initial profile $u_s|_{t = 0} = U_s$:
\begin{itemize}
\item[(H1)] There exist $m \in \mathbb N\setminus\{1,\,2\}$ and $\lambda >0$ such that $U_s\in H^{m+1}_\lambda(\mathbb{R}_+)$.
\item[(H2)] $U_s(0) = U_s''(0) = \dots = U_s^{2 \lfloor \frac{m}{2} \rfloor }(0) =  0$.
\item[(H3)] There exist $a_0>0$ and $0< d < a_0$ such that $U_s$ admits the following quadratic representation:
\begin{equation*}
    U_s(y) = U_s(a_0) + \frac{U_s''(a_0)}{2} (y-a_0)^2,\qquad \text{for any}\quad y \in \left[ a_0 - d, a_0 + d\right],
\end{equation*}
for some given $ U_s(a_0) \in \mathbb R$ and $U_s''(a_0)<0$. 
\end{itemize}
\end{hypotheses}
\noindent 
Under these assumptions, the unique  solution $u_s=u_s(t,y)$ to  \eqref{eq:heat-equation} is smooth and satisfies (cf.~\Cref{lemma:Green-formula-in-Hnmu})
\begin{equation*}
    u_s \in 
    \mathcal{C}^\infty( \mathbb R_+ \times \mathbb R_+ ) \cap 
    \mathcal{C}([0, \infty[, H^{m+1}_\lambda(\mathbb R_+)) \cap  \mathcal{C}^1([0, \infty[, H^{m-1}_\lambda(\mathbb R_+)),
\end{equation*}
which in particular implies exponential decay in  $y$ for each $t \geq 0$. As a consequence,  projecting \eqref{eq:lin-Prandtl} onto Fourier modes yields, for each $k \in \mathbb{Z}$, the problem
\begin{equation}\label{eq:lin-Prandtl-projected-k}
   \begin{cases}
       \, \partial_t u_k  + i k \, \big(  u_s \, u_k   -  \partial_y u_s \int_0^y u_k  \big)
       - \partial_y^2 u_k  = 0
       \qquad 
       &\!\!(t,y)\in(0,\delta) \times \mathbb R_+,
       \\
       \,u_k |_{y = 0} = \lim\limits_{y \to +  \infty} u_k = 0 
       \qquad 
       & 
       \phantom{(t,y) \in } (0,\delta),
        \\
       \,u_k\,|_{t = 0} = u_{{\rm in},k}
       \qquad 
       & 
       \phantom{\,(t,y) \in  (0,\delta) \times }  \mathbb R_+,
   \end{cases} 
\end{equation}
For each fixed $k$, this system can be treated as a linear perturbation of the heat equation on the half-line with Dirichlet boundary conditions, and for any $u_{{\rm in},k} \in H^1_0(\mathbb R_+)$ it admits a unique solution
\begin{equation}\label{eq:fct-space-uk}
    u_k \in \mathcal{C}([0, \delta], H^1_0(\mathbb R_+, \mathbb C)) \cap L^2(0, \delta ; H^2(\mathbb R_+, \mathbb C))\cap W^{1,2}(0, \delta; L^2(\mathbb R_+, \mathbb C)),
\end{equation}
for any $\delta>0$ (cf.~the a-priori estimate \eqref{eq:control-of-int0y-deltauk}).
%
%
The key assumption for the existence of unstable $u_k$ is (H3) in \Cref{Hypotheses-US-us}. It ensures the existence of a non-degenerate critical point at $y = a_0>0$, i.e.~$U_s'(a_0) = 0$ and $U_s''(a_0)\neq 0$.  
We stress, however, that the quadratic structure  near $y = a_0$ is essential for our construction, and not merely the existence of a non-degenerate critical point (see \Cref{sec:novelties-some-remarks}).

\smallskip 
\noindent 
Let us introduce the explicit form of the unstable initial datum $u_{\rm in, k}  =\frac{d}{dy}\phi_{{\rm unst}, k}$ for System \eqref{eq:lin-Prandtl-projected-k}. Its construction relies on an entire function arising in \cite{DeAnnaKortum2025} in the analysis of the Prandtl equations around a quadratic shear flow.
\begin{definition}\label{def:uink-into}
Let  $\Upsilon:\mathbb C \to \mathbb C$ be defined by
\begin{equation*}
    \Upsilon(\zeta):= 
    \frac{1}{\sqrt{2\pi}}\,\zeta\, e^{-\frac{\zeta^2}{2}} +
    \frac{1}{2}  
    \big(1+\zeta^2\big) 
    \left(1 + \erf\left( \frac{\zeta}{\sqrt{2}} \right) \right).
\end{equation*}  
Moreover, let $H$ denote the Heaviside function and let $\chi \in \mathcal{C}^\infty_c(\mathbb R_+)$ be a cutoff function satisfying 
\begin{equation*}
    \chi(y) = 1 \quad \text{for all } y\in \Big[a_0-\frac{d}{4}, a_0+\frac{d}{4}\Big]
    \qquad \text{and}\qquad 
    \chi(y) = 0\quad \text{for all } y \in \mathbb R_+\setminus \Big[a_0-\frac{d}{2}, a_0+\frac{d}{2}\Big].
\end{equation*}
For any $k \in \mathbb N$, we define the unstable stream function $\phi_{{\rm unst}, k}:\mathbb R_+ \to \mathbb C$ as follows: 
\begin{equation}\label{eq:intro-def-phiinnk}
\begin{aligned}
        \phi_{{\rm unst}, k}(y):= 
        &\chi(y) \,
        \Upsilon
        \Bigg( (y-a_0) \, e^{-\frac{i\pi}{8}} \sqrt[4]{ k \frac{|U_s''(a_0)|}{2}} \,\Bigg)
        \sqrt{\frac{|U_s''(a_0)|}{2k}}
        \, e^{i\frac{5\pi}{4}}
        \,+
        \\
        +&
        \Big(1-\chi(y) \Big)\, H(y-a_0) \bigg(\, U_s(y) -U_s(a_0) + e^{i\frac{5\pi}{4}}\sqrt{\frac{|U_s''(a_0)|}{2k}} \; \bigg).
\end{aligned}  
\end{equation}
\end{definition}
\noindent 
We defer further discussion of $\phi_{{\rm unst}, k}$ to the next section. For now, we note that it matches two profiles via the cutoff $\chi$: an inner layer based on $\Upsilon$ near $y = a_0$, and an outer layer at $y > a_0 + d/2$, determined by $U_s$. Moreover, $\Upsilon$ can be associated with the so-called shear-layer profile introduced in \cite{MR2601044} (cf.~Appendix \ref{sec:comparison-with-shear-layer}).

\smallskip 
\noindent 
The inner layer is the main source of instability (cf.~\Cref{lemma:quadratic-unstable-quasi-eigenmode}), and the following theorem shows that the outer layer and the matching scheme do not suppress the resulting growth, at least for short times $t \lesssim k^{-\frac{1}{4}}$.

\begin{theorem}\label{thm:inflating-solutions}
Let $u_s= u_s(t,y)$ be solution of the heat equation \eqref{eq:heat-equation} with $u_s|_{t= 0} = U_s$ satisfying \Cref{Hypotheses-US-us}.

\smallskip 
\noindent 
Consider the initial datum $u_{{\rm in}, k}= \frac{d}{dy} \phi_{{\rm unst}, k}: \mathbb R_+ \to \mathbb C$ defined in \Cref{def:uink-into}. Then the following statements hold true:
\begin{itemize}
    \item[(i)] One has $u_{{\rm in}, k} \in H^{m}_\lambda(\mathbb R_+, \mathbb C)$ and  $u_{{\rm in}, k}(y) = 0$ for any $y\in [0, a_0-d/2]$. Moreover, there exists a constant $\mathcal{D}_m = \mathcal{D}_m(U_s, d, \chi)>0$, depending only on $m$, $U_s$, $d$ and $\chi$, such that
\begin{equation*}
    \| \, u_{{\rm in}, k} \,\|_{H^{m}_\lambda }
    \leq \mathcal{D}_m e^{2\lambda a_0} \,  k^{\frac{2m-1}{4} },
    \qquad \forall \;k \in \mathbb N.
\end{equation*} 
\item[(ii)]  Let $u_k$  denote the unique classical solution of \eqref{eq:lin-Prandtl-projected-k} in the space \eqref{eq:fct-space-uk}. 
Then there exists $K \in \mathbb N$, depending only on $U_s$, $d$ and $\lambda$, such that for all $k \geq K$ and  all 
\begin{align*}
     t \in \left[ 0,\, 
       \frac{1}{\sqrt[4]{k}}
       \frac{d}{16(1 +\| U_s'' \|_{L^\infty} )}
       \right],
\end{align*}
one has the lower bound
\begin{align*}
        \big\| \, u_k(t , \cdot )\, \big\|_{L^2(\mathbb R_+)}
        \geq  
        \frac{1}{4} 
        \bigg( 
        \int_{a_0+\frac{d}{2}}^\infty | \, U_s'(y)\,|^2 dy \,\bigg)^\frac{1}{2}
        \exp \left( \frac{t}{2}  \sqrt{k \frac{|U_s''(a_0)|}{2} }\right).
    \end{align*}
\end{itemize}
\end{theorem}
\noindent 
Part $(i)$ shows that $u_{{\rm in}, k}$ inherits the regularity of $U_s' \in H^m_\lambda(\mathbb R_+)$ and satisfies the boundary conditions. 
Moreover, its norm grows at most polynomially in $k\in \mathbb N$. An exact form of $\mathcal{D}_m>0$ given in \Cref{prop:estimates-remainders-and-uf}. 


\smallskip 
\noindent 
Part $(ii)$ presents the key aspect of the theorem. The lower bound already holds in the unweighted space $L^2(\mathbb{R}_+)$-norm (not weighted) and all constants are explicit. 

\smallskip 
\noindent 
Since $U_s$ is quadratic in $[a_0-d, a_0+d]$, its derivative $U_s'$ does not vanish identically in $(a_0 + d/2, +\infty)$ and the lower bound is nontrivial.
We also observe that the exponential growth is proportional to $|\,U_s''(a_0)\,|^{1/2}$ and degenerates when $U_s''(a_0) = 0$,  highlighting the role of the non-degeneracy as in \cite{MR2601044}. 

\smallskip 
\noindent
In addition to this, Part $(ii)$ highlights the dependence of the maximal time of norm-inflation with respect to $d > 0$. In our construction, the lifespan over which instability persists, shrinks as the interval in (H3) of  \Cref{Hypotheses-US-us}  collapses to the critical point $y = a_0$. As a consequence, our result cannot be extended to more general non-monotonic profiles by a simple approximation argument. Moreover, we infer  that $K=K(U_s, d, \lambda)\in \mathbb N$ diverges as $d \to 0+$.

\smallskip 
\noindent 
Building upon the quasi-eigenmodes constructed in \Cref{thm:inflating-solutions}, we obtain, by superposition, a smooth initial datum for the system \eqref{eq:lin-Prandtl}, for which no weak solution exists in the following sense:
\begin{definition}\label{def:weak-solution-Prandtl}
    Let $u_{\rm in} \in L^2(\mathbb T\times \mathbb R_+)$. A weak solution to \eqref{eq:lin-Prandtl} is a function
    \begin{equation*}
        u \in L^\infty(0, \delta; L^2(\mathbb T \times \mathbb R_+)) \cap  
        L^2(0,\delta; H^{0,1}_0(\mathbb T \times \mathbb R_+)),\quad 
        \text{with}\quad 
        H^{0,1}_0(\mathbb T \times \mathbb R_+) = 
        \overline{\mathcal{C}_c^\infty( \mathbb T \times \mathbb R_+)}^{\| \cdot \|_{H^{0,1}}},
    \end{equation*}
    and $\| f \|_{H^{0,1}} =\|  f \|_{L^2}  +  \| \partial_y f \|_{L^2}$, 
    such that for all test functions $\varphi \in \mathcal{C}^\infty_c([0, \delta) \times \mathbb T \times \mathbb R_+)$, one has
    \begin{align*}
        \int_0^\delta \int_{\mathbb T} \int_{\mathbb R_+} 
        &\bigg( 
            u \, \partial_t \varphi  + u_s  \, u \,  \partial_x \varphi  - \partial_y u  \, \partial_y \varphi  
        \bigg) (t,x,y) 
        \, dy dx dt
        \, +
        \\
        &-
        \int_0^\delta \int_{\mathbb T} \int_{\mathbb R_+} 
         u_s(t,y)  \bigg( \int_0^y u 
        (t,x,z)        
        dz \bigg) 
        \partial_x \varphi (t,x,y)
        dy dx dt        
        =
        -
        \int_{\mathbb T} \int_{\mathbb R_+} u_{\rm in}(x,y) \varphi(0, x,y)dy dx.
    \end{align*}
\end{definition}
\noindent 
This corresponds to a distributional solution with natural energy regularity for the nonlinear Prandtl equations.
We now construct an initial datum that does not admit any weak solution, under the further assumption that $U_s \in \mathcal{C}^\infty_c (\mathbb{R}_+)$. In this case, conditions (H1) and (H2) in \Cref{Hypotheses-US-us} are satisfied for all $m \in \mathbb{N}\setminus\{1, 2\}$ and all $\lambda > 0$, and the ill-posedness holds true regardless of their values. 

\smallskip
\noindent 
Consider the  Fourier series
\begin{equation}\label{eq:intro-Uincomplex-series}
        \mathbf{U}_{\rm in}(x,y) := \sum_{k = 1}^\infty e^{-\sigma_0 \sqrt[4]{k}} \frac{d}{dy} \phi_{{\rm unst,k}} (y) e^{ \im k x}  \in \mathbb C,\quad (x,y) \in \mathbb T \times \mathbb R_+.    
\end{equation}
where $\sigma_0>0$ is  sufficiently small constant. 
By Part $(i)$ of \Cref{thm:inflating-solutions}, 
the series converges 
in $\mathcal{D}(\mathbb T \times \mathbb R_+, \mathbb C)$. Moreover, $\mathbf{U}_{\rm in}$ is Gevrey-class $4$ in $x\in \mathbb T$ and Sobolev in $y\in \mathbb R_+$ in the following sense: a function $f : \mathbb T \times  \mathbb R_+  \to \mathbb C$  belongs to $\mathcal{G}_\sigma^4(\mathbb T, H^m_\lambda(\mathbb R, \mathbb C))$ with $\sigma>0$ and $m\in \mathbb N_0$, if there exists a sequence $(f_k)_{k\in \mathbb Z} \subset H^m_\lambda(\mathbb R_+, \mathbb C)$ such that
\begin{equation*}
    \| \,f \,\|_{\mathcal{G}^4_\sigma H^m_\lambda }:= 
    \sup_{k \in \mathbb Z} \Big( e^{\sigma \sqrt[4]{k}} \| f_k \|_{H^m_\lambda} \Big)  < + \infty,\quad \text{and}\quad 
    f(x,y) = \sum_{k \in \mathbb Z} f_k(y)e^{i k x}
    \quad\text{for a.e.}\quad (x,y) \in \mathbb T \times \mathbb R_+.
\end{equation*}
\begin{theorem}\label{thm:non-existence-of-solutions}
  Let $u_s= u_s(t,y)$ be solution of \eqref{eq:heat-equation}, with initial datum $ U_s \in \mathcal{C}^\infty_c(\mathbb R_+)$ satisfying (H3) of \Cref{Hypotheses-US-us}.  
  Let 
  $$0< \sigma_0 <\, \frac{d|U''(a_0)|^{1/2}}{32(1+\no{U''_s}{L^\infty})}$$
  and let  $\mathbf{U}_{\rm in}$ be defined in \eqref{eq:intro-Uincomplex-series}. 
  Then $\mathbf{U}_{\rm in} \in \mathcal{G}^4_\sigma(\mathbb T, H^m_\lambda(\mathbb R_+, \mathbb C))$, for all $\sigma \in ( 0, \sigma_0)$, all $m\in \mathbb N$ and all $\lambda>0$. Moreover, 
  for at least one of the initial data
    \begin{equation*}
        u_{\rm in} \in \Big\{ \,{\rm Re}(\mathbf{U}_{\rm in}),\, 
        {\rm Im}(\mathbf{U}_{\rm in})\, \Big\}
        \subset 
        \mathcal{C}^\infty_c ( \mathbb T \times \mathbb R_+) 
        \cap 
        \bigg(
        \bigcap_{m \in \mathbb N}  
        \bigcap_{\lambda>0}  
        \bigcap_{0 < \sigma < \sigma_0 }  
        \mathcal{G}^4_\sigma\big(\mathbb T, H^m_\lambda(\mathbb R_+)\big)
        \bigg),
    \end{equation*}
    the system \eqref{eq:lin-Prandtl} admits no a weak solution in the sense of \Cref{def:weak-solution-Prandtl}.
\end{theorem}
\noindent 
In summary, \Cref{thm:non-existence-of-solutions} establishes Hadamard ill-posedness of the non-autonomous linearised Prandtl equations~\eqref{eq:lin-Prandtl} both in Sobolev spaces and Gevrey classes of order $4$. In contrast to  \cite{MR2601044}, this ill-posedness persists even when the initial time is no variable of the system.

\smallskip

\subsection{Structure of the proof}\label{sec:novelties-some-remarks}$\,$

\noindent 
The proof follows a perturbative strategy: we construct exponentially growing approximate solutions and show that the associated error can be controlled.
We outline the main steps in the proof of \Cref{thm:inflating-solutions}. The proof of \Cref{thm:non-existence-of-solutions} is deferred to Section \ref{sec:proof-of-second-main-thm}.
\begin{itemize}[leftmargin=0.7cm]
    \item[a)] In \Cref{sec:matched-asymptotics} and \Cref{sec:estimates-forced-terms}, we build a family of forced quasi-eigenmodes $u_k^{\rm fr}$ satisfying estimates analogous to those in \Cref{thm:inflating-solutions}. In particular, there exist constants $c_0,\, c_1>0$ such that $\| \, u_k^{\rm fr}(t, \cdot) \,\|_{L^2} \geq c_0 \, e^{c_1 \sqrt{k} \, t}$. Moreover, $u_k^{\rm fr}$  generates a forcing term $f_k(t,y) e^{c_1 \sqrt{k} \, t}$ with 
    \begin{equation}\label{intro:behaviour-forcing-term}
       \| \, f_k(t, \cdot)\,  \|_{L^2(\mathbb R_+)} =\mathcal{O}\Big( e^{-c_2/t} + e^{-c_3\sqrt{k}}\Big)
    \end{equation}
    for some constants $c_2,\, c_3>0$ (see \Cref{prop:remainder-exact-form} and \Cref{prop:estimates-remainders-and-uf}). The initial datum  satisfies $u_k^{\rm fr}(0, \cdot) $ coincides with $\frac{d}{dy} \phi_{{\rm unst}, k}$ introduced in \Cref{def:uink-into}.
    \item[b)] 
    In \Cref{sec:correcting-the-perturbed-solution}, with \Cref{prop:correction}, we estimate the correction $\delta u_k= u^{\rm fr}_k- u_k$ between the approximate solution $u^{\rm fr}_k$ and the exact solution $u_k$  with $\delta  u_k (0, \cdot) = 0$. Let $T_k(\tau, t)\in \mathcal{L}(L^2(\mathbb{R}_+), L^2(\mathbb{R}_+))$  denote the associated two-parameter semigroup on $L^2(\R_+)$. Then $u_k$ satisfies the Duhamel formula
    \begin{equation*}
        u_k(t,\cdot) = u_k^{\rm fr} (t, \cdot) - \delta u_k(t, \cdot),\quad \text{with}\quad \delta u_k(t,\cdot) = \int_0^t T_k(\tau, t) f_k(\tau, \cdot) e^{c_1 \sqrt{k} \, \tau }d\tau.
    \end{equation*}
    To control this term, we establish the upper bound 
    \begin{equation}\label{intro:upper-bound-of-Tk}
        \|\, T_k(\tau, t) \,\|_{\mathcal{L}(L^2(\mathbb R_+), L^2(\mathbb R_+))} \leq e^{c_4 \sqrt{k} (t-\tau)},
    \end{equation}
    for constants $c_4>c_1>0$, uniformly in $k\in \mathbb N$. Combining these ingredients yields
    \begin{align*}
        \|\, u_k(t, \cdot) \,\|_{L^2(\mathbb R_+)} 
        &\geq \|\, u_k^{\rm fr}(t,\cdot)\, \|_{L^2(\mathbb R_+)} - \|\, \delta u_k (t, \cdot)\, \|_{L^2(\mathbb R_+)} \\
        &\geq c_0 e^{c_1 \sqrt{k} \, t} - c_5  \int_0^t e^{c_4\sqrt{k}(t-\tau) + c_1 \tau} \big( e^{ - c_2/\tau } + e^{ - c_3 \,\sqrt{k}} \big)d\tau \\
        &\geq c_0 e^{c_1 \sqrt{k} \, t} - c_5 \big( e^{ c_4 \sqrt{k} \, t - c_2/t } + e^{ c_4 \sqrt{k} \, t - c_3 \,\sqrt{k}} \big) t, 
    \end{align*}
    for some constant $c_5>0$.
    \item[c)] We conclude the proof of \Cref{thm:inflating-solutions} in \Cref{sec:proof-of-first-theorem} by choosing a time interval of length $\sim k^{-\frac{1}{4}}$, for which the last negative term remains  $\smallO(1)$ as $k \to \infty$, whereas  the leading term $c_0 e^{c_1 \sqrt{k} \, t}$ dominates. For instance, one may take $t \leq (c_2/c_4)^{\frac{1}{2}} k^{-\frac{1}{4}}$ with $k\gg 1$. 

    \smallskip 
    \noindent
    Tracking the dependence of $c_0, \dots, c_5 > 0$ on $U_s$ and $d > 0$ yields the explicit bounds stated in \Cref{thm:inflating-solutions}.
\end{itemize}
We mention that we cannot reach an improved time $k$-scaling, since $c_4 > c_1$. 

\smallskip 
\noindent 
The construction of $u_k^{\rm fr}$ in Part-(a) is at the core of our approach. The main novelty is the asymptotics \eqref{intro:behaviour-forcing-term} of the forcing term:  G\'erard-Varet and Dormy \cite{MR2601044} determined, for any $N \in \mathbb N$, approximate solutions with  $f_k \in \mathcal{O}\big(t^N + k^{-\frac{N}{2}} \big) $, namely polynomially decaying both as $k\to \infty$ and $t \to 0$; We establish  that, if $u_s$ satisfies the heat equation, one can reach the asymptotics of  \eqref{intro:behaviour-forcing-term} (the essential factor for Part~(b) and Part~(c)).

\smallskip 
\noindent
The function $u_k^{\rm fr}$ is defined via a matched asymptotic expansion (e.g.~\cite{holmes2012introduction}). We set in \Cref{def:inner-outer-profiles}
\begin{equation*}
    u_k^{\rm fr}(t,y) = \partial_y \phi_{{\rm inn}, k}(t,y) + \partial_y \phi_{{\rm out}, k}(t,y),\qquad (t,y) \in [0, T_{\rm max}] \times \mathbb R_+,
\end{equation*}
for some $T_{\rm max}>0$ and where $\phi_{{\rm inn}, k}$ and $\phi_{{\rm out}, k}$ are layers localised near and far from $y = a_0$, respectively.  

\smallskip 
\noindent
The inner profile $\phi_{{\rm inn}, k}$ is supported in $y \in [a_0 - \tfrac{d}{2}, a_0 + \tfrac{d}{2}]$ and relies on an instability ansatz of the equations in the region $y \in [a_0 - \tfrac{d}{4}, a_0 + \tfrac{d}{4}]$, where $U_s$ is quadratic. The outer layer $\phi_{{\rm out}, k}$ is supported in $y \geq a_0 + \tfrac{d}{4}$ and exhibits a similar instability to $\phi_{{\rm inn}, k}$ for $y \geq a_0 + \tfrac{d}{2}$. 
Below, we provide some heuristics and refer to \Cref{sec:matched-asymptotics} for further details.


\smallskip 
\noindent 
The construction of $\phi_{{\rm out}, k}$  relies mainly on the following observation:  for any $\sigma_k \in \mathcal{C}([0,T_{\rm max}],\mathbb{C})$, then
\begin{equation*}
    \phi_{{\rm out},k}(t,y) 
    := 
    \bigg( u_s(t,y) + \frac{\sigma_k(t)}{ik} \bigg) 
    \exp \bigg( \int_0^t \sigma_k(\tau ) d \tau\bigg) \qquad (t,y) \in [0, T_{\rm max}] \times \Big[ a_0 + \frac{d}{2}, + \infty \Big[
\end{equation*}
satisfies the identity:
\begin{align*}
   \Big(\,
    \partial_t + ik \, u_s 
    - \partial_y^2 \,\Big) 
    &\partial_y \phi_{{\rm out},k} 
    - ik \, \partial_y u_s \, \phi_{{\rm out},k}
    = 
    \\
    &= 
    \bigg( 
    \partial_t \partial_y u_s  + \sigma_k \, \partial_y  u_s + i k \,u_s\, \partial_y u_s  - \partial_y^3 u_s 
    -
    i k \, \partial_y u_s 
    \Big( u_s + \frac{\sigma_k }{ik} \Big)
    \bigg) e^{ \int_0^t \sigma_k }
    = 0.
\end{align*}
Notably, $\partial_t \partial_y u_s - \partial_y^3 u_s = 0$, since $\partial_y u_s$ satisfies the heat equation. Due to the boundary conditions, these functions cannot define alone ``unstable'' quasi-eigenmodes on $y>0$. On the other hand, by first building an ``unstable'' inner layer, we can incorporate the corresponding $\sigma_k$ to the above expression. This is precisely possible because $u_s$ satisfies the heat equation. By contrast (e.g.~if $u_s(t,y)=U_s(y)$) an $\mathcal{O}(1)\cdot e^{c_1 \sqrt{k} t}$ remainder would  need to be corrected. In \cite{MR2601044}, the correction led to the mentioned  $\mathcal{O}(t^N + k^{-\frac{N}{2}})\cdot e^{c_1 \sqrt{k} t}$. If $u_s$ satisfies the heat equation, no correction is needed.

\smallskip 
\noindent 
The inner layer $\phi_{{\rm inn}, k}$ is more involved; its construction is based on two aspects: first, $u_s(0,y)= U_s(y)$ is quadratic in $y \in [a_0 - \tfrac{d}{4},\, a_0 + \tfrac{d}{4}]$; second, for $t>0$, the deviation of $u_s(t,y)$ from its second-order Taylor expansion becomes in the same region exponentially small as $t \to 0+$. 

\smallskip 
\noindent 
We recall from  \cite{DeAnnaKortum2025} that, for any constants $\sigma_k \in \mathbb{C}$,  $a >0$, $\alpha \in \mathbb R$ and  $\beta <0 $, the spectral ODE arising from the autonomous linear Prandtl equations around a quadratic shear profile
\begin{equation*}
    \bigg( \sigma_k  + i k \,\Big( \alpha + \beta (y-a)^2 \Big) \bigg) \phi_k'(y) - 2 i k\,  \beta (y-a) \phi_k(y) - \phi_k'''(y) = 0
\end{equation*}
admits explicit general solutions in terms of hypergeometric functions. Among these solutions, we focus on the one characterised by $\Upsilon$ of \Cref{def:uink-into}:
\begin{equation*}
\sigma_k = \sqrt{|\beta | k} \, e^{-\frac{\im \pi}{4}} - \im  k \,\alpha \in \mathbb C,\qquad 
    \phi_k(y) =  
    \Upsilon 
    \Big( \,(y-a) \, e^{-\frac{i\pi}{8}} \sqrt[4]{ k |\beta|} \, \Big)
    \sqrt{\frac{|\beta|}{k}} e^{i \frac{5 \pi}{4}}.
\end{equation*}
For simplicity, a concise proof of this assertion is provided in \Cref{lemma:quadratic-unstable-quasi-eigenmode}. As the real part of $\sigma_k$ is strictly positive, the corresponding spectral mode exhibits exponential growth in time. 

\smallskip 
\noindent 
Our strategy in \Cref{sec:matched-asymptotics} involves tracking the critical point $\partial_y u_s(t,a(t)) = 0$ with $a(0) = a_0$, for finite time $t\in [0, T_{\rm max}]$ and define $\phi_{{\rm inn}, k}$ in the region $(t,y) \in [0, T_{\rm max}] \times  [a_0 - \tfrac{d}{4},\, a_0+\tfrac{d}{4}]$ via the above $\Upsilon$-profile:
\begin{align*}
    \alpha(t) &:= u_s(t,a(t)),\qquad 
    \beta(t) := \partial_y^2u_s(t,a(t)),\qquad 
    \sigma_k(t) 
    := 
    \sqrt{|\beta(t) | k} \, e^{-\frac{\im \pi}{4}} - \im  k \,\alpha(t),\\
    \phi_{{\rm inn}, k}(t,y) 
    &:= 
    \Upsilon 
    \Big( \,(y-a(t)) \, e^{-\frac{i\pi}{8}} \sqrt[4]{ k |\beta(t)|} \, \Big)
    \sqrt{\frac{|\beta(t)|}{k}} e^{i \frac{5 \pi}{4}}
    \exp \bigg( \int_0^t \sigma_k(\tau) d \tau\bigg).
\end{align*}
As explained, this construction determines simultaneously $\phi_{{\rm out}, k}$. Unlike the outer layer, however, $\phi_{{\rm inn}, k}$ does not solve System \eqref{eq:lin-Prandtl}  in the region $(t,y) \in ]0, T_{\rm max}[ \times [a_0-\frac{d}{4},\, a_0+\frac{d}{4}]$, since $u_s(t,y)$ ceases to be quadratic at any $t>0$. Consequently, nontrivial remainders appear near $y = a_0$ at any $t>0$. 

\smallskip 
\noindent 
From the properties of the function $\Upsilon$, however, we establish that these remainders mainly depend on the deviation of $u_s $ from its second Taylor approximation (cf.~\Cref{prop:remainder-exact-form}-Part (a)). Thanks to the Green’s formula, they depend roughly on
\begin{equation*}
	u_s(t,y) -\alpha(t)  - \frac{\beta(t) }{2}(y-a(t))^2
    = 
    \frac{(y-a(t))^3}{6}
    \frac{1}{\sqrt{4\pi t}}
	\int_0^\infty \Big( e^{- \frac{|\xi(t,y) -w|^2}{4t}} - e^{- \frac{|\xi(t,y) + w|^2}{4t}} \Big) U_s^{(3)}(\omega)d\omega,
\end{equation*}
where $\xi(t,y)$ lies between $y$ and $a(t)$. The third derivative of $U_s $ is zero in $[a_0-d, a_0+d]$ and the last integral is therefore $\mathcal{O}(e^{-c_2/t})$ for some constant $c_2 > 0$ proportional to $d > 0$. Consequently, the remainders of $\phi_{{\rm inn}, k}$ near $y = a_0$ are estimated by the first asymptotics of \eqref{intro:behaviour-forcing-term}. Here the quadratic assumption becomes crucial: if $U^{(3)}_s \not\equiv 0$ near $y = a_0$, the integral would be $\mathcal{O}(t)$ at best. 

\smallskip 
\noindent 
Once the inner and outer layers are defined on $y \in [a_0 - \tfrac{d}{4}, a_0 + \tfrac{d}{4}]$ and $y \geq a_0 + \tfrac{d}{2}$, respectively, we match them in the remaining regions. The transition is handled via the cutoff $\chi$ of \Cref{def:uink-into}. 

\smallskip 
\noindent 
Multiplying the $\Upsilon$-profile by $\chi$, the inner layer $\phi_{{\rm inn}, k}$ extends to $(t,y) \in [0, T_{\rm max}] \times  \mathbb R_+$ as follows:
\begin{equation*}
    \phi_{{\rm inn}, k}(t,y) 
    = \chi(y)\Upsilon 
    \Big( \,(y-a(t)) \, e^{-\frac{i\pi}{8}} \sqrt[4]{ k |\beta(t)|} \, \Big)
    \sqrt{\frac{|\beta(t)|}{k}} e^{i \frac{5 \pi}{4}}
    \exp \bigg( \int_0^t \sigma_k(\tau) d \tau\bigg).
\end{equation*}
It satisfies $\phi_{{\rm inn}, k}(t, 0) = \partial_y \phi_{{\rm inn}, k}(t, 0) = 0$. Additionally, due to the exponential decay $\Upsilon(\zeta) \sim e^{-\zeta^2/2}$ in the sector $\arg(\zeta)\in ]\frac{3\pi}{4}, \frac{5\pi}{4}[$, the remainders generated by $\phi_{{\rm inn}, k}$ in the region $y \in [0, a_0 - \tfrac{d}{4}]$ are indeed of order $\mathcal{O}(e^{-c_3\sqrt{k}}) \cdot e^{c_1 \sqrt{k} t}$ as in \eqref{intro:behaviour-forcing-term}. 

\smallskip 
\noindent 
The transition  in the region $y \in [a_0+\tfrac{d}{4},\, a_0 + \tfrac{d}{2}]$ is more subtle. We globally extend $\phi_{\text{out},k}$ via the cutoff function:
\begin{equation*}
    \phi_{{\rm out},k}(t,y) 
    := 
    (1-\chi(y) )H(y-a(t)) \bigg( u_s(t,y) + \frac{\sigma_k(t)}{ik} \bigg) 
    \exp \bigg( \int_0^t \sigma_k(\tau ) d \tau\bigg).
\end{equation*}
Both $\phi_{\text{inn},k}$ and $\phi_{\text{out},k}$ are $\mathcal{O}(1) \cdot e^{c_1\sqrt{k} t}$ in  $(t,y) \in [0, T_{\rm max}] \times  [a_0+\tfrac{d}{4}, a_0 + \tfrac{d}{2}]$. However, from the properties of $\Upsilon$ in \Cref{def:uink-into}, and the fact that
\begin{equation*}
    \erf(\zeta) \to 1,\qquad \text{as}\qquad 
    \zeta\to \infty\qquad\text{with}\qquad 
    \arg(\zeta) = -\frac{\pi}{8},
\end{equation*}
the asymptotics of $\phi_{\text{inn},k}(t,y)$ in the region 
$(t,y) \in [0, T_{\rm max}] \times  [a_0+\tfrac{d}{4}, a_0 + \tfrac{d}{2}]$ as $k\to +\infty$ is given by
\begin{equation*}
\begin{aligned}
    \phi_{\text{inn},k}(t,y) 
    &\sim
   \chi(y) \left(
        1+ \sqrt{|\beta(t)|k}\,(y-a(t))^2e^{-\frac{\im \pi}{4}}
    \right)
    \frac{1}{2}
    \left(
        1 + 1 \right)
    e^{\im \frac{5\pi}{4}}
    \sqrt{\frac{|\beta(t)|}{2k}}
     \exp \bigg( \int_0^t \sigma_k(\tau) d \tau\bigg)
    \\
    &\sim
    \chi(y)\left(
        \frac{\sigma_k(t)}{ik} + u_s(t,a(t))  +
        \frac{\partial_y^2 u_s(t, a(t))}{2}\,(y-a(t))^2
    \right)
    \exp \bigg( \int_0^t \sigma_k(\tau) d \tau\bigg).
\end{aligned}
\end{equation*}
This term is balanced by the corresponding $\chi$-term in the outer layer $\phi_{\text{out},k}$: in the region $(t,y) \in [0, T_{\rm max}] \times  [a_0+\tfrac{d}{4}, a_0 + \tfrac{d}{2}]$, where the main transition occurs, the leading-order term of the remainders still depends on
\begin{equation*}
    \chi(y) 
    \Big( 
        u_s(t,a(t)) + 
        \frac{\partial_y^2 u_s(t, a(t))}{2}\,(y-a(t))^2 -u_s(t,y)
    \Big).
\end{equation*}
Once more, under the quadratic assumption and the fact that $u_s$ satisfies the heat equation, the order of these remainders is $\mathcal{O}(e^{-c_2/t})$.

\smallskip 
\noindent 
We conclude this introduction with a final comparison between our approach and a prior result.
\begin{remark}\label{remark:GV-Nguyen}
Our result extends that of \cite{MR2952715}. For the autonomous case $u_s(t,y) = U_s(y)$ of System \eqref{eq:lin-Prandtl},  Gérard-Varet and Nguyen identified initial data $u_{\rm in} \in e^{-y} H^\infty(\mathbb{T} \times \mathbb{R}_+)$ that, for any $T>0$, fail to generate distributional solutions  with $
   u \in  L^\infty(0, T; L^2(\mathbb T \times \mathbb R))$ and $\partial_y u \in  L^2((0,T) \times \mathbb T \times \mathbb R))$.  In this remark, we highlight why \Cref{thm:non-existence-of-solutions} is not a simple adaptation.

   \smallskip 
   \noindent 
   The proof in \cite{MR2952715} relies on a contradiction, assuming that \eqref{eq:lin-Prandtl} defines a continuous linear map
\begin{equation*}
    \mathcal{T} : e^{-y} H^\infty(\mathbb T\times \mathbb R_+) \to L^\infty(0, T; L^2(\mathbb T \times \mathbb R)) \times L^2((0,T) \times \mathbb T \times \mathbb R),\quad 
    \text{with}\quad \mathcal{T}(u_{\rm in}) = (u, \, \partial_y u).
\end{equation*}
Namely, any $u_{\rm in} \in e^{-y} H^\infty(\mathbb T \times \mathbb R)$ defines a unique $u$, and there are $K \in \mathbb{N}$ and $\mathcal{C} > 0$ (independent of $u_{\rm in}$) with
\begin{equation*}
    \sup_{t \in [0, T]} \| u(t) \|_{L^2_{x,y}} + \| \partial_y u \|_{L^2_{t, x, y}} \leq \mathcal C \| e^y u_{\rm in} \|_{H^{K}_{x,y}}.
\end{equation*}
We refer to page 5 and 6 of \cite{MR2952715}. The associated autonomous semigroup $S(t)(u_{\rm in}) = (\mathcal{T}(u_{\rm in}))_1(t) \in L^2_{x,y}$ satisfies for any $0 \leq s \leq t \leq T$:
\begin{equation*}
    \| S(t-s) \|_{\mathcal{L}(H^{K}_{x,y}, L^2_{x,y})} \leq \mathcal{C}
    \quad \Leftrightarrow \quad \max_{0 \leq t \leq T}\| S(t) \|_{\mathcal{L}(H^{K}_{x,y}, L^2_{x,y})}  \leq  \mathcal{C}<+\infty.
\end{equation*}
This uniform bound is fundamental: at page 6 of \cite{MR2952715}, with $\ee = 1/\sqrt{k}$, they leverage the unstable approximate solutions $u_\ee^n$ of \cite{MR2601044} with remainders $r_\ee^n$  to estimate a correction $v(t) = u - u_\ee^n$ via the Duhamel relation
\begin{equation*}
    \| v(t) \|_{L^2_{x,y}} = 
    \bigg\| \int_0^t S(t-s) r_\ee^n(s) ds \bigg\|_{L^2_{x,y}} 
    \leq \mathcal{C} \int_0^t \| e^y r_\ee^n (s) \|_{H^{K}_{x,y}} ds.
\end{equation*}
Their analysis leads eventually to an estimate that fails to hold for small $\ee > 0$ and $t>0$. In the non-autonomous case, however, the corresponding inequality would be
\begin{equation*}
    \| v(t) \|_{L^2_{x,y}} \leq 
   \int_0^t \| S(s,t) r_\ee^n(s) \|_{L^2_{x,y}} ds \leq 
   \mathcal{C} \int_0^t \| e^y r_\ee^n (s) \|_{H^{K}_{x,y}} ds,
\end{equation*}
which holds only under the further assumption that  $S$ is uniformly continuous: $\mathcal C>0$ and $K\in \mathbb N$ do not change if the initial time for $\mathcal T$ is replaced by $t =s\geq 0$. Hence, any eventual contradiction based uniquely on this approach would not distinguish whether is from non-existence of solutions or from the fact that $S$ is simply not uniformly continuous. Therefore, the second case would not improve the result of \cite{MR2601044}. 
\end{remark} 

\section{Ill-posedness induced by the initial time}\label{sec:toy-model}

\noindent 
We present a toy model for a two parameter semigroup $S = S(\tau, t)$ satisfying the relations introduced in \eqref{eq:intro-S-semigroup-relations}. For simplicity, we restrict our attention to a linear PDE in the variables $(t,x) \in \mathbb R_+ \times \mathbb T$. 

\smallskip 
\noindent 
We introduce an operator $\mathcal{A}\in L^\infty(\mathbb R_+, \mathcal{L}(H^m(\mathbb T), H^{m-\frac{1}{2}}(\mathbb T) ))$, for any $m \in \mathbb{R}$. Given a function $\phi \in H^m(\mathbb{T})$
\begin{equation*}
	\phi(x) = \sum_{k\in \mathbb Z} \phi_{k} e^{ikx}, \qquad 
    \sum_{k \in \mathbb Z} (1+k^2)^m|\phi_k|^2<+\infty,
\end{equation*}
$[\mathcal{A}(t)](\phi) \in  H^{m-\frac{1}{2}}(\mathbb T)$ acts on $\phi$ via the following Fourier multipliers at any time $t\geq 0$:
\begin{equation*}
	[\mathcal{A}(t) \phi](x)
	:=
	\sum_{k\in \mathbb Z\setminus\{0\}}
	\sqrt{|k|}					
	A_k(t) 
	\phi_k	
	e^{i kx}
	\quad 
	\text{with}
	\quad 
	A_k(t) 
	:= -		
		\mathbf{1}_{\left[0, \frac{1}{2\sqrt[4]{|k|}}\right)}(t)  	
		+
		\mathbf{1}_{\left[\frac{1}{2\sqrt[4]{|k|}}, \frac{1}{\sqrt[4]{|k|}}\right)}(t).
\end{equation*}
Each $A_k$ belongs to $L^\infty(\mathbb R)$ and $\| \sqrt{|k|} A_k \|_{L^\infty(\mathbb R_+)}\leq \sqrt{|k|}$, so that $[\mathcal{A(\cdot)}](\phi) \in L^\infty(\mathbb R_+,H^{m-\frac{1}{2}}(\mathbb T)) $.  

\smallskip 
\noindent 
Next, given an arbitrary initial time $\tau \geq 0$ and initial datum $\phi_{in} \in H^m(\mathbb T)$, we address the Cauchy problem:
\begin{equation}\label{eq:evol-eq}
	\begin{cases}
		\phi_\tau'(t)  = 
		\mathcal{A}(t) \phi_\tau(t)\qquad \qquad t \in [\tau , \infty) ,
		\\
		\phi_\tau(\tau ) = \phi_{in}.
	\end{cases}
\end{equation}
As main property: when $\tau = 0$, each Fourier multiplier $\sqrt{|k|}A_k(t)$ produces a short-time dissipation $e^{-\sqrt[4]{|k|}t}$ on the mode $\phi_{0,k}(t)$ for $t \in [0, \,\tfrac{1}{2}|k|^{-1/4})$, followed by an opposite inflation phase in $t \in [\tfrac{1}{2}|k|^{-1/4}, \, |k|^{-1/4})$ that restores the initial amplitude; thereafter, the solution remains constant.

\smallskip 
\noindent 
The following result states  that \eqref{eq:evol-eq} is weakly well-posed in any Sobolev space for any initial time $\tau \geq 0$.
\begin{lemma}\label{thm:well-posedness-Sobolev}
	For any $m \in \mathbb R$, any $ \tau  \geq 0$ and any $\phi_{in} \in H^m(\mathbb T)$, 
	\eqref{eq:evol-eq} admits a unique solution
	\begin{equation}\label{fct-space-phi}
		\phi_\tau  \in \mathcal{C}( [\tau, +\infty), H^m(\mathbb T))\cap 
		W^{1, \infty}( [\tau, +\infty), H^{m-\frac{1}{2}}(\mathbb T)).
	\end{equation}
	Additionally, given the constants  $
	C_0 := 1 $ and $C_\tau := e^{1/\tau}$, for any $\tau >0$, the following inequality holds true:
\begin{equation}\label{eq:cont-wrt-initial-data}
		\max_{\tau \leq t \leq +\infty} \| \phi_\tau(t) \|_{H^m(\mathbb T)} \leq C_\tau \| \phi_{\rm in} \|_{H^m(\mathbb T)}.
\end{equation}
\end{lemma}
\noindent 
Thus, by setting $S(\tau, t)(\phi_{\rm in}) := \phi_\tau(t)$, we define a (non-uniform) strongly continuous evolution family of linear bounded operator on $H^m(\mathbb T)$. Moreover, since $C_0 =1$, then $S(0, \cdot)$ satisfies
\begin{equation*}
    \sup_{0 \leq t < + \infty} 
    \big\| 
        \,S(0,t) \,
    \big\|_{\mathcal{L}(H^{m}(\mathbb T), H^{m}\mathbb T))} \leq 1,\qquad 
    \text{for any }m \in \mathbb R,
\end{equation*}
which is the first inequality of \eqref{eq:intro-S-semigroup-relations}. The next lemma establishes the second relation.
\begin{lemma}\label{thm:impostor-identity}
For any $\delta >0$, any $0 \leq \sigma < 1$ and any $m_1,\, m_2 \in \mathbb R$ the following identity holds true:
\begin{equation}\label{eq:impostor-identity}
	\sup_{0 \leq \tau \leq t \leq \delta} 
	\left\|
		e^{-\sigma (t-\tau) \sqrt{|\partial_x|} }
		S(\tau,t)
	\right\|_{\mathcal{L}( H^{m_1}(\mathbb T), H^{m_2}(\mathbb T)\big)}
	= + \infty.
\end{equation}
\end{lemma}
\begin{proof}[Proof of \Cref{thm:well-posedness-Sobolev}]
We address \eqref{eq:evol-eq} at each frequency $k\in \mathbb Z\setminus\{0\}$ (the $0$-mode remains constant).
We denote by $\phi_{\tau, k} \in \mathcal{C}([\tau, +\infty)) \cap  W^{1,\infty}(\tau, +\infty)$ the unique weak solution of the linear ODE:
\begin{equation}\label{eq:ODE-projected}
		\phi_{\tau,k}'(t) = \sqrt{|k|} A_k(t)\phi_{\tau,k}(t)\quad t \in [\tau, \infty),
        \qquad 
        \qquad 
		\phi_{\tau,k}(\tau) = \phi_{in, k},
\end{equation}
where $\phi_{in, k} \in \mathbb C$ is the $k$-mode of $\phi_{in }$. $\phi_{\tau,k}$ is explicitly determined by
\begin{equation}\label{eq:def-Tk}
	\phi_{\tau,k}(t) = S_k(\tau, t) \phi_{in, k},
	\quad
	\text{with}
	\quad 
	S_k(\tau, t) := 
	\exp 
	\left(
		\sqrt{|k|} \int_\tau^t 
		A_k(\omega) 
		d\omega 
	\right)\in \mathbb R.
\end{equation}
We next consider the time-dependent Fourier series $\phi_\tau(t,x) := \sum_{k \in \mathbb{Z}} \phi_{\tau,k}(t) e^{i k x}$ and show that it converges in $\mathcal{C}([\tau, \infty), H^m(\mathbb{T}))$. We first address $\tau = 0$ and remark that, from the definition of $S_k$ and $A_k$,
\begin{equation*}
\begin{alignedat}{4}
	&t \in \left[0, \frac{1}{2\sqrt[4]{k}}\right) \quad 
	&&\Rightarrow \quad 
	0 \leq S_k(0,t) = 
	\exp 
	\left(
		-		
		t \sqrt{|k|} 
	\right)
	\leq 1,	
	\\
	&t \in \left[\frac{1}{2\sqrt[4]{k}},\frac{1}{\sqrt[4]{k}}\right) \quad 
	&&\Rightarrow \quad 
	0 \leq S_k(0,t) = 
	\exp 
	\left(
		\sqrt{|k|}
		\left( -  \frac{1}{2\sqrt[4]{k}} + t -  \frac{1}{2\sqrt[4]{k}} \right) 
	\right)
	\leq 1,
	\\
	&t \in \left[\frac{1}{\sqrt[4]{k}}, +\infty\right) \quad 
	&&\Rightarrow \quad 
	0 \leq S_k(0,t) = 1.	
\end{alignedat}
\end{equation*}
Thus $0\leq S_k(0, t)\leq 1$ for any $k \in \mathbb Z\setminus\{0 \}$ and  any $t\geq 0$. This implies in particular
\begin{equation*}
	\| \phi_0(t) \|_{H^m(\mathbb T)}^2 = 	
	\sum_{k \in \mathbb Z}(1+ k^2)^\frac{m}{2} \left| S_k(0,t)\phi_{k,in} \right|^2 \leq \| \phi_{in} \|_{H^m(\mathbb T)}^2 < \infty.
\end{equation*}
Since $ |S_k(0,t)\phi_{k,in} | \leq |\phi_{k,in}|$, the function $\phi_0$ belongs to $\mathcal{C}([0, \infty), H^m(\mathbb T))$ thanks to the dominated convergence theorem for series and the continuity of $S_k(0,\cdot)$. We have thus established the constant $C_0 = 1$ in \eqref{eq:cont-wrt-initial-data}. 

\smallskip 
\noindent 
Next, we handle $\tau>0$ 
and split the Fourier series $\phi_\tau $ into two components:
\begin{equation*}
	\phi_\tau (t,x )
    =
    \phi^1_\tau (t,x ) 
    +
    \phi^2_\tau (t,x )
	:= 
	\sum_{|k| \leq \frac{1}{\tau^4}}  S_k(\tau ,t)\phi_{k,in}e^{i k x} 
	+
	\sum_{|k| > \frac{1}{\tau^4}}S_k(\tau ,t)\phi_{k,in}e^{i k x}.
\end{equation*}
The second series $ \phi^2_\tau$ handles frequencies satisfying $1/\sqrt[4]{|k|}<\tau $. Hence, for any  $t \geq \tau $, we obtain the identity $S_k(\tau, t) = \exp( \sqrt{|k|}\int_\tau^t A_k(\omega) d\omega)=  \exp( \sqrt{|k|}\int_\tau^t 0 \,d \omega )= 1$. Therefore $ \phi^2_\tau(t,x)$ is constant in time and satisfies $\| \phi^2_\tau(t,\cdot) \|_{H^m(\mathbb T)} \leq \| \phi_{in} \|_{H^m(\mathbb T)}$. 

\noindent 
The first sum $\phi^1_\tau$ is finite and therefore in $\mathcal{C}([\tau, \infty), H^m(\mathbb{T}))$. Moreover, by definition, $A_k(t) \leq \mathbf{1}_{[0,1/\sqrt[4]{|k|})}(t)$ for any $t \geq 0$. Thus, for any $|k| \leq 1/\tau^{4}$, we gather that
\begin{equation*}
	0 \leq S_k(\tau,t) 
    = \exp\bigg( \sqrt{|k|} \int_\tau^tA_k(\omega) d \omega   \bigg) 
    \leq \exp\left( \sqrt{|k|}  \frac{1}{\sqrt[4]{|k|}} \right) 
    = e^{ \sqrt[4]{|k|}  }
    \leq 
    e^{1/\tau}.
\end{equation*}
We have established therefore that $\phi_\tau \in \mathcal{C}([\tau, \infty), H^m(\mathbb T))$ satisfies  \eqref{eq:cont-wrt-initial-data}. 

\smallskip 
\noindent 
We conclude observing that
$
\partial_t \phi_\tau = [\mathcal{A}(\cdot)](\phi_\tau(\cdot)) \in L^\infty(\tau, \infty; H^{m-1/2}(\mathbb{T}))
$
and that the uniqueness follows from the uniqueness of the ODE \eqref{eq:ODE-projected} for each Fourier mode.
\end{proof}

\begin{proof}[Proof of \Cref{{thm:impostor-identity}}]
\noindent 
Let $\delta>0$ and $\sigma\in[0,1)$.  
We consider a general $k\in\mathbb N$ satisfying 
 $\frac{1}{\sqrt[4]{k}}\leq \delta$. Then, by setting $0 < \tau_k =\tfrac{1}{2\sqrt[4]{k}} < t_k =  \tfrac{1}{\sqrt[4]{k}} \leq \delta $, we obtain 
\begin{align*}
	\sup_{0 \leq \tau  \leq t \leq \delta} 
	\left\|
		e^{-\sigma (t-s) \sqrt{|\partial_x|} }
		S(\tau ,t)
	\right\|_{\mathcal{L}( H^{m_1}(\mathbb T), H^{m_2}(\mathbb T)\big)}
	&\geq 
    e^{-\sigma (t_k-\tau_k) \sqrt{k} } S(\tau_k,t_k)\left( 1+ k^2 \right)^\frac{m_2-m_1}{2} 
    \\
	&\geq 
     \exp \left( -\sigma \frac{\sqrt[4]{k}}{2}\right) 
	S_k\left( \frac{1}{2\sqrt[4]{k}}, \frac{1}{\sqrt[4]{k}} \right) \left( 1+ k^2 \right)^\frac{m_2-m_1}{2}. 
\end{align*}
Next, since $A_k \equiv 1$ on the interval $[\tau_k, t_k] = \big[\frac{1}{2\sqrt[4]{|k|}},\frac{1}{\sqrt[4]{|k|}}\big]$, we obtain
\begin{align*}
	S_k\left( \frac{1}{2\sqrt[4]{k}}, \frac{1}{\sqrt[4]{k}} \right) 
	=
	\exp 
	\left(
		\sqrt{k} \int_{\frac{1}{2\sqrt[4]{k}}}^\frac{1}{\sqrt[4]{k}} 
		A_k(\omega) 
		d\omega 
	\right)
    = 
	\exp 
	\left(
		\frac{1}{2}\sqrt[4]{k}
	\right).
\end{align*}
Summarising, for any $k\in\mathbb N$ with
$\frac{1}{\sqrt[4]{|k|}}\le \delta$, the following lower bound holds true:
\begin{equation*}
	\sup_{0 \leq \tau \leq t \leq \delta} 
	\left\|
		e^{-\sigma (t-\tau) \sqrt{|\partial_x|} }
		S(\tau,t)
	\right\|_{\mathcal{L}( H^{m_1}(\mathbb T), H^{m_2}(\mathbb T)\big)}
	\geq 
	\exp 
	\left(
		\frac{1-\sigma}{2}\sqrt[4]{k}
	\right)
	\left( 1+ k^2 \right)^\frac{m_2-m_1}{2}
\end{equation*}
Since $\sigma\in [0, 1)$, by sending $k\to\infty$, we obtain the relation \eqref{eq:impostor-identity}, which concludes the proof of \Cref{thm:impostor-identity}.
\end{proof}
\section{The matched asymptotic expansion}\label{sec:matched-asymptotics}

We begin with the proof of \Cref{thm:inflating-solutions} and \Cref{thm:non-existence-of-solutions} and devote this section to the construction of the approximate solution $u_k^{\rm fr} = \partial_y \phi_{{\rm inn}, k}+  \partial_y \phi_{{\rm out}, k}$ as described in Part-(a) of \Cref{sec:novelties-some-remarks}. The first building block for the inner layer $\phi_{{\rm inn}, k}$, is the entire function $\Upsilon$ of \Cref{def:uink-into}. For simplicity, we recall its definition:
\begin{equation}\label{def:Upsilon-fct} 
    \Upsilon(\zeta):= 
    \frac{1}{\sqrt{2\pi}}\,\zeta\, e^{-\frac{\zeta^2}{2}} +
    \frac{1}{2}  
    \big(1+\zeta^2\big) 
    \left(1 + \erf\left( \frac{\zeta}{\sqrt{2}} \right) \right),
\end{equation}
where $\erf : \mathbb C \to \mathbb C$ satisfies $\erf(0) = 0$ and $\erf'(\zeta) = \frac{2}{\sqrt{\pi}} e^{-\zeta^2}$. In the next lemma, we establish via $\Upsilon$ a solution of the spectral ODE for the linear Prandtl equations around a quadratic shear flow.
\begin{lemma}\label{lemma:quadratic-unstable-quasi-eigenmode}
    Let $a>0$, $\alpha \in \mathbb R$, $\beta<0$ and $k \in \mathbb N$. Introduce  $\sigma_k\in \mathbb C$ and the change of variable $z : \mathbb C \to \mathbb C$    
    \begin{equation*}
        \sigma_k:= \sqrt{|\beta | k} \, e^{-\frac{\im \pi}{4}} - \im  k \,\alpha 
        \in \mathbb C
        \qquad\qquad 
        z(y) := \sqrt[4]{|\beta|}(y-a)e^{-\frac{\im \pi}{8}} \in \mathbb C.
    \end{equation*}
    Then, the entire function $\phi_k(y) =\Upsilon( \sqrt[4]{k}\, z(y) ) \sqrt{\frac{|\beta|}{k}}
        e^{i\frac{5\pi}{4}}$  satisfies the following ODE in a classical sense:
    \begin{equation}\label{eq:lemma-Upsilon-identity}
         \Big( \sigma_k  + i k \,\Big( \alpha + \beta (y-a)^2 \Big) \Big) \phi_k'(y) - 2 i k\,  \beta (y-a) \phi_k(y) - \phi_k'''(y) = 0,\qquad y \in \mathbb C.
    \end{equation}
\end{lemma}
\begin{proof} We compute the first three derivatives of $\Upsilon$:
\begin{equation*}
    \Upsilon'(\zeta) = \sqrt{\frac{2}{\pi}}e^{-\frac{\zeta^2}{2}} + 
    \zeta\bigg(1 +  \erf\left( \frac{\zeta}{\sqrt{2}}\right) \bigg),
    \qquad 
    \Upsilon''(\zeta) = 1+  \erf\Big( \frac{\zeta}{\sqrt{2}}\Big),\qquad 
    \Upsilon'''(\zeta) = \sqrt{\frac{2}{\pi}} e^{-\frac{\zeta^2}{2}}.
\end{equation*}
Hence, we remark that
\begin{align*}
    -\big( 1 + \zeta^2 \big) \Upsilon'(\zeta) +  2 \zeta \Upsilon(\zeta) + \Upsilon'''(\zeta) 
    =
    &-\big( 1 + \zeta^2 \big)\sqrt{\frac{2}{\pi}}  e^{-\frac{\zeta^2}{2}}
    -
    \big( 1 + \zeta^2 \big)\zeta 
    \bigg(1 +  \erf\left( \frac{\zeta}{\sqrt{2}}\right) \bigg)
    +
    \\
    &+\sqrt{\frac{2}{\pi} }\zeta^2 \,e^{-\frac{\zeta^2}{2}} + \zeta (1+\zeta^2) \bigg(1 +  \erf\left( \frac{\zeta}{\sqrt{2}}\right) \bigg) + \sqrt{\frac{2}{\pi}} e^{-\frac{\zeta^2}{2}} 
    = 0.
\end{align*}
By setting the substitution $\zeta = \zeta(y) =  \sqrt[4]{k}\,z(y) =\sqrt[4]{-\beta k}  (y-a)e^{-\frac{i\pi}{8}}$, we recast each $\zeta$-derivative into:
\begin{equation*}
    \frac{d}{d\zeta} = \frac{e^\frac{\rm i\pi}{8}}{\sqrt[4]{-\beta k}} \frac{d}{dy},
    \qquad
    \frac{d^2}{d\zeta^2} = \frac{e^\frac{\rm i\pi}{4}}{\sqrt{-\beta k}} \frac{d^2}{dy^2},
    \qquad
    \frac{d^3}{d\zeta^3} = \frac{e^{\im \frac{\rm 3\pi}{8}}}{(-\beta k)^\frac{3}{4}} \frac{d^3}{dy^3}.
\end{equation*}
Thus, we obtain
\begin{equation*}
    \bigg( 
    - 
    \Big( 1 + \sqrt{-\beta k}(y-a)^2e^{-\frac{\im \pi}{4}} \Big)
    \frac{e^\frac{\rm \im\pi}{8}}{\sqrt[4]{-\beta k}} 
    \frac{d}{dy}
    +
    2\sqrt[4]{-\beta k}(y-a)e^{-\frac{\im\pi}{8}}
    +
    \frac{e^{\im\frac{\rm 3\pi}{8}}}{(-\beta k)^\frac{3}{4}} 
    \frac{d^3}{dy^3}
    \bigg)\Big[ \Upsilon( \sqrt[4]{k}\, z(y) )\Big] = 0.
\end{equation*}
Multiplying this identity by $-(-\beta k)^{3/4} e^{-\frac{3 \im \pi}{8}}$:
\begin{equation*}
    \bigg( 
    \Big( \sqrt{-\beta k} e^{-\frac{\im \pi}{4}}  + (-\beta k) (y-a)^2e^{-\frac{\im \pi}{2}} \Big)
    \frac{d}{dy} 
    -
    (-\beta k)(y-a)e^{-\frac{\im \pi}{2}}
    -
    \frac{d^3}{dy^3}
    \bigg) \Big[ \Upsilon( \sqrt[4]{k}\, z(y) )\Big] = 0.
\end{equation*}
Recalling that $\sigma_k =  \sqrt{-\beta k}  e^{-\frac{\im \pi}{4}} - \im  k \alpha$, we finally obtain
\begin{equation}\label{eq:lemma-Upsilon-identity-proof-sigma}
    \bigg( \Big( \sigma_k + \im k\,\Big( \alpha  + \beta (y-a)^2 \Big) \Big)
    \frac{d}{dy} 
    -
    2 \im k\, \beta (y-a)
   -
    \frac{d^3}{dy^3}\bigg) \Big[ \Upsilon( \sqrt[4]{k}\, z(y) )\Big] = 0.
\end{equation}
Multiplying this identity by $ \sqrt{\frac{|\beta|}{k}} e^{\frac{5\pi}{4}i}$ concludes the proof of the lemma.
\end{proof}
\noindent 
Building upon this $\Upsilon$-profile, we next define the inner and outer layers $\phi_{{\rm inn}, k}$ and $\phi_{{\rm out}, k}$. We invoke the quadratic assumption (H3) of \Cref{Hypotheses-US-us} for $u_{s|t = 0}= U_s$: 
\begin{equation*}
    U_s(y) = U_s(a_0) + \frac{U_s''(a_0)}{2} (y-a_0)^2,\qquad\text{with}\quad U_s''(a_0) <0,\qquad \text{for any }y \in \left[ a_0 - d, a_0 + d\right].
\end{equation*}
Similarly as in \cite{MR2601044}, recalling that $u_s$ satisfies the heat equation \eqref{eq:heat-equation}, we track the evolution of the critical point  $y = a_0$ by solving the nonlinear ODE:
\begin{equation}\label{def:a(t)}
    \partial_y^3 u_s(t,a(t)) \,+\,a'(t) \,\partial_y^2 u_s(t,a(t)) = 0,\qquad 
    a(0) = a_0,\qquad t \in [0,T_{\rm max}].
\end{equation}
The Cauchy--Lipschitz theorem guaranties the existence of a local-in-time smooth solution $a \in \mathcal{C}^\infty([0,T_{\text{max}}])$.
Without loss of generality, we can restrict $T_{\text{max}} > 0$  to be finite and sufficiently small to ensure:
\begin{equation}\label{eq:assumptions-for-Tmax}
    a(t) \in 
    \left[ a_0 - \frac{d}{8},\, a_0 + \frac{d}{8} \right]
    \quad \text{and}
    \quad 
    \partial_y^2 u_s(t,a(t)) \in 
    \Big[2U_s''(a)
    ,\, \frac{U_s''(a_0)}{2}\Big],
    \quad 
    \text{for any}\quad 
    t \in [0, T_{\rm max}]    
    .
\end{equation}
Moreover, to simplify some estimates, we assume from the continuity in time of $\| \partial_y u_s(t, \cdot) \|_{L^2(a_0+\frac{d}{2}, \infty)}$ that
\begin{equation}\label{eq:assumption-for-Tmax-and-partialyus}
    \| \partial_y u_s(t, \cdot )\|_{L^2(a_0+\frac{d}{2}, \infty)}
    \geq 
    \frac{1}{2}
    \| U_s'\|_{L^2(a_0+\frac{d}{2}, \infty)}>0,\qquad \text{for any }t \in [0, T_{\rm max}].
\end{equation}
We recall from \Cref{def:uink-into} that $\chi \in \mathcal{C}^\infty_c(\mathbb{R}_+)$ is a cutoff function satisfying the relations:
\begin{equation*}
       \chi(y) = 1 \quad \text{for any }y \in 
       \left[a_0-\frac{d}{4},\, a_0+\frac{d}{4}\right],
       \qquad 
       \chi(y) = 0 \quad 
       \text{for any }y \in  \mathbb R_+\setminus \left]a_0-\frac{d}{2},\, a_0+\frac{d}{2} \right[.
\end{equation*}
\begin{definition}\label{def:inner-outer-profiles}
    For each frequency $k \in \mathbb{N}$, define  $z \in \mathcal{C}^\infty([0, T_{\rm max}] \times \mathbb R_+ , \mathbb C)$ and $\sigma_k \in \mathcal{C}^\infty([0, T_{\rm max}], \mathbb C)$:
    \begin{equation}\label{eq:z-sigmak-in-def-of-phi-inn-out}
        z(t,y) :=(y-a(t)) e^{-\frac{\im \pi}{8}} 
        \sqrt[4]{\frac{|\partial_y^2u_s(t,a(t))|}{2}} \in \mathbb C,
        \quad
        \quad
        \sigma_k(t):=
        \sqrt{k} \,e^{-\frac{\im \pi}{4}}  \sqrt{\frac{|\partial_y^2u_s(t,a(t))|}{2}} - 
        \im k \, u_s( t,a(t))\in \mathbb C,
    \end{equation}
    We define the inner and outer layers at any $(t,y) \in [0, T_{\rm max}] \times \mathbb R_+$ as
    \begin{equation}\label{eq:phi-inn-out}
    \begin{aligned}
        (i)\quad &\phi_{{\rm inn}, k}(t,y) 
        := 
        \chi(y) 
        \Upsilon( \sqrt[4]{k} \, z(t,y) ) 
        \,
            e^{\im \frac{5\pi}{4}}
            \sqrt{\frac{|\partial_y^2u_s(t,a(t))|}{2k}} 
        \exp \bigg( \int_0^t \sigma_k(\tau) d \tau\bigg)
        ,\\
        (ii)\quad &\phi_{{\rm out}, k}(t,y) 
        :=
        \big(1 - \chi(y)\big)
        H\big(y-a(t)\big)
        \bigg( 
            u_s(t,y)
            +
            \frac{\sigma_k(t)}{\im k}
        \bigg)
        \exp \bigg( \int_0^t \sigma_k(\tau) d \tau\bigg)
        ,
    \end{aligned}    
    \end{equation}
    where $H$ denotes the Heaviside function.
\end{definition}
\begin{remark}
    Under (H1)-(H2) of \Cref{Hypotheses-US-us}, we recall that $u_s \in \mathcal{C}([0, \infty[, H^{m+1}_\lambda(\mathbb R_+))$ with $m\geq 3$. Hence
    \begin{equation}\label{def:phi-inn-out-regularities}
    \begin{aligned}
        &\phi_{{\rm inn}, k} \in 
        \mathcal{C}^\infty_c\big(\,[0, T_{\rm max}] \times ]0, \infty[\,\,\big),\quad 
        \phi_{\rm out, k}\in 
        \mathcal{C}^\infty(\,]0, T_{\rm max}] \times [0, \infty[\,)\cap \mathcal{C}([0, T_{\rm max}], H_\lambda^{m+1}(\mathbb R_+) \oplus \mathbb C),
    \end{aligned}
    \end{equation}
    Furthermore, the associated  $u^{\rm fr}_k := \partial_y 
    ( \phi_{{\rm inn}, k} +   \phi_{{\rm out}, k}  ) \in  \mathcal{C}^\infty(\,]0, T_{\rm max}] \times [0, \infty[\,)$ belongs to
\begin{equation}\label{eq:ukfr-vkfr-def-and-spaces}
	u^{\rm fr}_k 
    \in 
    \mathcal{C}([0, T_{\rm max}], H_\lambda^m(\mathbb R_+))
    \cap L^2(0, T_{\rm max}; H^{m+1}_\lambda(\mathbb R_+))
        \cap W^{1,2}(0, T_{\rm max}; H^{m-1}_\lambda(\mathbb R_+)).
\end{equation}
Notably, $(1-\chi(y))$ is zero around $y = a(t)$, eliminating any discontinuities introduced by the Heaviside function.  
The boundary conditions $u_k^{\rm fr}(t,0)  = \phi_{{\rm inn}, k}(t, 0)= \phi_{{\rm out}, k}(t, 0)= 0$, for any $t \in [0, T_{\rm max}]$, are also satisfied.
\end{remark}

\smallskip 
\noindent 
The next proposition determines the forcing term generated by $u_k^{\text{fr}}$ when inserted into System \eqref{eq:lin-Prandtl-projected-k}.
For the sake of abbreviation, we introduce the short notation $\alpha,\, \beta,\,\gamma \in \mathcal{C}^\infty([0, T_{\rm max}])$ with 
\begin{equation}\label{eq:notation-alpha-beta-gamma}
    \alpha(t):= u_s(t,a(t)),\qquad 
    \beta(t) := \frac{1}{2}\partial_y^2 u_s(t,a(t)),\qquad 
    \gamma(t) :=
    \sqrt[4]{-\beta(t)} = \sqrt[4]{\frac{|\partial_y^2 u_s(t,a(t))|}{2}}>0,
\end{equation}
which implies $z(t,y) = \gamma(t) (y-a(t))e^{-\frac{i\pi}{8}}$, thanks to \eqref{eq:z-sigmak-in-def-of-phi-inn-out}. 
\begin{prop}\label{prop:remainder-exact-form}
    For any frequency $k \in \mathbb N$, 
    $u_k^{\rm fr}$ 
    satisfies the following system in a classical sense:
    \begin{align*}
    \left\{
    \begin{alignedat}{4}
        &\Big(
            \partial_t + \im k\, u_s(t,y)  - 
            \partial_y^2 
        \Big) u_k^{\rm fr}(t,y)  
        - 
        \im k\, \partial_y u_s(t,y) \int_0^y u_k^{\rm fr}(t, \omega)d \omega = 
        f_k(t,y)
        e^{ \int_0^t \sigma_k(\tau) d \tau} ,\qquad 
        &&
        (t,y) \in 
        \,]0, T_{\rm max}[ \times \mathbb ]0, \infty[,
        \\
        &\,u_k^{\rm fr}(t,0) = \lim_{y\to +\infty} u_k^{\rm fr}(t,y) = 0
        && \hspace{0.58cm}
        t \in [0, T_{\rm max}],
        \\
        &\,u_k^{\rm fr}(0,y) = \partial_y \phi_{\rm inn, k}(0,y)
        +
        \partial_y \phi_{\rm out, k}(0,y)
        && \hspace{0.58cm}
        y \in  ]0, \infty[,
    \end{alignedat}
    \right.
    \end{align*}
where $f_k \in  \mathcal{C}^\infty_c([0, T_{\rm max}] \times \mathbb R_+, \mathbb C) $ is the sum 
\begin{equation*}
    f_k(t,y) := \sum_{j = 1}^7\mathcal{R}^1_{j,k}(t,y)  +
            \sum_{j = 1}^3\mathcal{R}^2_{j,k}(t,y) ,\qquad (t,y) \in [0, T_{\rm max}] \times [0, \infty[
\end{equation*} 
of two family of remainders: 
\begin{itemize}[leftmargin=0.7cm]
    \item[(a)] The remainders 
\begin{equation}\label{eq:first-family-remainders}
\begin{aligned}
    \mathcal{R}^1_{1,k}(t,y)
    &=
    \bigg[\,  
        u_s(t,y)  - 
        \alpha(t) -
        \beta(t) (y-a(t))^2
    \,\bigg] 
    \chi (y)
    \Upsilon'(\sqrt[4]{k}\,z(t,y) ) 
    \gamma(t)^3
   e^{\im \frac{13\pi}{8}} k^\frac{3}{4}
    ,
    \\
    \mathcal{R}^1_{2,k}(t,y)
    &=
    \bigg[  
        u_s(t,y)  - 
        \alpha(t) -
        \beta(t) (y-a(t))^2
    \bigg]
    \Big(
        \chi'''(y) -
        \Big(
            \sigma_k(t)   +  \im  k \, u_s(t,y)   
        \Big)
        \chi'(y)
    \Big)
    H(y-a(t)),
    \\
    \mathcal{R}^1_{3,k}(t,y)
    &=
    \bigg[ 
        \partial_y u_s(t,y)
        - 
        2\beta(t)(y-a(t))
    \bigg]
    \chi(y) 
    \gamma(t)^2 
    \Upsilon( \sqrt[4]{k} \, z(t,y))
     e^{\im \frac{3\pi}{4}}
     \sqrt{k},
    \\
    \mathcal{R}^1_{4,k}(t,y)
    &=
    \Big[
       \partial_y u_s(t,y)
       - 
       2\beta(t)(y-a(t))
    \Big]
    3
    \chi''(y)
    H(y-a(t)),
    \\
    \mathcal{R}^1_{5,k}(t,y)
    &=
    \bigg[\partial_y^2 u_s(t,y) -2\beta(t) \bigg]
    2\chi'(y)
    H(y-a(t)),
    \\
    \mathcal{R}^1_{6,k}(t,y)
    &=\gamma'(t) 
    \bigg(
        \chi'(y) \frac{2\gamma(t)e^{\im \frac{5\pi}{4}}}{\sqrt{k}}
        \Big( 
            \Upsilon(\sqrt[4]{k} \,z(t,y))
            -
           H(y-a(t))
        \Big)
        +
        \chi(y) \frac{3\gamma(t)^2e^{\im \frac{9\pi}{8}}}{\sqrt[4]{k}} 
       \Upsilon'(\sqrt[4]{k} \,z(t,y)) 
    \bigg),
    \\
    \mathcal{R}^1_{7,k}(t,y)
    &=
    \Big[ \gamma'(t)(y-a(t)) - \gamma(t) a'(t) \Big]
    \bigg(
        \chi'(y) \frac{e^{\im \frac{9\pi}{8}}\gamma(t)^2}{\sqrt[4]{k}} 
        \Upsilon'(\sqrt[4]{k} \,z(t,y))
        -
        \chi(y) \gamma(t)^3  
       \Upsilon''(\sqrt[4]{k} \,z(t,y))
    \bigg).
\end{aligned}
\end{equation}
\item[(b)] The remainders 
\begin{equation}\label{eq:second-family-remainders}
\begin{aligned}
    \mathcal{R}^2_{1,k}(t,y) 
    &=
    \bigg[
    \,
    \Upsilon(\sqrt[4]{k}\,z(t,y)) 
    -
    \big(  1 + \sqrt{k}\, z(t,y)^2\, \big)
    H(y-a(t))\,
    \bigg]
    \Big(
        \big(   \sigma_k(t)  +  \im k\, u_s(t,y) \big)\chi'(y) 
    - \chi'''(y) 
    \Big)
    \frac{e^{\im \frac{5\pi}{4}}\gamma(t)^2 }{\sqrt{k}}
    ,
    \\
    \mathcal{R}^2_{2,k}(t,y) 
    &=
    \bigg[\, 
        \Upsilon'(\sqrt[4]{k}\,z(t,y)) - 2 \sqrt[4]{k} \,z(t,y) H(y-a(t))
    \,\bigg]
    \big(
    -
    3 \chi''(y)
    \big)
    \frac{e^{\im \frac{9\pi}{8}}\gamma(t)^3 }{\sqrt[4]{k}} 
    ,
    \\
    \mathcal{R}^2_{3,k} (t,y)
    &=
    \bigg[\, 
        \Upsilon''(\sqrt[4]{k}\,z(t,y))-2 H(y-a(t)) 
    \,\bigg]
    \big( -3 \chi'(y) \big)\beta(t) .
\end{aligned}
\end{equation}
\end{itemize}
\end{prop}
\noindent 
The leading terms are isolated in the first (squared) brackets. We note that each remainder depends on the cutoff function $\chi$ or its derivatives, supported in $[0, T_{\text{max}}] \times [a_0 - \tfrac{d}{2}, a_0 + \tfrac{d}{2}]$. Furthermore, since $u_s$ is smooth in this region ($U_s$ is quadratic), it follows that $f \in \mathcal{C}_c^\infty([0, T_{\text{max}}] \times \,]0, \infty[)$.

\smallskip 
\noindent 
The family $\mathcal{R}_{1,k}^1, \dots, \mathcal{R}_{7,k}^1$ depends mainly on the deviation of $u_s$ from its second-order Taylor expansion centered at $y = a(t)$. As discussed in \Cref{sec:novelties-some-remarks}, these are therefore the $\mathcal{O}(e^{-c_2/t})$-remainders.

\smallskip
\noindent 
The second family $\mathcal{R}_{1,2}^2$, $\mathcal{R}_{1,2}^2$ and $\mathcal{R}_{1,2}^3$ has leading-order terms governed by the interaction between the function $\Upsilon$ and the Heaviside function $H$. Eventually, these yield  estimates of the form
\begin{equation*}
	\bigg| 
        \,
        \erf\bigg( \frac{\sqrt[4]{k}\,z(t,y)}{\sqrt{2}}\bigg) - H(y-a(t)) 
        \,
    \bigg| 
    \lesssim    
    e^{    -\sqrt{k} \,(y-a(t))^2  \big(\frac{|U_s''(a_0)|}{2} \big)^{1/2}}
\end{equation*}
Since $\chi' \equiv 0$ on the interval $[a_0 - \tfrac{d}{4}, a_0 + \tfrac{d}{4}]$ and given our assumption that $a(t) \in [a_0 - \tfrac{d}{8}, a_0 + \tfrac{d}{8}]$, we conclude that these are the $\mathcal{O}(e^{-c_3\sqrt{k}})$-remainders anticipated in \Cref{sec:novelties-some-remarks}. We formalise these estimates in \Cref{prop:estimates-remainders-and-uf}.
\begin{proof}[Proof of \Cref{prop:remainder-exact-form}]
We begin with the following identity form Lemma~\ref{lemma:quadratic-unstable-quasi-eigenmode} (cf.~\eqref{eq:lemma-Upsilon-identity-proof-sigma}):
\begin{equation*}
    \Bigg[ 
    \bigg(
    \sigma_k(t)  +  \im k \Big( \alpha(t) + \beta(t) (y-a(t))^2  \Big) 
    \bigg) 
    \partial_y 
    -2 i k \,\beta(t) (y-a(t))
    -
    \partial_y^3
    \Bigg] 
    \Big( \Upsilon(\sqrt[4]{k}\,z(t,y)) \Big) = 0,
\end{equation*}
for any $(t,y) \in [0, T_{\rm max}] \times \mathbb R_+$. 
Next, we replace the quadratic terms with the shear flow $u_s$ and move any remaining terms to the right-hand side:
\begin{equation}\label{eq:Upsilon-eq-in-us-with-remainders}
\begin{aligned}
    \bigg[ 
    \Big(
    \sigma_k(t)  
    +  
    \im k\, u_s(t,y) 
    \Big) 
    \partial_y 
    -\im k \,\partial_y u_s(t,y) 
    &-
    \partial_y^3
    \bigg] 
    \Big( \Upsilon(\sqrt[4]{k}\,z(t,y)) \Big) = 
    \\
     =
    \im k \, \Big( u_s(t,y) - \alpha(t) -\beta(t)(y&-a(t))^2 \Big) \Upsilon'(\sqrt[4]{k}\,z(t,y))\sqrt[4]{k} \,\gamma(t) e^{-\frac{\im \pi}{8}} 
    +\\
    &-
    \im k \, \Big( \partial_y u_s(t,y) - 2\beta(t)(y-a(t)) \Big) 
    \Upsilon(\sqrt[4]{k}\,z(t,y)),
\end{aligned}
\end{equation}
using $z(t,y) = \gamma(t)(y-a(t))e^{-\frac{\im \pi}{8}}$, which implies $\partial_y z(t,y) = \gamma(t) e^{-\frac{\im \pi}{8}}$ 
Recall from \eqref{eq:phi-inn-out} that the inner layer is
\begin{equation*}
    \phi_{\text{inn}, k}(t,y) = 
    \chi(y)   \Upsilon(\sqrt[4]{k}\,z(t,y)) 
    \frac{e^{\im \frac{5\pi}{4}} \gamma(t)^2}{\sqrt{k}} 
    \exp \bigg( \int_0^t \sigma_k(\tau) d \tau\bigg),
    \qquad 
    (t,y) \in [0, T_{\rm max}] \times \mathbb R_+.
\end{equation*}
Thus, multiplying the identity \eqref{eq:Upsilon-eq-in-us-with-remainders}  by $\chi(y)  \frac{e^{\im \frac{5\pi}{4}}\gamma(t)^2}{\sqrt{k}} e^{\int_0^t \sigma_k(\tau)d\tau}$ and applying the product rule, we obtain:
\begin{equation*}
\begin{aligned}
    \bigg[ 
    \Big(
    \sigma_k(t)  
    +  
    \im k\, u_s(t,y) 
    &\Big) 
    \partial_y 
    -\im k \,\partial_y u_s(t,y) 
    -
    \partial_y^3
    \bigg] \phi_{\text{inn}, k}(t,y)
    = 
    \\
    =
    e^{\int_0^t \sigma_k(\tau)d\tau }
    \Bigg\{
    &
    \underbrace{
    \bigg( 
    \Big(
    \sigma_k(t)  
    +  
    \im  k\, u_s(t,y) 
    \Big) 
    \chi'(y)
    -
    \chi'''(y)
    \bigg) 
    \frac{e^{\im \frac{5\pi}{4}} \gamma(t)^2}{\sqrt{k}}  
    \Upsilon(\sqrt[4]{k}\,z(t,y)) 
    }_{=:\mathcal{I}_1}
    \,+
    \\
    &\hspace{1cm}
    \underbrace{
    - 3 \chi''(y) \frac{e^{\im \frac{9\pi}{8}}\gamma(t)^3}{\sqrt[4]{k}}\Upsilon'(\sqrt[4]{k}\,z(t,y)) 
    - 
    3 
    \chi'(y) 
    \beta(t)
    \Upsilon''(\sqrt[4]{k}\,z(t,y))
    }_{=:\mathcal{I}_2}\,+
    \\
    &\hspace{2cm}
    +
    \underbrace{
    \im  k \, \Big( u_s(t,y) - \alpha(t) -\beta(t)(y-a(t))^2 \Big) 
    \chi(y)
     \frac{e^{\im \frac{9\pi}{8}}\gamma(t)^3}{\sqrt[4]{k}}
    \Upsilon'(\sqrt[4]{k}\,z(t,y))
     }_{= \mathcal{R}_{1,k}^1(t,y)}\,+\\
    &\hspace{3cm}
    \underbrace{-
    \im  k \, \Big( \partial_y u_s(t,y) - 2\beta(t)(y-a(t)) \Big) 
    \chi(y)  \frac{e^{\im \frac{5\pi}{4}}\gamma(t)^2}{\sqrt{k}}
    \Upsilon(\sqrt[4]{k}\,z(t,y))
     }_{= \mathcal{R}_{3,k}^1(t,y)}
    \Bigg\}.
\end{aligned}
\end{equation*}
In the third line with $\mathcal{I}_2$, we used $e^{\im \frac{5\pi}{4}}\gamma(t)^4e^{-\frac{\im \pi}{4}} = - \gamma(t)^4 = \beta(t)$. Additionally, we identified the remainder terms $\mathcal{R}_{1,k}^1(t,y) $ and $\mathcal{R}_{3,k}^1(t,y) $ in the last two lines, thanks to the identity $ \im  k \frac{e^{\im \frac{9\pi}{8}}}{\sqrt[4]{k}} =e^{\im \frac{13\pi}{8}} k^\frac{3}{4}$, as well as $- \im k  \frac{e^{\im \frac{5\pi}{4}}}{\sqrt{k}} = e^{\im \frac{3\pi}{4}}\sqrt{k}$. Using further the notation introduced in \eqref{eq:first-family-remainders} and \eqref{eq:second-family-remainders}, we decompose $\mathcal{I}_1$ and $\mathcal{I}_2$ as:
\begin{align*}
    \mathcal{I}_1 
    &=    
    \mathcal{R}_{1, k}^2(t,y) 
    +
    \bigg( 
    \Big(
    \sigma_k(t)  
    +  
    \im k\, u_s(t,y) 
    \Big) 
    \chi'(y)
    -
    \chi'''(y)
    \bigg)
    \bigg( 
            \frac{  \gamma(t)^2 e^{\im\frac{5\pi}{4}} }{\sqrt{k}}
            +
            \beta(t)(y-a(t))^2
    \bigg)
    H(y-a(t)),
    \\
    \mathcal{I}_2 
    &=    
    \mathcal{R}_{2, k}^2(t,y) 
    +
    \mathcal{R}_{3, k}^2(t,y) 
    - 6\chi'(y) \beta(t)H(y-a(t))
    - 6\chi''(y) \beta(t) (y-a(t)) H(y-a(t))
   ,
\end{align*}
where we have also used the fact that $\gamma(t)^4 = - \beta(t)$ and $z(t,y) = \gamma(t)(y-a(t))e^{-\frac{i\pi}{8}}$, so that
\begin{equation*}
\begin{aligned}
	\big( 1+ \sqrt{k} \, z(t,y)^2 \big) 
	\frac{e^{i \frac{5 \pi}{4}}\gamma(t)^2}{\sqrt{k}}H(y-a(t))
	&= 
	\bigg(
	\frac{e^{i \frac{5 \pi}{4}}\gamma(t)^2}{\sqrt{k}}
	+ 
	\frac{e^{i \frac{5 \pi}{4}}\gamma(t)^2}{\sqrt{k}}
	\sqrt{k} \, e^{-\frac{i\pi}{4}  }\gamma(t)^2 (y-a(t))^2
	\bigg) H(y-a(t))\\
	&=
	\bigg(
	\frac{e^{i \frac{5 \pi}{4}}\gamma(t)^2}{\sqrt{k}}
	+
	\beta(t) (y-a(t))^2
	\bigg) H(y-a(t))
\end{aligned}
\end{equation*}
and similarly
\begin{align*}
	2 \sqrt[4]{k} \,z(t,y)  H(y-a(t)) \frac{e^{i \frac{9\pi}{8}}\gamma(t)^3}{\sqrt[4]{k}} 
	=2 \gamma(t)^4 (y-a(t)) e^{-\frac{i \pi}{8}}e^{i \frac{9\pi}{8}}H(y-a(t)).
\end{align*}
This yields the identity:
\begin{equation*}
\begin{aligned}
    \bigg[ 
    \Big(
    \sigma_k(t)  
    +  
    i k\, u_s(t,y) 
    \Big) 
    \partial_y 
    -i k \,\partial_y u_s(t,y) 
    -
    \partial_y^3
    \bigg] \phi_{\text{inn}, k}(t,y)
    = 
    e^{\int_0^t \sigma_k(\tau)d\tau }
    \Bigg\{
    \mathcal{R}_{1, k}^1(t,y) +
    \mathcal{R}_{3, k}^1(t,y) +
    \mathcal{R}_{1, k}^2(t,y) + 
    \\
    +
    \mathcal{R}_{2, k}^2(t,y) + 
    \mathcal{R}_{3, k}^2(t,y) 
    \Bigg\}
    +
    \Bigg\{
    \bigg( 
    \Big(
    \sigma_k(t)  
    +  
    i k\, u_s(t,y) 
    \Big) 
    \chi'(y)
    -
    \chi'''(y)
    \bigg)
    \bigg( 
            \frac{  \gamma(t)^2 e^{i\frac{5\pi}{4}} }{\sqrt{k}}
            +
            \beta(t)(y-a(t))^2
    \bigg)
   +\\
    - 6\chi''(y) \beta(t) (y-a(t)) 
     - 6\chi'(y) \beta(t)
    \Bigg\}
    H(y-a(t))e^{\int_0^t \sigma_k(\tau)d\tau }.
\end{aligned}
\end{equation*}
Finally, we replace the multiplication by $\sigma_k(t)$ on the left-hand side 
with the time derivative and move any additional term to the right-hand side. Observing that
\begin{align*}
    \partial_t  \partial_y\phi_{\text{inn}, k}(t,y)
    &-
    \sigma_k(t) \partial_y\phi_{\text{inn}, k}(t,y) 
    =
    \partial_t \partial_y 
    \bigg( 
        \chi(y)   \Upsilon(\sqrt[4]{k}\,z(t,y)) 
        \frac{e^{\im\frac{5\pi}{4}} \gamma(t)^2}{\sqrt{k}} 
    \bigg)  e^{\int_0^t \sigma_k(\tau)d\tau }
    \\
    &=
    \partial_t  
    \bigg( 
        \chi'(y)   \Upsilon(\sqrt[4]{k}\,z(t,y)) 
        \frac{e^{\im \frac{5\pi}{4}} \gamma(t)^2}{\sqrt{k}} 
        +
        \chi(y)   
        \Upsilon'(\sqrt[4]{k}\,z(t,y)) 
        \frac{e^{\im \frac{9\pi}{8}} \gamma(t)^3}{\sqrt[4]{k}} 
    \bigg)  
    e^{\int_0^t \sigma_k(\tau)d\tau }
    \\
    &=
       \gamma'(t)
        \bigg(
        \chi'(y) \frac{e^{\im \frac{5\pi}{4}} 2\gamma(t)}{\sqrt{k}}   \Upsilon(\sqrt[4]{k}\,z(t,y)) 
        +
        \chi(y)   
        \Upsilon'(\sqrt[4]{k}\,z(t,y)) 
        \frac{e^{\im \frac{9\pi}{8}} 3\gamma(t)^2}{\sqrt[4]{k}} 
    \bigg)  
    e^{\int_0^t \sigma_k(\tau)d\tau }
    +\\
    &\hspace{1cm}
    +
    \partial_t z(t,y) e^{\frac{\im \pi}{8}} 
    \bigg( 
        \chi'(y) 
        \frac{e^{\im \frac{9\pi}{8}} \gamma(t)^2}{\sqrt[4]{k}} 
        \Upsilon'(\sqrt[4]{k}\,z(t,y))
        -
        \chi(y) 
        \gamma(t)^3
        \Upsilon''(\sqrt[4]{k}\,z(t,y)) 
    \bigg)
    e^{\int_0^t \sigma_k(\tau)d\tau } 
    \\
    &=
    \bigg( 
       \mathcal{R}_{6, k}^1(t,y)
       +
       2\gamma'(t)\frac{e^{\im \frac{5\pi}{4}}\gamma(t) }{\sqrt{k}}H(y-a(t))
       +
       \mathcal{R}_{7, k}^1(t,y) 
    \bigg)  
    e^{\int_0^t \sigma_k(\tau)d\tau },
\end{align*}
we gather finally
\begin{equation}\label{eq:proof-eq-of-phiinn}
\begin{aligned}
    \bigg[ 
    \Big(
    \partial_t   
    +  
    \im k\, u_s  
    \Big) 
    \partial_y 
    -\im k \,\partial_y u_s   
    -
    \partial_y^3
    \bigg] \phi_{\text{inn}, k}(t,y)
    = 
    e^{\int_0^t \sigma_k(\tau)d\tau }
    \Bigg\{
    \mathcal{R}_{1, k}^1(t,y)+ 
    \mathcal{R}_{3, k}^1(t,y)+
    \mathcal{R}_{6, k}^1(t,y)+
    \mathcal{R}_{7, k}^1(t,y)+\\
    +
    \mathcal{R}_{1, k}^2(t,y)+
    \mathcal{R}_{2, k}^2(t,y)+
    \mathcal{R}_{3, k}^2(t,y)
    \bigg\} 
    +
    F(t,y)
    e^{\int_0^t \sigma_k(\tau)d\tau },
\end{aligned}
\end{equation}
where the term  $F$ will cancel out with an analogous contribution of the outer profile and it is given by:
\begin{equation}\label{eq:F-proof-prop}
\begin{aligned}
    F(t,y) 
    :=
    \Bigg\{
    \bigg( 
            \Big(  
                \sigma_k(t)  +  \im k\, u_s(t,y) 
            \Big) 
            \chi'(y)
            &-
            \chi'''(y)
        \bigg)
        \bigg( 
         \frac{  \gamma(t)^2 e^{\im \frac{5\pi}{4}} }{\sqrt{k}}
         +
          \beta(t)(y-a(t))^2
        \bigg)
        \,+\\
        & - 6\chi''(y) \beta(t) (y-a(t))
        - 6\chi'(y) \beta(t)
        +
       2\gamma'(t)\frac{e^{\im \frac{5\pi}{4}}\gamma(t) }{\sqrt{k}} 
    \Bigg\}
    H(y-a(t)).
\end{aligned}
\end{equation}
We now focus on the outer layer $\phi_{{\rm out}, k}$. Our starting point is the identity
\begin{equation}\label{eq:identity-for-exact-sol-large-y-prop}
    \bigg[ \Big( \partial_t + \im  k \, u_s(t,y) \Big)\partial_y  - \im k\, \partial_y u_s(t,y)  - \partial_y^3 \bigg]
       \left[
       \bigg( 
            u_s(t,y) + \frac{\sigma_k(t)}{\im k}
        \bigg)
        e^{ \int_0^t \sigma_k(\tau) d \tau}
        \right]
        =
        0,\qquad (t,y) \in [0, T_{\rm max}]\times \mathbb R_+.
\end{equation}
Recall from \eqref{eq:phi-inn-out} that $\phi_{{\rm out}, k}$ is defined as
\begin{equation*}
    \phi_{{\rm out}, k}(t,y) 
        =
        \big(1 - \chi(y)\big)
        H\big(y-a(t)\big)
        \bigg( 
            u_s(t,y)
            +
            \frac{\sigma_k(t)}{\im k}
        \bigg)
        e^{ \int_0^t \sigma_k(\tau) d \tau}.
\end{equation*}
We hence apply the following product rule
\begin{equation*}
\begin{aligned}
    \big(1 - \chi(y)\big)
        H\big(y-a(t)\big)
        &
        \bigg[  
            \Big( \partial_t + \im k \, u_s(t,y) \Big)\partial_y
        \bigg]
        \bigg[
            \bigg( 
            u_s(t,y)
            +
            \frac{\sigma_k(t)}{\im k}
        \bigg)e^{ \int_0^t \sigma_k(\tau) d \tau}
        \bigg]
        =
    \\
    &= 
    \Big( \partial_t + \im k \, u_s(t,y) \Big)
    \bigg[
        \big(1 - \chi(y)\big)
        H\big(y-a(t)\big)
       \partial_y
        \bigg( 
            u_s(t,y)
            +
            \frac{\sigma_k(t)}{\im k}
        \bigg)e^{ \int_0^t \sigma_k(\tau) d \tau}
     \bigg],
\end{aligned}
\end{equation*}
since $\partial_t [(1-\chi(y))H(y-a(t))] = 0$ for any $(t,y) \in [0, T_{\rm max}]\times \mathbb R_+$.  Additionally, by bringing the $\partial_y$ derivative outside the brackets, we can rewrite the last expression as:
\begin{equation*}
\begin{aligned}
     \big(1 &- \chi(y)\big)
        H\big(y-a(t)\big)
        \bigg[  
            \Big( \partial_t + \im k \, u_s(t,y) \Big)\partial_y
        \bigg]
        \bigg[
            \bigg( 
            u_s(t,y)
            +
            \frac{\sigma_k(t)}{\im k}
        \bigg)e^{ \int_0^t \sigma_k(\tau) d \tau}
        \bigg]
        =
    \\
    &
    = 
    \Big( \partial_t + \im k \, u_s(t,y) \Big)\partial_y \phi_{\text{out}, k}(t,y)
    +
    \Big( \partial_t + \im k \, u_s(t,y) \Big)
    \bigg[
        \chi'(y)
        H\big(y-a(t)\big)
            \bigg( 
            u_s(t,y)
            +
            \frac{\sigma_k(t)}{\im k}
        \bigg)e^{ \int_0^t \sigma_k(\tau) d \tau}
     \bigg].
\end{aligned}
\end{equation*}
We multiply \eqref{eq:identity-for-exact-sol-large-y-prop} by the factor $(1-\chi(y))H(y-a(t))$, commuting this factor with the third-order derivative $\partial_y^3$
\begin{align*}
    \bigg[ \Big( \partial_t 
    &+ \im k \, u_s(t,y)  \Big)\partial_y  - \im k\, \partial_y u_s(t,y)   - \partial_y^3 \bigg]
    \phi_{\text{out,k}}(t,y)
    =
    \\
    &=
    -
    \Big(\partial_t + \im k u_s(t,y) \Big)
    \bigg[ 
        \chi'(y) H(y-a(t)) 
        \bigg( 
            u_s(t,y)
            +
            \frac{\sigma_k(t)}{\im k}
        \bigg)
        e^{\int_0^t \sigma_k(\tau)d\tau }
    \bigg]
    +
    \\
    &\hspace{1cm}
    +
    e^{\int_0^t \sigma_k(\tau)d\tau }
    H\big(y-a(t)\big)
    \Bigg\{
    \chi'''(y) 
    \bigg( 
            u_s(t,y)
            +
            \frac{\sigma_k(t)}{\im k}
     \bigg)
    +3 \chi'(y)  \partial_y^2 u_s(t,y)
    +3 \chi''(y) \partial_y u_s(t,y) 
    \Bigg\}.
\end{align*}
We can further rearrange this identity to isolate the time derivative:
\begin{equation}\label{eq:id1-for-phi-out}
\begin{aligned}
    \bigg[ \Big( \partial_t 
    + \im k \, u_s (t,y) \Big)\partial_y  - \im k\, \partial_y u_s (t,y)  - \partial_y^3 \bigg]
    \phi_{\text{out,k}}(t,y)
    =
    e^{\int_0^t \sigma_k(\tau)d\tau }
    H(y-a(t))
    \Bigg\{ 
      -
     \chi'(y) 
     \partial_t\bigg(
         u_s(t,y) + \frac{\sigma_k(t)}{\im k}
     \bigg)
     +
     \\
     -\,
     \bigg(
        \Big( \sigma_k(t) + \im k\, u_s(t,y) \Big) \chi'(y) 
        -\chi'''(y)
     \bigg)
     \bigg( 
        u_s(t,y) 
        +
        \frac{\sigma_k(t)}{\im k}
     \bigg) 
     +3 \chi'(y)   \partial_y^2 u_s(t,y)
     +3 \chi''(y)  \partial_y u_s(t,y) 
    \Bigg\}.
\end{aligned}
\end{equation}
To develop the first term on the right-hand side, we expand  
$\tfrac{\sigma_k(t)}{\im k} =\tfrac{e^{\im \frac{5\pi}{4}}\gamma(t)^2}{\sqrt{k}} -\alpha(t) =
\tfrac{e^{\im \frac{5\pi}{4}} \gamma(t)^2}{\sqrt{k}} -u_s(t, a(t))$. Furthermore, since $u_s$ is a solution to the heat equation, this allows us to express the time derivative as follows:
\begin{align*}
    \partial_t \Big( u_s(t,y) + \frac{\sigma_k(t)}{\im k} \Big)
    &= 
    \partial_t u_s(t,y) + \frac{d}{dt}\left[\frac{e^{i\frac{5\pi}{4}}\gamma(t)^2}{\sqrt{k}} - u_s(t,a(t)) \right]
    \\
    &= 
    \partial_y^2 u_s(t,y)
    +
    \frac{e^{\im \frac{5\pi}{4}}2\gamma(t)\gamma'(t)}{\sqrt{k}}  
    - 
    (\partial_t u_s)(t,a(t)) 
    -
    \underbrace{\partial_y u_s(t,a(t))}_{ = 0}a'(t),    
    \\
    &
    = 
    \partial_y^2 u_s(t,y)
    -
    \partial_y^2 u_s(t,a(t))
    +
    \frac{e^{\im \frac{5\pi}{4}}2\gamma(t)\gamma'(t)}{\sqrt{k}}
    \\
    &
    = 
    \partial_y^2 u_s(t,y)
    -
    2\beta(t)
    +
    \frac{e^{\im \frac{5\pi}{4}}2\gamma(t)\gamma'(t)}{\sqrt{k}}. 
\end{align*}
Substituting this result into the identity \eqref{eq:id1-for-phi-out}, we reorder the right-hand side of the equation as follows:
\begin{equation*}
\begin{aligned}
    \bigg[ \Big( \partial_t 
    + \im  k \, u_s (t,y) \Big)\partial_y  - \im k\, \partial_y u_s (t,y)  - \partial_y^3 \bigg]
    \phi_{\text{out,k}}(t,y)
    =
    e^{\int_0^t \sigma_k(\tau)d\tau }
    H(y-a(t))
    \Bigg\{ 
    -
      \chi'(y) 
      \Big(
           \partial_y^2 u_s(t,y) - 2\beta(t)
      \Big)
    +
    \\
    - \chi'(y) \frac{e^{\im \frac{5\pi}{4}}2\gamma(t)\gamma'(t)}{\sqrt{k}}
    -  
    \bigg(
       \Big( \sigma_k(t) + \im k\, u_s(t,y) \Big) \chi'(y)
       -\chi'''(y) 
    \bigg)
    \bigg( 
      u_s(t,y) -\alpha(t) -\beta(t) (y-a(t))^2
    \bigg) 
      \,+
      \\ 
    -  
    \bigg(
       \Big( \sigma_k(t) + \im k\, u_s(t,y) \Big) \chi'(y)
       -\chi'''(y) 
    \bigg)
    \bigg( 
        \frac{\gamma(t)^2 e^{\im \frac{5\pi}{4}}}{\sqrt{k}}
        +
        \beta(t) (y-a(t))^2
    \bigg) 
    +
    3 \chi'(y)  
    \Big( 
      \partial_y^2 u_s(t,y) -2\beta(t)
    \Big)
    +
    \\
    +6 \chi'(y)  \beta(t) 
    +
    \Big( 
      \partial_y u_s(t,y) - 2\beta(t) (y-a(t)) 
    \Big)
    3 \chi''(y) 
      + 6 \chi''(y)  \beta(t) (y-a(t)) 
    \Bigg\}.
\end{aligned}
\end{equation*}
We invoke once more the notation for the remainder terms as defined in \eqref{eq:first-family-remainders} and \eqref{eq:second-family-remainders}:
\begin{equation*}
\begin{aligned}
    \bigg[ \Big( \partial_t 
    + \im k \, u_s (t,y) \Big)\partial_y  - \im k\, \partial_y u_s (t,y)  - \partial_y^3 \bigg]
    \phi_{\text{out,k}}(t,y)
    =
    e^{\int_0^t \sigma_k(\tau)d\tau }
    \Bigg\{ 
    \mathcal{R}_{5, k}^1(t,y)
    - \chi'(y) \frac{e^{\im \frac{5\pi}{4}}2\gamma(t)\gamma'(t)}{\sqrt{k}}
    +
    \mathcal{R}_{2,k}^1(t,y)
    \\
    -
    \bigg(
       \Big( \sigma_k(t) + \im k\, u_s(t,y) \Big) \chi'(y)
       -\chi'''(y) 
    \bigg)
    \bigg( 
        \frac{\gamma(t)^2 e^{\im \frac{5\pi}{4}}}{\sqrt{k}}
        +
        \beta(t) (y-a(t))^2
    \bigg) H(y-a(t))
    +
    6\chi'(y)\beta(t)H(y-a(t))   
    \\
    + \mathcal{R}_{4, k}^1(t,y)
    + 6 \chi''(y) \beta(t) (y-a(t))H(y-a(t)) 
    \Bigg\}.
\end{aligned}
\end{equation*}
From the definition of $F$ in \eqref{eq:F-proof-prop}, this leads us to the following simplified form:
\begin{equation}\label{eq:proof-eq-of-phiout}
\begin{aligned}
    \bigg[ \Big( \partial_t 
    + \im k \, u_s (t,y) \Big)\partial_y  &- \im k\, \partial_y u_s (t,y)  - \partial_y^3 \bigg]
    \phi_{\text{out,k}}(t,y)
    =
    \\
    &=
    e^{\int_0^t \sigma_k(\tau)d\tau }
    \Bigg\{ 
    \mathcal{R}_{2, k}^1(t,y)+
    \mathcal{R}_{4, k}^1(t,y)+
    \mathcal{R}_{5, k}^1(t,y)
    \bigg\}
    -F(t,y)e^{\int_0^t \sigma_k(\tau)d\tau }.
\end{aligned}
\end{equation}
Finally, we sum Equation \eqref{eq:proof-eq-of-phiinn} for the inner layer and  \eqref{eq:proof-eq-of-phiout} for the outer layer. This yields
\begin{equation}
\begin{aligned}
    \bigg[ \Big( \partial_t 
    + \im k \, u_s (t,y) \Big)\partial_y  - \im k\, \partial_y u_s (t,y)  - \partial_y^3 \bigg]
    \Big(
        \phi_{\text{inn,k}}(t,y)
        +
        \phi_{\text{out,k}}(t,y)
    \Big)
    =
    e^{\int_0^t \sigma_k(\tau)d\tau }
    \bigg(
    \sum_{j = 1}^7
    \mathcal{R}_{j, k}^1
    +
    \sum_{j = 1}^3
    \mathcal{R}_{j, k}^2
    \bigg),
\end{aligned}
\end{equation}    
which concludes the proof of the proposition.
\end{proof}

\section{Estimates of the remainders}\label{sec:estimates-forced-terms}
\noindent 
In the next proposition, we establish some estimates for the forced solution $u_k^{\rm fr}$ and its associated remainders along those described by Part-(a) of \Cref{sec:novelties-some-remarks}. 
\begin{prop}\label{prop:estimates-remainders-and-uf}
    Consider $u^{\rm fr}_k = \partial_y (\phi_{{\rm inn}, k} +\phi_{{\rm out}, k} ) $ of \Cref{def:inner-outer-profiles} and the two families of remainders $\mathcal{R}_{1,k}^1,\dots,\mathcal{R}_{7,k}^1$ and $\mathcal{R}_{1,k}^2,  \mathcal{R}_{2,k}^2, \mathcal{R}_{3,k}^2$ introduced in \Cref{prop:remainder-exact-form}. Define the positive constants
    \begin{alignat*}{4}
        C_{\mathcal{R}} 
        &=
        C_{\mathcal{R}}\big( U_s,\, d,\,\chi \big) 
        &&:=
            10^2
            \bigg( 1 +  \frac{2}{|U_s''(a_0)|} \bigg) 
            \|\chi \|_{H^3} 
            \big( 1 + \|  U_s \|_{W^{4,1}}\big)^2 
            \big(1+d\,\big)^4,
            \\
            \mathcal{D}_m   
            &=
            \mathcal{D}_m  \big( U_s,\, d,\,\chi \big)  
            &&:=
             m! 
             \,2^{m+2} ( 1 + d )^{m+1}
            \Big( 1 + \|U_s\|_{H^m_\lambda}+\|U_s\|_{W^{3,1}} \Big)^m
            \Big( 1 + \|\chi \|_{W^{m+1,\infty}} \Big).
    \end{alignat*}
    Then the following results hold true:
    \begin{itemize}[leftmargin=0.7cm]
        \item[(i)] For any $t\in [0, T_{\rm max}]$ and any $k\in \mathbb N$  
        \begin{equation} \label{eq:EstimateRemainderFirstKind}
            \sum_{j = 1}^7
            \big\|\, 
                \mathcal{R}_{j,k}^1(t, \cdot) 
            \,\big\|_{L^2(\mathbb R_+)} 
            \leq 
            C_{\mathcal{R}}\, k 
            \frac{ e^{ - \frac{d^2}{16 t} } }{\sqrt{t}}.
        \end{equation}
        \item[(ii)] For any $t\in [0, T_{\rm max}]$ and any $k\in \mathbb N$ 
        \begin{equation}\label{eq:prop-remainders-exp-decay-in-k}
            \sum_{j = 1}^3 \big\|\,\mathcal{R}_{j,k}^2(t, \cdot) \,\big\|_{L^2(\mathbb R_+)} 
            \leq 
            C_{\mathcal{R}}
            k^\frac{3}{4}
            \frac{1}{d} 
           e^{    -\sqrt{k} \,\big(\frac{d}{16}\big)^2  \big(\frac{|U_s''(a_0)|}{2} \big)^{1/2}}.
        \end{equation}
        \item[(iii)] For any time $t\in [0, T_{\rm max}]$ and any $k\in \mathbb N$ 
        \begin{equation} \label{eq:ForcedSolutionLowerBound}
            \|  u_k^{\rm fr}(t, \cdot) \|_{L^2(\mathbb R_+)}
            \geq 
            \frac{1}{2}
            \| U_s' \|_{L^2(a_0+\frac{d}{2}, \infty)}
            e^{t \sqrt{k} \frac{|U_s''(a_0)|^{1/2}}{2\sqrt{2}}}.
        \end{equation}
        \item[(iv)] For any $k\in \mathbb N$, $m \in \mathbb N\setminus\{1, 2\}$ and $\lambda>0$, the initial data $u_k^{\rm fr}(0, \cdot)$ satisfies
        \begin{equation}  \label{eq:InitialDataUpperBound}
            \|  u_k^{\rm fr}(0, \cdot ) \|_{H^m_\lambda(\mathbb R_+)} 
            \leq \mathcal{D}_m \,e^{2 \lambda a_0}  k^{\frac{2m-1}{4} }.
        \end{equation}
    \end{itemize}
\end{prop}
%
%
%
%
%
\subsection{The upper bound on the first family of remainders} 
We begin with the proof of part~$(i)$ of \Cref{prop:estimates-remainders-and-uf} on the first family of remainders $\mathcal{R}_{j,k}^1$. The argument relies on the following  estimates about the deviation of $u_s$ from its second-order Taylor polynomial at $y = a(t)$, the derivatives $a'(t)$ and $\gamma'(t)$, and the growth properties of $\Upsilon$:
\begin{itemize}[leftmargin=0.7cm]
 \item[(a)] 
 For any time $t \in [0, T_{\rm max}]$ and any index $n \in \{0,1,2\}$
    \begin{equation}\label{eq:ineq-in-lemma-us}
    \begin{aligned}
        &\max_{a_0-\frac{d}{2}\leq y \leq  a_0+ \frac{d}{2}}
        \bigg|
            \partial_y^n
            \Big( 
                u_s(t,y) 
                -  
                \alpha(t)
                - \beta(t)  (y-a(t))^2 
            \Big)
        \bigg|
        \leq 
        d^{3-n}
        \big\|  U_s^{(3)}\big\|_{L^1(\mathbb R_+)}
        \frac{e^{ - \frac{d^2}{16 t} }}{\sqrt{\pi t}}
        .
    \end{aligned}
    \end{equation}
    \item[(b)] For any time $t \in [0, T_{\rm max}]$
    \begin{equation}\label{eq:ineq-in-lemma-us-a'-gamma'}
    \begin{aligned}
        |a'(t)| 
        &\leq 
       \frac{2}{|U_s''(a_0)|}
       \| U_s^{(3)} \|_{L^1(\mathbb R_+)}
       \frac{e^{ - \frac{d^2}{16 t} }}{\sqrt{\pi t}},
       \qquad 
       |\gamma'(t)|
       \leq 
       \frac{2}{|U_s''(a_0)|^\frac{3}{4}}
       \| U_s^{(3)} \|_{W^{1,1}(\mathbb R_+)}
       \frac{e^{ - \frac{d^2}{16 t} }}{\sqrt{\pi t}}
       .
    \end{aligned}
    \end{equation}
    \item[(c)] For any frequency $k \in \mathbb{N}$, any time $t \in [0, T_{\rm max}]$ and any index $n\in \{ 0, 1, 2\}$
    \begin{equation}\label{eq:ineq-in-lemma-us-Upsilon}
        \max_{a_0-\frac{d}{2}\leq y \leq  a_0+ \frac{d}{2}}
       \Big| \Upsilon^{(n)}\big(\sqrt[4]{k}\,z (t,y) \big) \Big|
       \leq 
       4 \big(  1+ |U_s''(a_0)|^{\frac{2-n}{4}} d^{2-n}   \big) k^\frac{2-n}{4}.
    \end{equation}
    \end{itemize}
In this subsection, we focus on the application of \eqref{eq:ineq-in-lemma-us}, \eqref{eq:ineq-in-lemma-us-a'-gamma'} and \eqref{eq:ineq-in-lemma-us-Upsilon} to the remainders $\mathcal{R}_{j,k}^1$, postponing their proofs to \Cref{lemma:estimates-for-us-appx} in the Appendix, where they are formalised using the Taylor theorem and the Green's formula for the heat kernel with Dirichlet boundary conditions.
\begin{proof}[Proof of \Cref{prop:estimates-remainders-and-uf}, Part~$(i)$]
To shorten the notation, we set the interval $I = [a_0 - \tfrac{d}{2}, a_0 + \tfrac{d}{2}]$. Additionally, we repeatedly use the following upper bounds for $|U_s''(a_0)|$ and $\gamma(t)$  at any time $t\in [0, T_{\rm max}]$:
\begin{equation}\label{eq:|Us''(a)|and|gamma'(t)|}
    |U_s''(a_0)|= 
    \bigg| \int_{a_0}^{+\infty} U_s^{(3)}(\omega) d\omega \bigg| \leq \| U_s^{(3)} \|_{L^1}
    ,\qquad 
    |\gamma(t)| 
    =
    \sqrt[4]{\frac{|\partial_y^2 u_s(t,a(t))|}{2}}
    \leq 
    \sqrt[4]{\frac{2|U_s''(a_0)|}{2}} = \| U_s^{(3)} \|_{L^1}^\frac{1}{4},
\end{equation}
which are satisfied thanks to our assumption in \eqref{eq:assumptions-for-Tmax}.  We begin by estimating the first five remainders, proving that they satisfy the bound
\begin{equation}\label{eq:ineq-for-first-five-remainders-first-family}
    \| \mathcal{R}_{j,k}^1 (t)  \|_{L^2}
    \leq k\frac{ e^{ - \frac{d^2}{16 t} } }{\sqrt{t}}\,
    \,\frac{4}{\sqrt{\pi}}
    \| \chi \|_{H^3}
    \big( 1 + \|  U_s \|_{W^{4,1}}\big)^2 
    \big(1+d\,\big)^4,\qquad j \in \{1, \dots, 5\}.
\end{equation}
The first remainder $\mathcal{R}_{1, k}^1$ in \eqref{eq:first-family-remainders} has $L^2$-norm equal to 
\begin{equation*}
    \| \mathcal{R}_{1,k}^1 (t)  \|_{L^2}
    = 
    k^\frac{3}{4}
    \gamma(t)^3 
    \bigg(
    \int_0^\infty 
    \bigg|
        \Big(
            u_s(t, y) -\alpha(t) - \beta(t) (y -a(t))^2
        \Big)
        \chi(y) \Upsilon'\big(\sqrt[4]{k}\, z(t, y) \big)
    \bigg|^2
    dy
    \bigg)^\frac{1}{2}.
\end{equation*}
We hence recall that $\chi$ is supported on $I = [a_0 - \tfrac{d}{2}, a_0 + \tfrac{d}{2}]$, therefore, thanks to \eqref{eq:|Us''(a)|and|gamma'(t)|} and the H\"older inequality:
\begin{equation*}
\begin{aligned}
    \| \mathcal{R}_{1,k}^1 &(t )  \|_{L^2}
    \leq 
    k^\frac{3}{4}
    \| U_s^{(3)} \|_{L^1}^\frac{3}{4}
    \left\| 
        u_s(t, \cdot ) -\alpha(t) - \beta(t) (\cdot -a(t))^2  
    \right\|_{L^\infty(I)}
    \| \chi \|_{L^2}
    \left\| 
        \Upsilon'\big(\sqrt[4]{k}\, z(t, \cdot)  
    \right\|_{L^\infty(I)}
\end{aligned}
\end{equation*}
where we denoted $L^p = L^p(\mathbb{R}_+)$ for any $p \geq 1$. Next, we invoke \eqref{eq:ineq-in-lemma-us} with $n = 0$ and \eqref{eq:ineq-in-lemma-us-a'-gamma'} with $n = 1$, to obtain
\begin{equation*}
\begin{aligned}
    \| \mathcal{R}_{1,k}^1 (t)  \|_{L^2}
    &\leq 
    k^\frac{3}{4} 
    \| U_s^{(3)} \|_{L^1}^\frac{3}{4}
    \;
    \big\|  U_s^{(3)} \big\|_{L^1}
    d^3\,
    \frac{ e^{ - \frac{d^2}{16 t} } }{\sqrt{\pi t}}
    \;
    \| \chi \|_{L^2}
    \;
    4 
    \big( 1 + \big\|  U_s^{(3)} \big\|_{L^1}^\frac{1}{4} d \big) \sqrt[4]{k}
    \\
    &\leq 
    4 k 
    \big( 1 + 
        \| U_s^{(3)} \|_{L^1}
    \big)^\frac{7}{4}
    \big( 1 + d \big)^3
    \frac{ e^{ - \frac{d^2}{16 t} } }{\sqrt{\pi t}}
    \;
    \| \chi \|_{H^3}
    \;
    \big( 1 + 
        \| U_s^{(3)} \|_{L^1}
    \big)^\frac{1}{4}
    \big( 1 + d \big)
    \\
    &
    \leq 
    k\frac{ e^{ - \frac{d^2}{16 t} } }{\sqrt{t}}\,
    \,\frac{4}{\sqrt{\pi}}
    \| \chi \|_{H^3}
    \big( 1 + \|  U_s \|_{W^{4,1}}\big)^2 
    \big(1+d\,\big)^4,
\end{aligned}
\end{equation*}
which is \eqref{eq:ineq-for-first-five-remainders-first-family} for $j = 1$. Next, we address the second remainder $\mathcal{R}_{2,k}^1$ in \eqref{eq:first-family-remainders}, whose $L^2$-norm is equal to
\begin{equation*}
    \| \mathcal{R}_{2,k}^1 (t)  \|_{L^2}
    =
    \bigg(
    \int_{a(t)}^{\infty} 
    \bigg|
    \Big(  
        u_s(t,y)  - 
        \alpha(t) -
        \beta(t) (y-a(t))^2
    \Big)
    \Big(
        \chi^{(3)}(y) -
        \Big(
            \sigma_k(t)   +  \im  k \, u_s(t,y)   
        \Big)
        \chi'(y)
    \Big)
    \bigg|^2
    dy
    \bigg)^\frac{1}{2}.
\end{equation*}
We hence remark that the function $\sigma_k(t)= e^{-\frac{i\pi}{4}}\gamma(t)^2 \sqrt{k} - i k \,u_s(t,a(t))$ defined in \eqref{eq:z-sigmak-in-def-of-phi-inn-out} satisfies 
\begin{equation}\label{eq:est-of-sigmak(t)}
    |\sigma_k(t)| \leq  
    \gamma(t)^2 \sqrt{k} +  k \| u_s \|_{L^\infty(\mathbb R_+ \times \mathbb R_+) }
    \leq
        \sqrt{k} \,
        | U_s''(a_0) |^\frac{1}{2} +  k \| U_s \|_{L^\infty(\mathbb R_+)} 
    \leq 
    2 k \big( 1 + \| U_s \|_{W^{4,1}}\big),
\end{equation}
thanks to \eqref{eq:|Us''(a)|and|gamma'(t)|} and the maximum principle of the heat equation. Thus applying the H\"older inequality and the relation \eqref{eq:ineq-in-lemma-us} with $n = 0$, we obtain
\begin{equation*}
\begin{aligned}
    \| \mathcal{R}_{2,k}^1 (t)  \|_{L^2}
    &\leq 
    \left\| 
        u_s(t, \cdot ) -\alpha(t) - \beta(t) (\cdot -a(t))^2  
    \right\|_{L^\infty(I)}
    \Big( \| \chi^{(3)} \|_{L^2} + \big( |\sigma_k(t)| + k 
    \| u_s(t) \|_{L^\infty} \big) \| \chi' \|_{L^2} \Big)
    \\
    &\leq 
    \big\|  U_s^{(3)} \big\|_{L^1}
    d^3\,
   \frac{e^{ - \frac{d^2}{16 t} } }{\sqrt{\pi t}}
    \| \chi \|_{H^3}
    \Big( 1 +2 k \big( 1 + \| U_s \|_{W^{4,1}}\big) + 
    k \| U_s \|_{L^\infty} \Big)
    \\
    &\leq 
    k \frac{ e^{ - \frac{d^2}{16 t} } }{\sqrt{t}}
    \frac{4}{\sqrt{\pi}}
    \| \chi \|_{H^3}
    \big( 1 + \|  U_s \|_{W^{4,1}}\big)^2 
    \big(1+d\,\big)^4,
\end{aligned}   
\end{equation*}
which is \eqref{eq:ineq-for-first-five-remainders-first-family} for $j = 2$. The third remainder $\mathcal{R}_{3,k}^1$ in \eqref{eq:first-family-remainders} is estimated using \eqref{eq:ineq-in-lemma-us} with $n=1$ and \eqref{eq:ineq-in-lemma-us-Upsilon} with $n=0$:
\begin{align*}
    \| \mathcal{R}_{3,k}^1 (t)  \|_{L^2} &= 
    \sqrt{k}\,
    \gamma(t)^2 
    \bigg(
    \int_0^\infty 
    \bigg|
        \Big(
            \partial_y u_s(t, y) - 2 \beta(t) (y -a(t))
        \Big)
        \chi(y) \Upsilon\big(\sqrt[4]{k}\, z(t, y) \big)
    \bigg|^2
    dy
    \bigg)^\frac{1}{2}
    \\
    &\leq
    \sqrt{k}\, \|U_s^{(3)}\|_{L^1}^\frac{1}{2} 
    \big\| \partial_y u_s(t, \cdot ) - 2 \beta(t) (\cdot  -a(t))
    \big\|_{L^\infty(I)} \| \chi \|_{L^2} 
    \|  \Upsilon\big(\sqrt[4]{k}\, z(t, \cdot ) \|_{L^\infty(I)}
    \\
    &\leq 
    \sqrt{k}\, 
    \|U_s^{(3)}\|_{L^1}^\frac{1}{2}
    d^2 \| U_s^{(3)} \|_{L^1} \frac{e^{ - \frac{d^2}{16 t} } }{\sqrt{\pi t}}
    \| \chi \|_{H^3} 
    4 \big( 1+ \| U_s^{(3)} \|_{L^1}^\frac{1}{2} d^2 \big) \sqrt{k}
    \\
    &
    \leq 
    k\frac{ e^{ - \frac{d^2}{16 t} } }{\sqrt{\pi t}}\frac{4}{\sqrt{\pi}}
    \| \chi \|_{H^3}
    \big( 1 + \|  U_s \|_{W^{4,1}}\big)^2 
    \big(1+d\,\big)^4,
\end{align*}
namely \eqref{eq:ineq-for-first-five-remainders-first-family} for $j = 3$. 
The fourth and fifth remainders $\mathcal{R}_{4,k}^1$, $\mathcal{R}_{5,k}^1$ in   \eqref{eq:first-family-remainders} can be tackled together using \eqref{eq:ineq-in-lemma-us} once more:
\begin{align*}
    \max
    \Big\{ 
        \| \mathcal{R}_{4,k}^1 (t)\|_{L^2},\,
        \| \mathcal{R}_{5,k}^1 (t)  \|_{L^2}
    \Big\}
    &\leq 
    \max
    \Big\{ 
        \| \partial u_s (t, \cdot ) - 2\beta(t)(\cdot -a(t)) \|_{L^\infty(I)},\, 
        \| \partial^2 u_s (t, \cdot ) - 2\beta(t) \|_{L^\infty(I)}
    \Big\}
    4 \| \chi \|_{H^3}
    \\
    &\leq 
    \max\{ d^2,\, d\}
     \| U_s^{(3)} \|_{L^1}\frac{e^{ - \frac{d^2}{16 t} } }{\sqrt{\pi t}}
    \leq 
    4\,
    k\,
    \| \chi \|_{H^3}
    \big( 1 + \|  U_s \|_{W^{4,1}}\big)^2 
    \big(1+d\,\big)^4
    \;
    \frac{ e^{ - \frac{d^2}{16 t} } }{\sqrt{\pi t}},
\end{align*}
which is \eqref{eq:ineq-for-first-five-remainders-first-family} for $j = 4,\,5$. Since $20/\sqrt{\pi} \leq 12$, we have proven therefore that
\begin{equation}\label{eq:sum-first-five-remainders-first-family}
    \sum_{j = 1}^5 
    \| \mathcal{R}_{j,k}^1 (t)  \|_{L^2}
    \leq k\frac{ e^{ - \frac{d^2}{16 t} } }{\sqrt{t}}\,
    \,12
    \| \chi \|_{H^3}
    \big( 1 + \|  U_s \|_{W^{4,1}}\big)^2 \big(1+d\,\big)^4.
\end{equation}
It remains to estimate $\mathcal{R}_{6, k}^1$ and $\mathcal{R}_{7, k}^1$ in \eqref{eq:first-family-remainders}. For $\mathcal{R}_{6, k}^1$, we first  apply  \eqref{eq:ineq-in-lemma-us} and \eqref{eq:ineq-in-lemma-us-Upsilon}, to gather
\begin{align*}
    \| \mathcal{R}_{6,k}^1 (t)\|_{L^2}
    &\leq 
    \frac{1}{\sqrt{k}}
    \gamma'(t)  \|\chi \|_{H^1} \Big( 2 \gamma(t) + 2 \gamma(t)
    \|  \Upsilon\big(\sqrt[4]{k}\, z(t, \cdot ) \|_{L^\infty(I)} 
    + 3 \gamma(t)^2
    \|  \Upsilon'\big(\sqrt[4]{k}\, z(t, \cdot ) \|_{L^\infty(I)} \Big)
    \\
    &\leq 
    \frac{1}{\sqrt{k}}
     \gamma'(t)  
     \|\chi \|_{H^1} 
     \bigg( 
      2 \| U_s^{(3)} \|_{L^1}^\frac{1}{4} + 
      8 \| U_s^{(3)} \|_{L^1}^\frac{1}{4} ( 1 +  \| U_s^{(3)} \|_{L^1}^\frac{1}{2} d^2 ) \sqrt{k} + 
      12\| U_s^{(3)} \|_{L^1}^\frac{1}{2}
      ( 1 +  \| U_s^{(3)} \|_{L^1}^\frac{1}{4} d ) \sqrt[4]{k}
      \bigg)
      \\
      & 
     \leq 
     22\,
     \gamma'(t)  
     \|\chi \|_{H^3} 
     \big( 1+ d\big)^2
     \big( 1 + \| U_s^{(3)} \|_{L^1} \big).
\end{align*}
Next, thanks to \eqref{eq:ineq-in-lemma-us-a'-gamma'} and the fact that $22/\sqrt{\pi} \leq 13$, we further obtain
\begin{equation*}
    \| \mathcal{R}_{6,k}^1 (t)\|_{L^2} \leq 
    k
    \frac{ e^{ - \frac{d^2}{16 t} } }{\sqrt{t}}
    \,
    13 \frac{2}{|U_s''(a_0)|}  
    \|\chi \|_{H^3} 
    \big( 1 + \|  U_s \|_{W^{4,1}}\big)^2 
    \big(1+d\,\big)^4.
\end{equation*}
Finally
\begin{align*}
    \| \mathcal{R}_{7,k}^1 (t)\|_{L^2}
    &\leq 
    \Big( |\gamma'(t)| d + \gamma(t)|a'(t)|\Big)
    \bigg( \| \chi' \|_{L^2} \frac{\gamma(t)^2}{\sqrt[4]{k}} 
     \|  \Upsilon'\big(\sqrt[4]{k}\, z(t, \cdot ) \|_{L^\infty(I)} 
     +
     \| \chi \|_{L^2}  \gamma(t)^3  \| \Upsilon''\big(\sqrt[4]{k}\, z(t, \cdot ) \|_{L^\infty(I)} 
     \bigg)
     \\
     &\leq 
     \| U_s^{(3)} \|_{W^{1,1}}
    \frac{ e^{ - \frac{d^2}{16 t} } }{\sqrt{\pi t}}
    \bigg( 
    \frac{2}{|U_s''(a_0)|^\frac{3}{4}} 
    d + \frac{2}{|U_s''(a_0)|^\frac{1}{4}} 
    \bigg)
    \| \chi \|_{H^3}
    4
    \bigg(  \| U_s^{(3)} \|_{L^1}  \big( 1 + \| U_s^{(3)} \|_{L^1}^\frac{1}{4}d \big)
     + \| U_s^{(3)} \|_{L^1}^\frac{3}{4}
     \bigg)
     \\
     &\leq 
     k 
     \frac{ e^{ - \frac{d^2}{16 t} } }{\sqrt{t}}
     \,
     5
     \bigg(  1 +  \frac{2}{|U_s''(a_0)|}  \bigg)
    \|\chi \|_{H^3} 
    \big( 1 + \|  U_s \|_{W^{4,1}}\big)^2 
    \big(1+d\,\big)^4.
\end{align*}
Taking the sum between \eqref{eq:sum-first-five-remainders-first-family} and the last two relations concludes the proof of Part $(i)$ of \Cref{prop:estimates-remainders-and-uf}.
\end{proof}
\subsection{An upper bound on the second family of remainders} We now deal with Part~(ii) of \Cref{prop:estimates-remainders-and-uf} about the second family of remainders $\mathcal{R}_{1,k}^2$, $\mathcal{R}_{2,k}^2$, $\mathcal{R}_{3,k}^2$ as defined in \eqref{eq:second-family-remainders}. We prove that their $L^2$-norms exhibit exponential decay in the frequencies $k \in \mathbb{N}$, as described by \eqref{eq:prop-remainders-exp-decay-in-k}.

\smallskip 
\noindent 
These terms depend mainly on combinations between the function $\Upsilon$ of \eqref{def:Upsilon-fct} and the Heaviside function $H(y-a(t))$. In particular, we will repeatedly rely on the following inequality about the $\erf$-term of $\Upsilon$ and $H$:
\begin{equation}\label{eq:erf-H-estimate}
    \bigg| 
        \,
        \erf\bigg( \frac{\sqrt[4]{k}\,z(t,y)}{\sqrt{2}}\bigg) - H(y-a(t)) 
        \,
    \bigg| 
    \leq     
    \frac{13}{|U_s''(a_0)|^\frac{1}{4}} 
    \frac{1}{\sqrt[4]{k}\,d}
    \,
    e^{    -\sqrt{k} \,\big(\frac{d}{16}\big)^2  \big(\frac{|U_s''(a_0)|}{2} \big)^{1/2}}
\end{equation}
for any $y \in \mathbb{R}_+ \setminus [a_0 - \tfrac{d}{4}, a_0 + \tfrac{d}{4}]$. Roughly speaking, this inequality holds true, since $z(t,y) = \gamma(t)(y-a(t))e^{-\frac{i\pi}{8}}$ is in the sector $|\arg(z) |< \tfrac{\pi}{4}$ if $y >a_0 +\tfrac{d}{2}>a(t)$ and $|\arg(-z) |< \tfrac{\pi}{4}$ if $y<a_0 -\tfrac{d}{2}<a(t)$, so that the $\erf$ function converges exponentially to $\pm 1$ with same sign to the one of $y-a(t)$. We postpone the formal proof of \eqref{eq:erf-H-estimate} to \Cref{lemma:erf-H}, in the Appendix, and focus on its application to $\mathcal{R}_{1,k}^2$, $\mathcal{R}_{2,k}^2$, and $\mathcal{R}_{3,k}^2$.
\begin{proof}[Proof of \Cref{prop:estimates-remainders-and-uf}, Part~$(ii)$] 
All remainders depend on some derivative of the cutoff function $\chi$, which is supported in $[a_0 - \tfrac{d}{2}, a_0 + \tfrac{d}{2}]$ and equal to $1$ on $[a_0 - \tfrac{d}{4}, a_0 + \tfrac{d}{4}]$. Consequently, $\mathcal{R}_{1,k}^2$, $\mathcal{R}_{2,k}^2$, and $\mathcal{R}_{3,k}^2$ have support in $[0, T_{\rm max}] \times K$, with $K = [a_0 - \tfrac{d}{2}, a_0 - \tfrac{d}{4}] \cup [a_0 + \tfrac{d}{4}, a_0 + \tfrac{d}{2}]$. We will use in particular that 
\begin{equation}\label{eq:estimate-y-a(t)-propositionest-2nd-part}
    \frac{d}{8}  <  |\, y-a(t) \, | \leq d, \qquad \text{for any}\quad (t,y) \in [0, T_{\rm max}] \times K,
\end{equation}
which follows from our assumption $a(t) \in [a_0-\tfrac{d}{8}, a_0+\tfrac{d}{8}]$, introduced in \eqref{eq:assumptions-for-Tmax}. 

\smallskip 
\noindent 
Our goal is to establish that, for each $j \in \{1,2,3\}$ and any $t \in [0, T_{\rm max}]$, the following upper bound holds:
\begin{equation}\label{eq:ineq-for-remainders-second-family}
    \| \mathcal{R}_{j,k}^2(t) \|_{L^2}
    \leq 
    \frac{k^\frac{3}{4} 
    }{d} 
    e^{    -\sqrt{k} \,\big(\frac{d}{16}\big)^2  \big(\frac{|U_s''(a_0)|}{2} \big)^{1/2}}
    24
    \| \chi \|_{H^3} 
    \Big( 1 + \frac{2}{|U_s''(a_0)|} \Big)
    \big( 1 + \| U_s \|_{W^{4,1}} \big)^2(1 +d\,)^4 
    .
\end{equation}
The desired inequality \eqref{eq:prop-remainders-exp-decay-in-k} and the conclusion of the proof follows then by summing over $j = 1, 2, 3$. We begin with the estimate of $\mathcal{R}_{1,k}^2$ in \eqref{eq:second-family-remainders}, whose $L^2$-norm satisfies at any $t\in [0, T_{\rm max}]$:
\begin{align*}
   \| \mathcal{R}_{1,k}^2&(t) \|_{L^2}
   = 
   \\
   &
   \frac{\gamma(t)^2 }{\sqrt{k}}
   \bigg( 
   \int_K
   \bigg|
    \Big[
    \Upsilon(\sqrt[4]{k}\,z(t,y)) 
    -
    H(y-a(t))
    \big( 
            1
            +
            \sqrt{k}\, z(t,y)^2
    \big)
    \Big]
    \Big(
        \big(   \sigma_k(t)  +  \im k\, u_s(t,y) \big)\chi'(y) 
        - \chi'''(y) 
    \Big)
    \bigg|^2
    dy
    \bigg)^\frac{1}{2}.
\end{align*}
By applying the H\"older inequality, we obtain
\begin{equation*}
    \| \mathcal{R}_{1,k}^2(t) \|_{L^2}
    \leq 
    \frac{\gamma(t)^2 }{\sqrt{k}} 
    \| \chi \|_{H^3}
    \Big( 
        1 + |\sigma_k(t)| + k \| u_s(t,\cdot) \|_{L^\infty}
    \Big)
    \Big\| 
        \Upsilon(\sqrt[4]{k}\,z(t,\cdot )) 
    -
    H(\cdot -a(t))
    \Big( 
            1
            +
            \sqrt{k}\, z(t,\cdot)^2
    \Big)
    \Big\|_{L^\infty(K)}.
\end{equation*}
We hence recall that  $\gamma(t)\leq \| U_s^{(3)} \|_{L^1}^{1/4}$ from \eqref{eq:|Us''(a)|and|gamma'(t)|} and $|\sigma_k(t)| \leq 2k ( 1 + \| U_s\|_{W^{4,1}})$ from \eqref{eq:est-of-sigmak(t)}. Thanks to the maximum principle of the heat equation $\| u_s \|_{L^\infty(\mathbb R_+ \times \mathbb R_+)} \leq \| U_s \|_{L^\infty(\mathbb R_+)}\leq \| U_s \|_{W^{4,1}}$, we deduce therefore that
\begin{equation}\label{eq:first-ineq-of-R1k2}
    \| \mathcal{R}_{1,k}^2(t) \|_{L^2}
    \leq 
    3
    \sqrt{k} \,
     \| \chi \|_{H^3}
    \big( 
        1 + \| U_s\|_{W^{4,1}}
    \big)
    \big\| \,
        \Upsilon(\sqrt[4]{k}\,z(t,\cdot ))-\big(1 + \sqrt{k}\, z(t,\cdot)^2 \big)H(\cdot -a(t))
    \,\big\|_{L^\infty(K)}.
\end{equation}
We develop the last $L^\infty$-norm using the definition of $\Upsilon$, provided in \eqref{def:Upsilon-fct}. More precisely, for any $(t,y) \in [0, T_{\rm max}]\times K$,
\begin{equation}\label{eq:Ups-H-formula}
\begin{aligned}
    \Upsilon(\sqrt[4]{k}\,z(t,y))-\Big(1 &+ \sqrt{k}\, z(t,y)^2 \Big)H(y -a(t)) = 
    \\
    &=
    \frac{ \sqrt[4]{k}\, z(t,y)}{\sqrt{2\pi}}e^{ - \sqrt{k} \frac{z(t,y)^2}{2}} 
    +
    \big( 1 + \sqrt{k}\,z(t,y)^2 \big)
    \frac{1}{2}
    \bigg( 
        \erf\bigg( \frac{\sqrt[4]{k}\,z(t,y)}{\sqrt{2}}\bigg)  
         - H(y-a(t))
    \bigg).
\end{aligned}
\end{equation}
We recall that $z(t,y) = \gamma(t) (y-a(t)) e^{-\frac{i\pi}{8}}$, by  definition in \eqref{eq:z-sigmak-in-def-of-phi-inn-out}, and observe that
\begin{equation*}
    \frac{|U_s''(a_0)|^{1/4}}{\sqrt{2}} \leq \gamma(t)\leq |U_s''(a_0)|^{1/4}\leq  \| U_s^{(3)} \|_{L^1}^{1/4}\qquad 
\end{equation*}
from the assumptions in \eqref{eq:assumptions-for-Tmax}. For any $(t,y) \in [0, T_{\rm max}] \times K$, the first element of the sum in \eqref{eq:Ups-H-formula} is bounded by
\begin{align*}
    \bigg| 
    \frac{ \sqrt[4]{k}\, z(t,y)}{\sqrt{2\pi}} 
    e^{ - \sqrt{k} \frac{z(t,y)^2}{2}} 
    \bigg|
    &= 
    \frac{\sqrt[4]{k}\,\gamma(t) |\,y-a(t)\,| }{\sqrt{2\pi}} 
    e^{ -\sqrt{k} \, \frac{\gamma(t)^2(y-a(t))^2}{2\sqrt{2}}}\\
    &\leq 
    \frac{\sqrt[4]{k}\| U_s^{(3)} \|_{L^1}^{1/4} \,d}{\sqrt{2\pi}}
    e^{    -\sqrt{k} \,\big(\frac{d}{16}\big)^2  \big(\frac{|U_s''(a_0)|}{2} \big)^{1/2}}
    \\
    &\leq 
    \sqrt[4]{k} \, \big( 1 + \| U_s \|_{W^{4,1}} \big)(1 +d\,)^4 
    \Big( 1 + \frac{2}{|U_s''(a_0)|} \Big)
    \frac{1}{d} 
    e^{    -\sqrt{k} \,\big(\frac{d}{16}\big)^2  \big(\frac{|U_s''(a_0)|}{2} \big)^{1/2}},
\end{align*}
where we used \eqref{eq:estimate-y-a(t)-propositionest-2nd-part}. Furthermore, applying \eqref{eq:erf-H-estimate}, we obtain
\begin{align*}
        \frac{1}{2}
        \Big| 
        \,\Big( 1 + \sqrt{k}\,z(t,y)^2 \Big)
    &   \Big( 
            \erf\bigg( \frac{\sqrt[4]{k}\,z(t,y)}{\sqrt{2}}\bigg) - H(y-a(t)) 
        \Big) 
        \Big| 
        \leq
        \\
        &\leq 
        \frac{1}{2}
        \Big(\, 1 + \sqrt{k} \, \gamma(t) (y-a(t))^2 \Big) 
        \frac{13}{|U_s''(a_0)|^\frac{1}{4}} 
        \frac{1}{\sqrt[4]{k}\,d}
        \,
        e^{    -\sqrt{k} \,\big(\frac{d}{16}\big)^2  \big(\frac{|U_s''(a_0)|}{2} \big)^{1/2}}
        \\
        &\leq 
        7 \sqrt[4]{k} \, \big(\, 1 + \| U_s \|_{W^{4,1}} \big)( 1 +  d \,)^4 
        \Big( 1 + \frac{2}{|U_s''(a_0)|} \Big)
        \frac{1}{d}
        e^{    -\sqrt{k} \,\big(\frac{d}{16}\big)^2  \big(\frac{|U_s''(a_0)|}{2} \big)^{1/2}}
\end{align*}
Coupling the last two relations together with \eqref{eq:first-ineq-of-R1k2} yields \eqref{eq:ineq-for-remainders-second-family} for $j = 1$.  We next address the second remainder $\mathcal{R}_{2, k}^2$ of \eqref{eq:second-family-remainders}, whose $L^2$-norm satisfies at any time $t\in [0, T_{\rm max}]$:
\begin{align*}
    \| \mathcal{R}_{2, k}^2(t) \|_{L^2} 
    &
    = 
    3\frac{\gamma(t)^3}{\sqrt[4]{k}} 
    \bigg( 
    \int_K
    \Big|
    \Big[\, 
        \Upsilon'(\sqrt[4]{k}\,z(t,y)) - 2 \sqrt{k} \,z(t,y) H(y-a(t))
    \,\Big] 
    \chi''(y)
    \Big|^2 
    dy
    \bigg)^\frac{1}{2}
    \\
    &\leq 
    \frac{3 \| U_s^{(3)} \|_{L^1}^\frac{3}{4}}{\sqrt[4]{k}} \| \chi \|_{H^3} 
    \big\| \,
        \Upsilon'(\sqrt[4]{k}\,z(t,\cdot)) - 2 \sqrt{k} \,z(t,,\cdot) H(\cdot-a(t))
    \,\big\|_{L^\infty(K)}
\end{align*}
From the derivative of $\Upsilon'$ we obtain with a similar argument as before 
\begin{align*}
    \Big| 
        \Upsilon'(\sqrt[4]{k}\,z(t,y)) 
        &- 2 \sqrt{k} \,z(t,y) H(y-a(t))
    \Big|
    =
    \bigg| 
    \sqrt{\frac{2}{\pi}}e^{ - \sqrt{k} \frac{z(t,y)^2}{2}}  + 
    \sqrt[4]{k}\, z(t,y) \Big( 
            \erf\bigg( \frac{\sqrt[4]{k}\,z(t,y)}{\sqrt{2}}\bigg) - H(y-a(t)) 
        \Big) 
    \bigg|
    \\
    &\leq 
    \sqrt{\frac{2}{\pi}}
    e^{ -\sqrt{k} \, \frac{\gamma(t)^2(y-a(t))^2}{2\sqrt{2}}}
    +
    \sqrt[4]{k}\,\gamma(t) |\, y-a(t) | 
    \Big| 
       \erf\bigg( \frac{\sqrt[4]{k}\,z(t,y)}{\sqrt{2}}\bigg) - H(y-a(t)) 
    \Big|
    \\
    &\leq 
    \sqrt{\frac{2}{\pi}}
    e^{    -\sqrt{k} \,\big(\frac{d}{16}\big)^2  \big(\frac{|U_s''(a_0)|}{2} \big)^{1/2}}
    +
    \sqrt[4]{k}\,|U_s''(a) | d 
    \frac{13}{|U_s''(a_0)|^\frac{1}{4}} 
    \frac{1}{\sqrt[4]{k}\,d}
    \,
    e^{    -\sqrt{k} \,\big(\frac{d}{16}\big)^2  \big(\frac{|U_s''(a_0)|}{2} \big)^{1/2}}
    \\
    &
    \leq 
    13
    \Big( 1 +  \| U_s^{(3)} \|_{L^1} \Big)  e^{    -\sqrt{k} \,\big(\frac{d}{16}\big)^2  \big(\frac{|U_s''(a_0)|}{2} \big)^{1/2}},
\end{align*}
which implies in particular the inequality of \eqref{eq:ineq-for-remainders-second-family} for $j = 2$. Finally, for $j = 3$, we use that $\Upsilon''(\zeta) = 1+\erf ( \zeta/ \sqrt{2}) $ so that we gather
\begin{align*}
    \| \mathcal{R}_{3,k}^2(t)  \|_{L^2}
    &= 
    3
    \beta(t)
    \bigg(
    \int_0^\infty 
    \Big|\,
        \big(  
        \Upsilon''(\sqrt[4]{k}\,z(t,y))-2 H(y-a(t)) 
        \Big)  
        \chi'(y) 
    \,\big|^2
    dy 
    \bigg)^\frac{1}{2}
    \\
    &\leq 
    3
    \sqrt{\frac{|\partial_y u_s(t,a(t))|}{2}}
    \big\|
        \Upsilon''(\sqrt[4]{k}\,z(t,\cdot ))-2 H(\cdot -a(t)) 
    \big\|_{L^\infty(K)} 
     \|  \chi' \|_{L^2}
     \\
     &\leq 
     3 \sqrt{|U_s''(a_0)|}\,
     \max_{y \in K } 
     \bigg| \, 
        \erf\Big( \frac{\sqrt[4]{k} \, z(t,y)}{\sqrt{2}} \Big) - 
        H(y-a(t)) \, 
    \bigg|.
\end{align*}
We hence apply the inequality in \eqref{eq:erf-H-estimate} to obtain
\begin{align*}
    \| \mathcal{R}_{3,k}^2(t)  \|_{L^2}
    &\leq 39  |U_s''(a_0)|^\frac{3}{4}
    \frac{1}{\sqrt[4]{k}\,d}
    \,
    e^{    -\sqrt{k} \,\big(\frac{d}{16}\big)^2  \big(\frac{|U_s''(a_0)|}{2} \big)^{1/2}}
    \\
    &\leq 
    20\Big( 1+ \frac{2}{|U_s''(a_0)|}\Big) 
    |U_s''(a_0)|^\frac{7}{4}
    \frac{1}{d}
    \,
    e^{    -\sqrt{k} \,\big(\frac{d}{16}\big)^2  \big(\frac{|U_s''(a_0)|}{2} \big)^{1/2}}
    \\
    &\leq 
    \frac{k^\frac{3}{4}}{d}
    e^{    -\sqrt{k} \,\big(\frac{d}{16}\big)^2  \big(\frac{|U_s''(a_0)|}{2} \big)^{1/2}}
    24\| \chi \|_{H^3} 
    \Big( 1+ \frac{2}{|U_s''(a_0)|}\Big) 
    \Big( 1+ \|U_s\|_{W^{4,1}} \Big)^2 (1+d)^4, 
\end{align*}
which is indeed \eqref{eq:ineq-for-remainders-second-family} for $j = 3$. This concludes the proof of \Cref{prop:estimates-remainders-and-uf}, Part $(ii)$.
\end{proof}

\subsection{A lower bound for the initial profile.}
We now address \Cref{prop:estimates-remainders-and-uf}, Part~(iii), which provides a lower bound for the $L^2$-norm of the forced solution $u_k^{\rm fr}$. By construction, the cutoff function~$\chi$ is identically zero on the interval $[a_0 + \tfrac{d}{2}, +\infty)$. Consequently, the inner profile~$\phi_{\text{inn}, k}$, as defined in \Cref{def:inner-outer-profiles}, is also zero identically in this region.

\noindent 
As a result, for all $(t,y) \in [0, T_{\rm max}] \times [a_0 + \tfrac{d}{2}, +\infty)$,
\begin{equation*}
     u_k^{\rm fr}(t,y) = \partial_y \phi_{\text{out}, k}(t,y) 
     = \partial_y u_s(t,y) e^{\int_0^t \sigma_k(\tau) d\tau},
\end{equation*}
and the following $L^2$-norm lower bound holds:
\begin{equation*}
    \| u_k^{\rm fr}(t) \|_{L^2(a_0+\frac{d}{2}, +\infty)} 
    \geq \frac{1}{2} \| U_s' \|_{L^2(a_0+\frac{d}{2}, +\infty)} 
    \exp\bigg( \int_0^t {\rm Re} \,\sigma_k(\tau) d\tau \bigg),
\end{equation*}
where we have exploited the assumption on $T_{\rm max} > 0$ given by \eqref{eq:assumption-for-Tmax-and-partialyus}. The result follows by the definition of $\sigma_k(t)$ always in \Cref{def:inner-outer-profiles}, which yields
\begin{align*}
    {\rm Re}\, \sigma_k(t) 
    &= 
    {\rm Re}\, \bigg(  \sqrt{k} \,e^{-\frac{\im \pi}{4}}  \sqrt{\frac{|\partial_y^2u_s(t,a(t))|}{2}} - 
        \im k \, u_s( t,a(t)) \bigg)
    \\
    &= 
     \sqrt{k} \,\frac{\sqrt{|\partial_y^2u_s(t,a(t))|}}{2} 
     \geq \sqrt{k} \, \frac{|U_s''(a_0)|^\frac{1}{2}}{2\sqrt{2}},
\end{align*}
as assumed in \eqref{eq:assumptions-for-Tmax}. This completes the proof of \Cref{prop:estimates-remainders-and-uf}, Part~(iii).

\subsection{An upper bound for the initial profile.} 
We conclude the proof of  \Cref{prop:estimates-remainders-and-uf} with Part $(iv)$, concerning the upper bound of the $H^m_\lambda$-norm of $u_k^{\rm fr}(0,y) = \partial_y \phi_{\text{inn}, k}(0,y) + \partial_y \phi_{\text{out}, k}(0,y)$ in $y\in \mathbb R_+$. 

\smallskip 
\noindent 
The component that grows polynomially in $k\in \mathbb N$ is $ \partial_y \phi_{\text{inn}, k}$ and most of its norm can be estimated using the following inequality for the $n$-th derivative of $\Upsilon$:
\begin{equation}\label{eq:ineq-Ups-n-derivative-prop}
        \max_{a_0-\frac{d}{2}\leq y \leq a_0 + \frac{d}{2}} 
        \big|\, \Upsilon^{(n)}(\sqrt[4]{k}\, z(t,y) )\, \big|
        \leq 
        (n-3)! \, 2\big(\, 1 + |\,U_s''(a_0)\,|^\frac{n-3}{4} d^{n-3}\,\big) k^\frac{n-3}{4},
\end{equation}
for any $n \in \mathbb N$ with $n \geq 3$ and any $t\in [0, T_{\rm max}]$. Indeed, the third derivative of $\Upsilon$ is the Gaussian function $\Upsilon^{(3)}(\zeta) = \sqrt{\tfrac{2}{\pi}}e^{-\frac{\zeta^2}{2}}$. Consequently, any higher-order derivative  $\Upsilon^{(n)}(\zeta)$ with $n> 3$  can be bounded (locally in $\zeta \in \mathbb C$) using the Hermite polynomial ${\rm He}_{n-3}(\zeta)$ of degree $n-3$. The proof of \eqref{eq:ineq-Ups-n-derivative-prop} is postponed to \Cref{lemma:appx-est-Upsilon-nth-derivative} in the Appendix, while we focus on its application to   $\partial_y \phi_{\text{inn}, k}$ at time $t = 0$.

\smallskip\noindent 
Using the \Cref{def:inner-outer-profiles},  we apply the Leibniz rule for the $(j+1)$-th derivative of the product with  $j \in \{0, \dots, m\}$:
\begin{align*}
    \partial_y^{j+1}
    \phi_{\text{inn}, k}(0,y) 
    &=
    \partial_y^{j+1}
    \bigg[ 
        \chi(y) \Upsilon \big( \sqrt[4]{k}\,z(0,y) \big)
    \bigg]
    \\
    &= 
    \sum_{n = 0}^{j+1} 
    \binom{j+1}{n} \chi^{(j+1-n)}(y)
    \,
    \partial_y^{n}
    \bigg[ \Upsilon\Big(\, \sqrt[4]{k}\,z(0,y) \,\Big) \bigg]
    \\
    &= 
    \sum_{n = 0}^{j+1} 
    \binom{j+1}{n} \chi^{(j+1-n)}(y)
    \,
    \Upsilon^{(n)}\big( \sqrt[4]{k}\,z(0,y) \big) 
    \, k^\frac{n}{4}
    \bigg(\frac{|U_s''(a_0)|}{2}\bigg)^\frac{n}{4} e^{-\frac{i n \pi}{8}}
\end{align*}
where we used $z(0,y) = \gamma(0)\, (y-a(0))\,e^{-\frac{i\pi}{8}}= ( |U_s''(a_0)|/2)^{1/4}(y-a_0)\,e^{-\frac{i\pi}{8}}$. For the $n$-th derivatives with $n \leq 2$, we apply the inequality 
\eqref{eq:ineq-in-lemma-us-Upsilon}, while for $n\geq 3$ we use \eqref{eq:ineq-Ups-n-derivative-prop}:
\begin{align*}
    \big| \partial_y^{j+1} \phi_{\text{inn}, k}(0,y) \big|
    \leq 
    \sum_{n = 0}^{\min\{ 2, j+1\}}
    \binom{j+1}{n} \big| \chi^{(j+1-n)}(y)\big| 
    4 \Big( 1 + |U_s''(a_0)|^{\frac{2-n}{4}} d^{2-n} \Big)
    k^{\frac{2-n}{4}}k^\frac{n}{4}\bigg(\frac{|U_s''(a_0)|}{2}\bigg)^\frac{n}{4} + 
    \\
    +
    \sum_{n = 3}^{j+1}
    \binom{j+1}{n}
    |\chi^{(j+1-n)}(y)|
    (n-3)! \, 2\big(\, 1 + |\,U_s''(a_0)\,|^\frac{n-3}{4} d^{n-3}\,\big) k^\frac{n-3}{4} 
    k^\frac{n}{4}\bigg(\frac{|U_s''(a_0)|}{2}\bigg)^\frac{n}{4},
\end{align*}
where the second sum is simply zero, if $j \leq 1$. Hence, for any  $j \in \{0, \dots, m\}$ and any $y \in \mathbb R$, we obtain
\begin{equation*}
    e^{\lambda y}
    \big| \partial_y^{j+1} \phi_{\text{inn}, k}(0,y) \big|
    \leq 
    k^{\frac{2m-1}{4} }
    e^{2 \lambda a_0}
    (m-2)! \,2 ( 1 + d )^m
    \Big( 1 + |U_s''(a_0)| \Big)^\frac{m}{4}
    \sum_{n =0}^{j+1} \binom{j+1}{n} | \chi^{(j+1-n)}(y)|,
\end{equation*}
where we used that $\chi$ has support in $[a_0-\tfrac{d}{2}, \, a_0 + \tfrac{d}{2}]$. This implies that the $H^m_\lambda$-norm of $\partial_y \phi_{\text{inn}, k}$ satisfies 
\begin{align*}
    \| \partial_y \phi_{\text{inn}, k} (0)\|_{H^m_\lambda}
    &=
    \sum_{j = 0}^m \| e^{\lambda y} \partial_y^{j+1} \phi_{\text{inn}, k}(0) \|_{L^2(\mathbb R_+)}
    \\
    &\leq 
    k^{\frac{2m-1}{4} }
    e^{2 \lambda a_0}
    (m-2)! \,2 ( 1 + d )^m
    \Big( 1 + |U_s''(a_0)| \Big)^\frac{m}{4}
    \| \chi \|_{H^m}
    \sum_{j = 0}^m 2^{j+1}
    \\
    &\leq 
    k^{\frac{2m-1}{4} }
    e^{2 \lambda a_0}
    (m-1)! \, ( 1 + d )^m
    \Big( 1 + \|U_s\|_{W^{3,\infty}} \Big)^\frac{m}{4}
    \| \chi \|_{W^{m+1,\infty}}\,d\,
    2^{m+2}
    ,
\end{align*}
which yields
\begin{equation}\label{eq:estimate-from-above-phiinn-in-Hm}
    \| \partial_y \phi_{\text{inn}, k} (0)\|_{H^m_\lambda}
    \leq
    k^{\frac{2m-1}{4} }
    e^{2 \lambda a_0} (m-1)! 
    \,2^{m+2} ( 1 + d )^{m+1}
    \Big( 1 + \|U_s\|_{H^m_\lambda}+\|U_s\|_{W^{3,1}} \Big)^m
    \Big( 1 + \|\chi \|_{W^{m+1,\infty}} \Big)
\end{equation}
We now deal with the $H^m_\lambda$-norm of the outer layer $\partial_y \phi_{\text{out}, k}$ at $t =0$, which is indeed $\mathcal{O}(1)$ in the frequencies $k\in \mathbb N$. Indeed, from \Cref{def:inner-outer-profiles}, we remark that
\begin{align*}
    \partial_y^{j+1} \phi_{\text{inn}, k} (0, y) 
    =
    \partial_y^{j+1}
    \Bigg[ \big( 1- \chi(y) \big) H(y-a_0)
    \bigg( 
        U_s(y)-U_s(a_0) + 
        e^{i\frac{5\pi}{4}}\sqrt{\frac{|U_s''(a_0)|}{2k}}
    \bigg) \Bigg]
    \\
    =
    (1-\chi(y)) H(y-a_0) U_s^{(j+1)}(y)  
    - 
    \sum_{n=1}^{j}
    \binom{j+1}{n}
    \chi^{(j+1-n)}(y)  
    H(y-a_0)
    U_s^{(n)}(y)
    \,+\\
    -\,
    \chi^{(j+1)}(y)H(y-a_0)
    \bigg( 
        U_s(y)-U_s(a_0) + 
        e^{i\frac{5\pi}{4}}\sqrt{\frac{|U_s''(a_0)|}{2k}}
    \bigg).
\end{align*}
This implies that for any $j\in \{ 0, \dots, m\}$
\begin{align*}
    \| e^{\lambda y} \partial_y^{j+1} \phi_{\text{out}, k}(0) \|_{L^2(\mathbb R_+)}
    \leq
    \Big( 1 + \| \chi \|_{W^{m+1,\infty}}\Big) 
    \Big( \| U_s \|_{H^m_\lambda} + \| U_s' \|_{L^1} + \| U_s^{(3)} \|_{L^1}^\frac{1}{2}\Big)
    \sum_{n=0}^{j+1} \binom{j+1}{n},
\end{align*}
which yields  the estimate
\begin{equation*}
    \| e^{\lambda y}  \partial_y^{j+1} \phi_{\text{out}, k} (0) \|_{L^2} 
    \leq 
    2^{j+1}
    \Big( 1 + \| \chi \|_{W^{m+1, \infty}} \Big)
    \Big( 1 + \| U_s \|_{H^{m+1}_\lambda} + \|U_s \|_{W^{3,1}}   \Big)^m.
\end{equation*}
Taking the sum
\begin{align*}
    \| \partial_y \phi_{\text{inn}, k} (0) \|_{H^m_\lambda} 
    &\leq 
    2^{m+2}
    \Big( 1 + \| U_s \|_{H^{m+1}_\lambda} + \|U_s \|_{W^{3,1}}   \Big)^m
    \Big( 1 + \| \chi \|_{W^{m+1, \infty}} \Big)
    \\
    &\leq
    k^{\frac{2m-1}{4} }
    e^{2 \lambda a_0} 
    \,2^{m+2} ( 1 + d )^{m+1}
    \Big( 1 + \|U_s\|_{H^m_\lambda}+\|U_s\|_{W^{3,1}} \Big)^m
    \Big( 1 + \|\chi \|_{W^{m+1,\infty}} \Big).
\end{align*}
Summing the last inequality together with \eqref{eq:estimate-from-above-phiinn-in-Hm}
concludes the proof of \Cref{prop:estimates-remainders-and-uf}.

\section{Correcting the perturbed solution}\label{sec:correcting-the-perturbed-solution}

\noindent 
Combining \Cref{def:inner-outer-profiles}, \Cref{prop:remainder-exact-form} and \Cref{prop:estimates-remainders-and-uf}, we have constructed a forced quasi-eigenmode solution $u(t,x,y) = u_k^{\rm fr}(t,y)e^{i kx}$ to the linearised Prandtl system \eqref{eq:lin-Prandtl}, which grows as $e^{\sigma_0 \sqrt{k} \,t}$ (cf.~\eqref{eq:ForcedSolutionLowerBound}). To extend this growth to an exact solution, we need to subtract a corrector $\delta u_k(t,y)e^{i k x}$ with same forcing term of $u_k^{\rm fr}$ but zero initial data. Naturally, we shall also establish some appropriate upper bounds of $\delta u_k$, as detailed in the next proposition. We recall that, by \Cref{prop:remainder-exact-form}, the forcing term $ f_k$ belongs to $\mathcal{C}^\infty_c([0, T_{\rm max}]\times \mathbb R_+, \mathbb C)$. 
\begin{prop}\label{prop:correction}
For any $k\in \mathbb N$, there exists a unique $\delta u_k : [0, T_{\rm max}] \times \mathbb R_+ \to \mathbb C$ in the function space
\begin{equation}\label{eq:function-space-for-deltauk}
    \delta u_k \in \mathcal{C}([0, T_{\rm max}], H^1_0(\mathbb R_+, \mathbb C)), 
    \quad 
    \text{with}\quad 
    \partial_t \delta u_k,\; \partial_y^2 \delta u_k \in  L^2((0, T_{\rm max}) \times \mathbb R_+, \mathbb C)
\end{equation}
that satisfies both $\delta u_{k }|_{t= 0 }\equiv 0$ as well as, for a.e.~$(t,y) \in (0, T_{\rm max}) \times \mathbb R_+$, the pointwise identity
\begin{equation}\label{eq:equation-for-deltauk}
    \partial_t  \delta u_k (t,y)+ i k\, u_s(t,y) \delta u_k (t,y)  - \partial_y^2  \delta u_k (t,y)  - i k\, \partial_y u_s(t,y) \int_0^y \delta u_k (t,\omega) d\omega = f_k(t,y) \, e^{\int_0^t \sigma_k(\tau) d\tau }.
\end{equation}
Moreover,  the $L^2(\mathbb R_+, \mathbb C)$-norm of $\delta u_k(t,\cdot)$ is bounded at any time $t\in [0, T_{\rm max}]$ by
\begin{equation}\label{eq:ineq-deltauk}
    \|\, \delta u_k(t) \,\|_{L^2} \leq 
    2 \bigg( 1 + k \,\frac{e^{2\lambda^2 t}} { \lambda } \| \,U_s' \,\|_{L^2_\lambda }  \bigg)
    \int_0^ t \big\| \,f_k(\tau  ) e^{\int_0^\tau \sigma_k(l) dl}  \big\|_{L^2}
    e^{ (t-\tau) \sqrt{k} (1+2\|U_s''\|_{L^\infty} )}
    \,d\tau. 
\end{equation}
\end{prop}
\begin{remark}
The equation \eqref{eq:equation-for-deltauk} for $\delta u_k$ is linear and non-homogeneous, with the heat equation as its leading term under Dirichlet boundary conditions. Existence and uniqueness of a solution $\delta u_k$ in \eqref{eq:function-space-for-deltauk} then follow from a standard fixed-point argument, relying also on the a-priori estimates:
\begin{equation}\label{eq:control-of-int0y-deltauk}
\begin{aligned}
    \int_0^\infty \Big| \im k \,\partial_y u_s(t,y) \int_0^y \delta u_k(t, \omega ) d \omega  \Big|^2 dy 
    &\leq 
    k^2 
    \int_0^\infty \big|\,\partial_y u_s(t,y) \,y \,\big|^2 dy \, 
    \|\delta u_k(t) \|_{L^2}^2 
    \\
    &\leq \frac{k^2}{\lambda^2 }
    \int_0^\infty \big|\,\partial_y u_s(t,y)   \big|^2 e^{2\lambda y} dy \,
    \| \delta u_k (t ) \|_{L^2}^2
    \\
    &\leq 
    \frac{k^2 }{\lambda^2 } e^{4\lambda^2 t}\| U_s' \|_{L^2_\lambda}^2 \| \delta u_k (t ) \|_{L^2}^2.
\end{aligned}
\end{equation}
A similar a-priori estimate holds also for the term $i k \,u_s(t,\cdot ) \delta u_k(t,\cdot )$ in $L^2(\mathbb R_+, \mathbb C)$. We shall therefore focus on the less-trivial inequality \eqref{eq:ineq-deltauk}.

\end{remark}
\begin{proof}[Proof of inequality \eqref{eq:ineq-deltauk}]
 We proceed with a similar ansatz as in \cite{MR3925144} and introduce some auxiliary functions for our estimates. We first consider the unique function
\begin{equation*}
    \psi_k \in \mathcal{C}([0, T_{\rm max}], H^1(\mathbb R_+, \mathbb C)),\quad \text{with}\quad 
    \partial_y \psi_k \in L^2(0, T_{\rm max}, H^1_0(\mathbb R_+, \mathbb C))
    \quad \text{and}\quad 
    \partial_t \psi_k  \in L^2( (0, T_{\rm max}) \times \mathbb R_+, \mathbb C),
\end{equation*}
that satisfies both $\psi_{k}|_{t = 0} \equiv 0$ and the pointwise identity, for a.e.~$(t,y) \in (0, T_{\rm max}) \times \mathbb R_+$,
\begin{equation*}
    \partial_t  \psi_k(t,y) + i k\, u_s(t,y)\, \psi_k(t,y)   - \partial_y^2  \psi_k(t,y) 
    +
    ik \,\partial_y u_s(t,y) \int_0^y \psi_k(t,\omega) d\omega  = \delta u_k(t,y).
\end{equation*}
Since $\delta u_k \in L^2((0, T_{\rm max}) \times \mathbb R_+, \mathbb C)$, the existence and uniqueness of $\psi_k$ follow once more from a standard fixed-point argument, now with the heat equation with homogeneous Neumann boundary conditions as the leading operator. We remark in particular that the following identity is satisfied pointwise for a.e.~$(t,y) \in (0, T_{\rm max}) \times \mathbb R_+$:
\begin{align*}
   \Big( \partial_t  +  i k \, u_s(t,y)  - \partial_y^2 \Big) \int_0^y \psi_k(t,\omega) d \omega = 
   \int_0^y \partial_t \psi_k(t, \omega) d\omega + 
   ik \,u_s(t,y)  \int_0^y   \psi_k(t,\omega)   d  \omega - \partial_y \psi_k(t,y) 
   \\
   = 
    \int_0^y \big( \partial_t -\partial_\omega^2 \big) \psi_k (t, \omega) d\omega
    + i k \int_0^y \partial_\omega \Big( u_s(t, \omega) \int_0^\omega \psi_k(t, l) dl \Big) d\omega = 0 .
\end{align*}
Next, we consider a second auxiliary function $\Psi_k\in \mathcal{C}([0, T_{\rm max}], H^1_0(\mathbb R_+, \mathbb C))$, with weak derivatives  $\partial_t \Psi_k, \partial_y^2 \Psi_k$ in  $L^2((0, T_{\rm max}) \times  \mathbb R_+, \mathbb C)$, defined by
\begin{equation}\label{eq:def-Psik}
    \Psi_k (t,y) := \delta u_k(t,y)  - i k\,
    \partial_y u_s(t,y) \int_0^y \psi_k(t, \omega ) d\omega,\qquad \text{for a.e.}\quad (t,y) \in (0, T_{\rm max}) \times \mathbb R_+.
\end{equation}
Indeed, $\delta u_k$ belongs to the same function space as $\Psi_k$, while for the second term of the sum, we exploit the fast decay of $u_s \in \mathcal{C}([0, T_{\rm max}], H^{m+1}_\lambda(\mathbb{R}_+))$, with $m \geq 3$, which also satisfies the heat equation. Moreover, we note that $\Psi_k |_{ t= 0}  \equiv  0 $ and $\Psi_k$ satisfies the following pointwise identity for a.e.~$(t,y)\in (0, T_{\rm max}) \times \mathbb R_+$:
\begin{align*}
    \Big( \partial_t &+ i k \, u_s(t,y) - \partial_y^2 \Big) \Psi_k (t, y) 
    = f_k(t,y) e^{\int_0^t \sigma_k(\tau) d\tau } + 
    i k\, \partial_y u_s(t,y) \int_0^y \delta u_k (t,\omega) d\omega 
    \, +
    \\
    &- i k \underbrace{ \big( \partial_t - \partial_y^2 \big) u_s(t,y)}_{=0} \int_0^y \psi_k(t, \omega ) d\omega
    - \underbrace{ \Big( \partial_t + i k \, u_s(t,y) - \partial_y^2 \Big)\int_0^y \psi_k(t, \omega ) d\omega}_{=0} - 2 i k \,\partial_y u_s(t,y) \psi_k(t,y). 
\end{align*}
In sum, the couple $(\Psi_k, \sqrt{k} \, \psi_k)$ satisfies pointwise, for  a.e.~in $(0, T_{\rm max})\times \mathbb R_+$, the following linear system:
\begin{equation}\label{eq:system-of-psik-and-Psik}
    \Big( \partial_t + i k \, u_s(t,y) - \partial_y^2 \Big)
    \begin{pmatrix}
        \Psi_k(t,y) \\ \sqrt{k} \, \psi_k (t,y)
    \end{pmatrix} 
    = 
    \sqrt{k} 
\begin{pmatrix}
        0 & 2i  \, \partial_y^2 u_s(t,y)  \\ 1 & 0
    \end{pmatrix}
    \begin{pmatrix}
        \Psi_k (t,y)\\ \sqrt{k} \, \psi_k (t,y)
    \end{pmatrix} 
    +
    \begin{pmatrix}
        1 \\ 0 
    \end{pmatrix} f_k(t,y)  e^{\int_0^t \sigma_k(\tau) d\tau }.
\end{equation}
Each summand and derivative in the above identity belongs at least to $L^2((0, T_{\rm max}) \times \mathbb{R}_+)$. We remark, moreover, that for any complex vector $(\xi_1, \xi_2) \in \mathbb C^2$ and any 
$(t,y) \in (0, T_{\rm max})\times \mathbb R_+$
\begin{equation}\label{eq:est-for-Mxi1xi2}
    \mathcal{M}(t,y)
    :=
    \begin{pmatrix}
        0 & 2i  \, \partial_y^2 u_s(t,y)  \\ 1 & 0
    \end{pmatrix} 
    \quad \Longrightarrow \quad 
    \bigg| 
    \begin{pmatrix}
        \overline{\xi_1} \\ \overline{\xi_2}
    \end{pmatrix} \cdot 
    \mathcal{M}(t,y)
    \begin{pmatrix}
        \xi_1 \\  \xi_2
    \end{pmatrix}
    \bigg| 
    \leq 
    \Big( 1 + 2 \| \,U_s ''\, \|_{L^\infty} \Big) 
    \bigg( \frac{|\xi_1|^2}{2} + \frac{|\xi_2|^2}{2} \bigg).
\end{equation}
We hence introduce the absolutely continuous energy functional $\EE_k \in AC([0, T_{\rm max}])$ defined by:
\begin{alignat*}{4}
    \EE_k(t) &:=    \frac{1}{2}\|\, \Psi_k(t ) \,\|_{L^2 }^2 + 
    \frac{\sqrt{k}}{2}
    \| \,\psi_k(t ) \,\|_{L^2 }^2   ,\qquad 
    &&\text{for any } t \in [0, T_{\rm max}],\\
    \EE_k'(t) &=    
    \langle \, \Psi_k(t) ,\, \partial_t \Psi_k(t)  \,\rangle_{L^2 } + 
    \sqrt{k} 
    \langle  \,\psi_k(t) ,\, \partial_t \psi_k(t)  \,\rangle_{L^2 },\qquad 
    &&\text{for a.e.~} t \in [0, T_{\rm max}].
\end{alignat*}
Thanks to the $L^2$-regularity of both derivatives $\partial_t \psi_k, \, \partial_y^2 \psi_k$ and $\partial_t \Psi_k,\, \partial_y^2 \Psi_k$, multiplying the identity \eqref{eq:system-of-psik-and-Psik} by $(\overline{\Psi_k}, \sqrt{k} \, \overline{\psi_k})^T $, integrating over $(0, t) \times \mathbb R_+$ and taking the real part yields the energy identity:
\begin{align*}
    \EE_k&(t) - \EE_k(0)    
    + 
    \int_0^t \|\partial_y \Psi_k(s) \|_{L^2 }^2 ds   + 
    \sqrt{k} 
    \int_0^t 
    \| \partial_y  \psi_k(s) \,\|_{L^2 }^2 ds
    = 
    \\
    &=
    \sqrt{k} \,
    \mathfrak{Re}
    \bigg( 
    \int_0^t 
    \int_{\mathbb R_+}
    \begin{pmatrix}
        \overline{\Psi_k}  \\ \sqrt{k} \, \overline{\psi_k}
    \end{pmatrix}
    \cdot 
    \mathcal{M}  
    \begin{pmatrix}
        \Psi_k \\ \sqrt{k} \, \psi_k
    \end{pmatrix}
    (s,y)
    dyds
    \bigg) 
    +
    \mathfrak{Re}
    \bigg( 
    \int_0^t 
    \int_0^\infty
    \overline{\Psi_k } (s,y)f_k(s,y) e^{\int_0^s \sigma_k(\tau) d \tau} 
    dyds
    \bigg),
\end{align*}
which is satisfied for any $t\in [0, T_{\rm max}]$. In particular, thanks to \eqref{eq:est-for-Mxi1xi2} and the fact that $\EE_k(0) = 0$ by construction, we further obtain
\begin{equation*}
    \max_{0 \leq \tau \leq t } \EE_k(\tau)   
    \leq 
    \sqrt{k}
    \,(1 + 2 \| \,U_s'' \,\|_{L^\infty})
    \int_0^t \EE_k(\tau) d\tau  +  
    \sqrt{2} \, 
    \int_0^t
    \big\| \, f_k(\tau) e^{\int_0^\tau \sigma_k (l) d l } \,\big\|_{L^2} \sqrt{\EE_k(\tau)} d\tau. 
\end{equation*}
for any~$t\in [0, T_{\rm max}]$. We shall now remark that $\max\limits_{0 \leq s \leq t}\EE_k(s) >0 $ for any $t>0$. Indeed, if zero, this would imply that $\Psi_k$ and $\psi_k$ vanish identically on $[0, t] \times \mathbb{R}_+$, which in turn would force both $\delta u_k$ and the forcing term $f_k$ to be identically zero in the same domain (which is a contradiction for $f_k$). For any $t\in ]0, T_{\rm max}]$ we can divide the last inequality by $\max\limits_{0 \leq \tau \leq t}\sqrt{\EE_k(t)} >0 $, to obtain\vspace{-0.2cm}
\begin{equation*}
    \max\limits_{0 \leq \tau \leq t}\sqrt{\EE_k(\tau)}
    \leq 
    \sqrt{k} 
    \Big( \,1 +  2 \| \,U_s''\, \|_{L^\infty} \Big)
    \int_0^t \max_{ 0 \leq l \leq \tau} \sqrt{\EE_k(l)}  d\tau   +  
    \sqrt{2}
    \int_0^t 
    \big\| \, f_k(\tau) e^{\int_0^\tau \sigma_k (l) d l } \, \big\|_{L^2} d\tau.
\end{equation*}
This last relation is satisfied also at $t = 0$, given that $\EE_k(0) = 0$. We apply the Gronwall's inequality to obtain the estimate
\begin{equation}\label{eq:est-final-EEk(t)}
   \sqrt{\EE_k(t)}    
   \leq 
   \sqrt{2} 
   \int_0^t 
   \big\| \,f_k(\tau,\cdot ) e^{\int_0^\tau \sigma_k (l) d l } \,\big\|_{L^2}  e^{ (t-\tau) \sqrt{k} \big(1 + 2\|U_s''\|_{L^\infty} \big)} d\tau .
\end{equation}
for any $t \in [0, T_{\rm max}]$. Finally, recalling the relation \eqref{eq:def-Psik} between $\Psi_k$, $\psi_k$ and $\delta u_k$, we remark that
\begin{equation*}
    \| \,\delta u_k(t) \, \|_{L^2} \leq \| \Psi_k(t) \|_{L^2} + 
    \bigg( 
        \int_0^\infty \Big| i k \,\partial_y u_s(t,y) \int_0^y \psi_k(t, \omega ) d \omega  \Big|^2 dy 
    \bigg)^\frac{1}{2}.
\end{equation*}
A similar argument as the one used in \eqref{eq:control-of-int0y-deltauk}  implies finally that for any $t \in [0, T_{\rm max}]$:
\begin{equation}\label{eq:final-estimate-for-deltauk-in-the-proof}
    \| \,\delta u_k(t) \, \|_{L^2} \leq \| \Psi_k(t) \|_{L^2} + 
     k\frac{e^{2 \lambda^2 t}}{\lambda}  
    \| U_s' \|_{L^2_\lambda} \| \psi_k (t) \|_{L^2}
    \leq 
    \bigg(
        1 + 
        k\frac{e^{2 \lambda^2 t}}{\lambda}  \| U_s' \|_{L^2_\lambda}
    \bigg)
    \sqrt{2} \sqrt{\EE_k(t)} .
\end{equation}
By combining \eqref{eq:est-final-EEk(t)} and \eqref{eq:final-estimate-for-deltauk-in-the-proof}, we obtain \eqref{eq:ineq-deltauk}, thus concluding the proof of the proposition.
\end{proof}
\section{Proof of Theorem \ref{thm:inflating-solutions}}\label{sec:proof-of-first-theorem}
With respect to the previous sections, we are in the position to prove Theorem \ref{thm:inflating-solutions}. The strategy consists of the following: ansatz \eqref{eq:phi-inn-out} gives rise to a forced solution of the linearised Prandtl equations as described in Proposition \ref{prop:remainder-exact-form}. In order to control the exact solution with initial data $u_k(0,\cdot )$, we make use of Proposition \ref{prop:estimates-remainders-and-uf} and \eqref{eq:ineq-deltauk}, which estimates the solution arising from the forcing terms with vanishing initial data $\delta u_k$.

The proof will be a consequence of the following simple lemma:
\begin{lemma} \label{lemma:AuxiliaryEstimateMainTheorem}
    Let  $b_1,b_2>0$ as well as $c_i, ~ i=1,...,4$ be positive constants with $c_1<c_2$ and  $h=h_{m,k}: [0,1] \to \R$ be a continuous function such that
    \begin{align} \label{eq:UglyConstantsEstimate}
         h(t) \geq  b_1 e^{c_1 \sqrt{k}t}- b_2 ke^{c_2 \sqrt{k}t} \int_0^t \left( e^{-c_3\sqrt{k}} + \frac{e^{-\frac{c_4}{s}}}{\sqrt{s}} \right) d s 
    \end{align}
    for all $k\in \NN$ and $m\in \NN_0$ and $t\in [0,1]$. Then there exists a $K=K(b_1,b_2,c_1,c_2,c_3,c_4) \in \NN$ such that for all $k\geq K$, it holds 
    \begin{align}
        h(t) \geq \frac{b_1}{2}  e^{c_1 \sqrt{k}t}
    \end{align}
    for all $t \in [0, \tfrac{\Ccal}{\sqrt[4]k} ]$ where $\Ccal = \frac12 \left( \frac{c_4}{c_2} \right)^{1/2} $.
\end{lemma}

\begin{proof}
    At first,  we apply the mean value theorem in integral form and  simplify \eqref{eq:UglyConstantsEstimate} to
    \begin{align*}
        h(t) 
         & \geq   b_1 e^{c_1\sqrt{k}t}  - \underbrace{b_2kt e^{c_2\sqrt{k}t-\frac{c_4}{t}} }_{=:I}  
         -\underbrace{b_2k t e^{(c_2t-c_3)\sqrt k}}_{:=II} .
    \end{align*}
    For the first term, we point out that $c_2\sqrt{k}t-\frac{c_4}{t}\leq -\frac{\sqrt{c_1c_2}}{\sqrt[4]k}$ is true for $0\leq t \leq  \frac{1}{\sqrt[4]k}\frac12 \left( \frac{c_4}{c_2} \right)^{1/2}$. Therefore, there exists a $K_I\in \NN$ such that
    \begin{align*}
        I&\leq b_2k e^{-\sqrt{c_1c_2}\sqrt[4]{k}} \leq  \frac{b_1}{4}
    \end{align*}
    for all $k\geq K_{I}$ and  $0\leq t \leq \frac12 \left( \frac{c_4}{c_2} \right)^{1/2} \frac{1}{\sqrt[4]{k}} $.
    With respect to the second term, we note that $0\leq t \leq \frac{c_3-1}{c_2}$ implies $c_2t-c_3\leq -1$. Hence there exists $K_{II} \in \NN$ such that 
    \begin{align*}
        II & \leq t b_2 k e^{-\sqrt{k}}  \leq t \frac{b_1 c_1 \sqrt{k} }{4} \leq \frac{b_1}{4} e^{c_1 \sqrt k t}
    \end{align*}
    for all $k\geq K_{II}$ and $0\leq t\leq \frac{c_3-1}{c_2}$. Setting $K := \max\{K_I,K_{II}, \lceil\left( \frac12 \left( \frac{c_4}{c_2} \right)^{1/2}\right)^{-1/4} \rceil\}$, the statement follows.
 \end{proof}

\begin{proof}[Proof of Theorem \ref{thm:inflating-solutions}]
    Let $u_k$ be the solution of \eqref{eq:lin-Prandtl-projected-k} subject to the initial datum $u_k(0, \cdot) =\del_y\phi_{k}(0,\cdot)$ of \eqref{eq:phi-inn-out}. Then it holds $u_k=u_k^{\rm fr}+ \delta u_k$ by linearity. It follows that 
    \begin{align*}
        \no{ u_k(t,\cdot)}{H^{m-1}_\lambda( \R^+)} & \geq \no{ u_k(t,\cdot)}{L^2( \R^+)}\\
        & \geq \no{u_k^{\mathrm{fr}}(t,\cdot)}{L^2(\R^+)} - \no{\delta u_k(t,\cdot)}{L^2(\R^+)} 
    \end{align*}
    Next, we substitute the estimate for the initial datum \eqref{eq:InitialDataUpperBound}, the lower bound on the forced solution \eqref{eq:ForcedSolutionLowerBound} and the estimate for $\delta u_k$ given by \eqref{eq:ineq-deltauk} with respect to the family of remainders derived in Proposition \eqref{prop:remainder-exact-form} to obtain
    \begin{align*}
    \no{u_k(t,\cdot)}{H^{m-1}_\lambda( \R^+)} &\geq  
            \frac{1}{2}
            \| U_s' \|_{L^2(a_0+\frac{d}{2}, \infty)}
            e^{t \sqrt{k} \frac{|U_s''(a_0)|^{1/2}}{2\sqrt{2}}} \\
            &\phantom{=}- 4 \bigg( 1 +  k\frac{e^{2\lambda^2 t}} {\sqrt{\lambda}} \| U_s'' \|_{L^2_\lambda }  \bigg) \times \\
    & \phantom{=}\phantom{=} \times \int_0^ t \no{ \sum_{j=1}^7 \Rcal^1_{j,k}(\tau ,\cdot)+\sum_{j=1}^3 \Rcal^2_{j,k}(\tau,\cdot)\, }{L^2} |e^{\int_0^\tau \sigma_k(l) dl} |
    e^{ (t-\tau) \sqrt{k} (1+2\|U_s''\|_{L^\infty} )}
    \,d\tau.
    \end{align*}
Furthermore, by inserting the estimates for the remainders, \eqref{eq:EstimateRemainderFirstKind} and \eqref{eq:prop-remainders-exp-decay-in-k}, and restricting ourselves to $0\leq t\leq 1$, we have
    \begin{align} \label{eq:EstimateProofMainTheorem}
    \begin{aligned}
    \no{ u_k(t,\cdot)}{H^m_\lambda( \R^+)} &\geq 
            \frac{1}{2}
            \| U_s' \|_{L^2(a_0+\frac{d}{2}, \infty)}
            e^{t \sqrt{k} \frac{|U_s''(a_0)|^{1/2}}{2\sqrt{2}}} \\
            &\phantom{=}- 4 \bigg( 1 +  \frac{e^{2\lambda^2 }} {\sqrt{\lambda}} \| U_s'' \|_{L^2_\lambda }  \bigg)
    \Ccal_\Rcal (1+d^{-1}) k  \, e^{  (1+2\|U_s''\|_{L^\infty} )\sqrt{k} t}
     \\
     &\phantom{===}\times \int_0^ t  \left[ 
            \frac{ e^{ - \frac{d^2}{16 \tau} } }{\sqrt{\tau }}+
           e^{    -\sqrt{k} \,\big(\frac{d}{16}\big)^2  \big(\frac{|U_s''(a_0)|}{2} \big)^{1/2}} \right] \underbrace{|e^{\int_0^\tau \sigma_k(l) dl} |
    e^{-\tau \sqrt{k} (1+2\|U_s''\|_{L^\infty} )}}_{\leq 1}
    \,d\tau.
       \end{aligned}
    \end{align}
    The inequality in the last line follows from the Definition of $\sigma_k$ in \eqref{eq:phi-inn-out}. With \eqref{eq:EstimateProofMainTheorem} at hand, all assumptions of Lemma \ref{lemma:AuxiliaryEstimateMainTheorem} are satisfied. It follows that there exists a $K\in \NN$ which only depend on $m, U_s, d,\lambda$ and $\Ccal_\Rcal$ such that the solution $ u_k$ satisfies
    \begin{align*}
        \no{ u_k(t,\cdot)}{H^{m-1}_\lambda(\R^+)}\geq \frac{\no{U_s'}{L^2(a_0+\tfrac{d}{2},\infty)}}{2} \exp \left( \frac12\sqrt{k\frac{|U''(a_0)|}{2}} t \right) 
    \end{align*}
    for all $k\geq K$ and all $0\leq t\leq \frac{\Ccal}{\sqrt[4]{k}}$ with $\Ccal = \frac{d}{16(1+\no{U''_s}{L^\infty} )}$.  This proves the claim.
\end{proof}

\section{Proof of Theorem \ref{thm:non-existence-of-solutions}} \label{sec:proof-of-second-main-thm}
This section is dedicated to the proof of Theorem \ref{thm:non-existence-of-solutions}. We show that for the specific choice of initial data  $u_{\rm in}$ as in
\eqref{eq:intro-Uincomplex-series},
\begin{equation*}
        \mathbf{U}_{\rm in}(x,y) := \sum_{k = 1}^\infty e^{-\sigma_0 \sqrt[4]{k}} \frac{d}{dy} \phi_{{\rm unst,k}} (y) e^{ \im k x}  \in \mathbb C,\quad (x,y) \in \mathbb T \times \mathbb R_+,    
\end{equation*}
there cannot exist a corresponding weak solution.
\textit{A contrario}, suppose there existed a weak solution $u$ in the sense of Definition \ref{def:weak-solution-Prandtl} with initial datum either $u_{\rm in } = \Re (\mathbf{U}_{\rm in})$ or $u_{\rm in } = \Im (\mathbf{U}_{\rm in})$ where $\mathbf{U}_{\rm in}=\mathbf{U}_{\rm in}(x,y)$ is defined in \eqref{eq:intro-Uincomplex-series}. These solutions, in particular, satisfy $u \in L^\infty(0,T;L^2(\mathbb T\times \R_+))$ for some $\delta>0$. Furthermore, there exists a Fourier expansion
\begin{align*}
    u(t,x,y) = \sum_{k\in \ZZ} u_k(t,y) e^{\im k x}
\end{align*}
which holds for almost every $t\in (0,\delta)$. Consider a  $t_0\in (0,\delta)$ such that the Fourier series converges. Since $u$ is a weak solution to the linearised Prandtl equations, the coefficients solve \eqref{eq:lin-Prandtl-projected-k} for every $k\in \ZZ$ with initial value $u_k(0,y) = e^{-\sigma_0 \sqrt[4]{k}}\frac{d}{dy} \phi_{\mathrm{unst},k}(y)$. The solutions $(u_k)_k$ are uniquely determined and by Theorem \ref{thm:inflating-solutions} satisfy
\begin{align*}
    \no{u_k(t,\cdot)}{L^2(\R^+)}\geq \Ccal e^{-\sigma_0 \sqrt[4] k + \frac12\sqrt{k\frac{|U''_s(a_0)|}{2}} ~t} 
\end{align*}
for every $k\geq K$, where $K\in \NN$ is fixed, and for every  $0\leq t \leq \tfrac{d}{16(1+\no{U''}{L^\infty})}  \frac{1}{\sqrt[4] k}$ with $\Ccal = \frac{1}{4} 
        \bigg( 
        \int_{a_0+\frac{d}{2}}^\infty | \, U_s'(y)\,|^2 dy \,\bigg)^\frac{1}{2}$.
We set $k_{t_0} := \left\lfloor \left( \frac{d}{16(1+\no{U''_s}{L^\infty} )} \cdot \frac{1}{t_0}\right)^4 \right\rfloor$. 
By Parseval's identity, it follows
\begin{align*}
    \no{u(t_0,\cdot)}{L^2(\mathbb T \times \R^+)} & \geq  \no{u_k(t_0,\cdot)}{L^2(\R^+)} \\
    &\geq \Ccal e^{- \sigma_0 \sqrt[4] k_{t_0} + \tfrac12 \sqrt{k_{t_0}\frac{|U''_s(a_0)|}{2}} ~t_0} \\
    & \geq \Ccal \exp \left( \sqrt[4]{k_{t_0}} \left[ \frac{d\sqrt{|U_s''(a_0)}}{\sqrt 2 \cdot 16 (1+ \no{U''_s}{L^\infty})}- \sigma_0\right] \right)
\end{align*}
Now, if $t_0 \to 0^+$ and if $\sigma_0  <\frac{d|U''(a_0)|^{1/2}}{32(1+\no{U''_s}{L^\infty})} < \frac{d\sqrt{|U_s''(a_0)}}{\sqrt 2 \cdot 16 (1+ \no{U''_s}{L^\infty})} $, this yields a contradiction to $u\in L^\infty(0,T;L^2(\mathbb T\times \R_+)) $. Therefore, the claim is proved.

\appendix

\section{A collection of auxiliary lemmata}\label{appx:lemma:estimates-for-us}

\noindent
This section is devoted to some technical lemmas used in the previous sections. The results depend on the notation and assumptions introduced throughout the paper, which we recall at the relevant steps.
\begin{lemma}\label{lemma:estimates-for-us-appx}
For any time $t \in [0, T_{\rm max}]$, any index $n \in \{0,1,2\}$ and any frequency $k \in \mathbb N$:
\begin{align}
        &
        \max_{a_0-\frac{d}{2}\leq y \leq  a_0+ \frac{d}{2}}
        \bigg|
            \partial_y^n
            \Big( 
                u_s(t,y) 
                -  
                \alpha(t)
                - \beta(t)  (y-a(t))^2 
            \Big)
        \bigg|
        \leq 
        d^{3-n}
        \big\|  U_s^{(3)}\big\|_{L^1}
        \frac{e^{ - \frac{d^2}{16 t} }}{\sqrt{\pi t}}
        ,\label{eq:ineq-in-lemma-us-appx}
\\
       &
       %
       |a'(t)| 
       \leq 
       \frac{2}{|U_s''(a_0)|}
       \| U_s^{(3)} \|_{L^1(\mathbb R_+)}
       \frac{e^{ - \frac{d^2}{16 t} }}{\sqrt{\pi t}},
       \qquad
       |\gamma'(t)|
       \leq 
       \frac{2}{|U_s''(a_0)|^\frac{3}{4}}
       \| U_s^{(3)} \|_{W^{1,1}}
       \frac{e^{ - \frac{d^2}{16 t} }}{\sqrt{\pi t}}
       ,\label{eq:ineq-in-lemma-us-a'-gamma'-appx}
\\
       &
       %
       \max_{a_0-\frac{d}{2}\leq y \leq  a_0+ \frac{d}{2}}
       \Big| \Upsilon^{(n)}\big(\sqrt[4]{k}\,z (t,y) \big) \Big|
       \leq 
       4 \big(  1+ |U_s''(a_0)|^{\frac{2-n}{4}} d^{2-n}   \big) k^\frac{2-n}{4}.\label{eq:ineq-in-lemma-us-Upsilon-appx}
    \end{align}
\end{lemma}
\begin{proof}
We begin by addressing \eqref{eq:ineq-in-lemma-us-appx}. Since $u_s \in \mathcal{C}^\infty(\big[a_0-\frac{d}{2}, \, a_0+ \frac{d}{2}\big] \times \mathbb R_+)$, invoking the Taylor theorem, there exists $\xi_n(t,y)\in \big[a_0-\frac{d}{2}, \, a_0+ \frac{d}{2}\big]$ between $y$ and $a(t)$ such that
    \begin{equation*}
    \begin{aligned}
      \partial_y^n 
      \Big[ 
        u_s(t,y) -  
        u_s(t,a(t)) - 
        \frac{\partial_y^2 u_s(t,a(t))}{2} &(y-a(t))^2
      \Big]
      = 
      \frac{\partial_y^{3} u_s(t,\xi_n(t,y))}{(3-n)!} (y-a(t))^{3-n}
       \\
       &= 
       \frac{(y-a(t))^{3-n}}{(3-n)!}
       \frac{1}{\sqrt{4\pi t}}
       \int_0^\infty 
       \bigg( 
            e^{-\frac{|\xi_n(t,y) - \omega|^2}{4t}}
            -
            e^{-\frac{|\xi_n(t,y) + \omega|^2}{4t}}
       \bigg)
       U_s^{(3)}(\omega) d\omega,
    \end{aligned}
    \end{equation*}
    using Green's formula of the heat kernel on the half-line with Dirichlet conditions. From the quadratic assumption,  $U_s^{(3)}\equiv 0$ in $\big[a_0-d, \, a_0+ d\big]$, thus the last integral is on $\mathbb R_+\setminus  \big[a_0-d, \, a_0+ d\big]$. 
    Inequality \eqref{eq:ineq-in-lemma-us-appx} follows thus from
    \begin{equation*}
        |y-a(t)| \leq d,\qquad 
        |\xi_n(t,y) - \omega|\geq  \frac{d}{2},
        \qquad 
        |\xi_n(t,y) + \omega| = \xi_n(t,y)+\omega \geq  \xi_n(t,y)\geq a_0 -\frac{d}{2} \geq \frac{d}{2},
    \end{equation*}
    for any  $\omega \in \mathbb R_+ \setminus [a_0-d, \, a_0+ d]$ and any $t \in [0, T_{\rm max}]$.

    \smallskip
    \noindent 
    We next address \eqref{eq:ineq-in-lemma-us-a'-gamma'-appx} and recall that $\partial_y^2 u_s(t,a(t))+ \partial_y^3 u_s(t,a(t)) = 0$ is for any $t\in [0, T_{\rm max}]$. Hence,
    \begin{equation*}
    \begin{aligned}
      |a'(t)| 
      = \left|\frac{\partial_y^3 u_s(t,a(t))}{\partial_y^2 u_s(t,a(t))} \right|
      &= 
      \frac{1}{|\partial_y^2 u_s(t,a(t))|}
      \frac{1}{\sqrt{4\pi t}}
      \left|
        \int_0^\infty 
       \bigg( 
            e^{-\frac{|a(t) - \omega|^2}{4t}}
            -
            e^{-\frac{|a(t) + \omega|^2}{4t}}
       \bigg)
       U_s^{(3)}(\omega) d\omega
      \right|.
    \end{aligned}
    \end{equation*}
    The first inequality of \eqref{eq:ineq-in-lemma-us-a'-gamma'-appx} follows from $|\partial_y^2 u_s(t,a(t))| > \frac{|U_s''(a_0)|}{2}$ and, once more, $|a(t) - \omega| \geq \frac{d}{2}$ and $|a(t) + \omega| \geq a(t) \geq a_0 - \frac{d}{2}\geq \frac{d}{2}$, for any $t \in [0, T_{\max}]$ and $\omega \in \mathbb R_+ \setminus [a_0-\frac{d}{2}, \, a_0+ \frac{d}{2}]$. For what concerns $\gamma'(t)$ in  \eqref{eq:ineq-in-lemma-us-a'-gamma'-appx}, we use the fact that $u_s$ satisfies the heat equation and therefore $\partial_t \partial_y^2 u_s = \partial_y^4 u_s$. We obtain
    \begin{equation*}
        \gamma'(t) = \frac{d}{dt}
        \Big[
        \Big( -\frac{\partial_y^2 u_s(t,a(t))}{2}\Big)^\frac{1}{4}
        \Big]
        = 
        \Big( -\frac{\partial_y^2 u_s(t,a(t))}{2}\Big)^{-\frac{3}{4}}
        \Big(  \partial_y^4 u_s(t, a(t)) + \partial_y^3u_s (t, a(t)) \Big).
    \end{equation*}
    The second inequality in \eqref{eq:ineq-in-lemma-us-a'-gamma'-appx} follows by applying Green’s formula to the resulting expression and using estimates analogous to those before.

    \smallskip
    \noindent 
    We finally address \eqref{eq:ineq-in-lemma-us-Upsilon-appx}. Recalling that  $z(t,y)= \gamma(t) (y-a(t))e^{-\frac{i\pi}{8}}$, we first remark that by substitution:
    \begin{align*}
        \frac{1}{2}
        \bigg| 
            1 + \erf\bigg( \frac{\sqrt[4]{k}\, z(t,y)}{\sqrt{2}} \bigg) 
        \bigg|
        = 
        \frac{1}{\sqrt{2\pi}}
        \bigg|
            \int_{-\infty}^{\sqrt[4]{k}\,\gamma(t) (y-a(t))}
            \exp\left( -e^{-\frac{i \pi}{4}}\frac{\omega^2}{2}\right)d\omega \,
            e^{-\frac{ i \pi}{8}}
        \bigg|
        \leq 
        \frac{1}{\sqrt{2\pi}}
        \int_{-\infty}^\infty e^{ - \frac{\omega^2}{2\sqrt{2}}}d\omega 
        = \sqrt[4]{2}.
    \end{align*}
    Hence, using the definition of $\Upsilon$ in \eqref{def:Upsilon-fct}, we find that for $n = 0$:
    \begin{equation*}
    \begin{aligned}
        |\Upsilon(&\sqrt[4]{k}\, z(t,y) ) | 
        =
        \bigg|\;
                \frac{1}{\sqrt{2\pi}}\,\zeta\, e^{-\frac{\zeta^2}{2}} + 
                (1+\zeta^2)  
                \frac{\erf\left( \zeta/\sqrt{2} \right) + 1}{2} 
        \;\bigg|_{\zeta =\sqrt[4]{k}\, z(t,y) = \sqrt[4]{k}\,\gamma(t) (y-a(t)) e^{-\frac{i\pi}{8}}}
        \\
        &\leq 
        \frac{1}{\sqrt{2\pi}} 
        \Big( 
            \sqrt[4]{k}\,\gamma(t) |\,y-a(t)| \,
        \Big)
        e^{- \frac{1}{2\sqrt{2}} \big(\sqrt[4]{k}\,\gamma(t) |y-a(t)|\,\big)^2}
        +
        \Big(
            1+ \sqrt{k} \,\gamma(t)^2 (y-a(t))^2
        \Big)
        \frac{1}{2}
        \bigg| 
            1 + \erf\bigg( \frac{\sqrt[4]{k}\, z(t,y)}{\sqrt{2}} \bigg) 
        \bigg|
        \\
        &\leq \frac{\sqrt[4]{2}}{\sqrt{2\pi e}} +
        \sqrt[4]{2}+ 
        \sqrt[4]{2}
        \sqrt{k} \,\gamma(t)(y-a(t))^2 
        \leq 
        1  +   |U_s''(a_0)|^\frac{1}{2} d^2  \sqrt{k} 
        \leq 
        4
        \big( 1  +   |U_s''(a_0)|^\frac{2-0}{4} d^{2-0} \big) k^\frac{2-0}{4} ,
    \end{aligned}
    \end{equation*}
    where we also used $\max\limits_{\mathrm{x}\in \mathbb R} \big|\mathrm{x} \, e^{-\mathrm{x}^2/(2\sqrt{2})}\big|=\sqrt[4]{2}/\sqrt{e}$ and $|\gamma(t)|\leq |U_s''(a_0)|^{1/4}$. For $n = 1$, we proceed similarly:
    \begin{equation*}
    \begin{aligned}
        |\Upsilon'(\sqrt[4]{k}\, z(t,y) ) | 
        &= 
        \bigg|
            \sqrt{\frac{2}{\pi}}e^{-\frac{\zeta^2}{\sqrt{2}}}
            +
            \zeta 
            \Big(
                1 + \erf\Big( \frac{\zeta}{2}\Big)
            \Big)
        \bigg|_{\zeta = \sqrt[4]{k}\,\gamma(t) (y-a(t)) e^{-\frac{i\pi}{8}}}
        \leq 
        \sqrt{\frac{2}{\pi}} + \sqrt[4]{k}\,\gamma(t) |y-a(t)|\,2 \sqrt[4]{2}
        \\
        &\leq \sqrt{\frac{2}{\pi}}
        +\sqrt[4]{k} |U_s''(a_0)|^{\frac{1}{4}}d 
        \,2 \sqrt[4]{2}
        \leq
        1 + 4 |U_s''(a_0)|^{\frac{1}{4}}d\sqrt[4]{k}
        \leq 
        4 \big( 1 + |U_s''(a_0)|^{\frac{2-1}{4}}d^{2-1} \big)k^\frac{2-1}{4}.
    \end{aligned}
    \end{equation*}
    The final case $n = 2$ is straightforward, since we simply obtain:
    \begin{equation*}
    \begin{aligned}
        |\Upsilon''(\sqrt[4]{k}\, z(t,y) ) | 
        =
        \bigg| 
            1 + \erf\bigg( \frac{\sqrt[4]{k}\, z(t,y)}{\sqrt{2}} \bigg) 
        \bigg|
        \leq 2\sqrt[4]{2} \leq 4 \big( 1 + |U_s''(a_0)|^{\frac{2-2}{4}}d^{2-2} \big)k^\frac{2-2}{4}.
    \end{aligned}
    \end{equation*}
    This concludes the proof of \Cref{lemma:estimates-for-us-appx}.
\end{proof}
\begin{lemma}\label{lemma:erf-H}
    For any time $t\in [0, T_{\rm max}]$ and any $y \in \mathbb R_+ \setminus [a_0 - \tfrac{d}{4}, a_0+\tfrac{d}{4}]$, the following inequality holds true:
    \begin{equation}\label{eq:erf-H-estimate-appx}
    \bigg| 
        \,
        \erf\bigg( \frac{\sqrt[4]{k}\,z(t,y)}{\sqrt{2}}\bigg) - H(y-a(t)) 
        \,
    \bigg| 
    \leq     
    \frac{13}{|U_s''(a_0)|^\frac{1}{4}} \frac{1}{\sqrt[4]{k}\,d}
    e^{    -\sqrt{k} \,\big(\frac{d}{16}\big)^2  \big(\frac{|U_s''(a_0)|}{2} \big)^{1/2}}.
\end{equation}
\end{lemma}
\begin{proof}
    We first address the case of $y > a_0+\tfrac{d}{4}$, which implies $ y-a(t) >\frac{d}{8}>0$, since $a(t) \in [a_0-\tfrac{d}{8}, a_0 + \tfrac{d}{8}]$ for any $t\in [0, T_{\rm max}]$. In particular $H(y-a(t)) = 1$, for $y > a_0+\tfrac{d}{4}$.
    
    \noindent 
    We write the $\erf$-function as a path-integral in the complex plane with respect to its derivative $\erf'(\zeta) = \frac{2}{\sqrt{\pi}}e^{\zeta^2}$. Since 
    \begin{equation*}
        \frac{\sqrt[4]{k}\,z(t,y)}{\sqrt{2}}
        = 
        \frac{\sqrt[4]{k}}{\sqrt{2}} \, (y-a(t)) \, \gamma(t) e^{-\frac{i \pi }{8}}\in \mathbb C,\quad  
        \text{with}
        \quad 
        \gamma(t)= \bigg( \frac{|\partial_y^2 u_s(t,a(t))|}{2}\bigg)^\frac{1}{4}
        \geq
        \bigg( \frac{|U_s''(a_0)|}{4}\bigg)^\frac{1}{4}
        =
        \frac{|U_s''(a_0)|^\frac{1}{4}}{\sqrt{2}}
        >0,
    \end{equation*}
    we choose the complex line path $\Gamma$ given by $ \Gamma(\omega) =  \omega  e^{-\frac{i \pi}{8}}$ with  $\omega \in [0, \frac{\sqrt[4]{k}}{\sqrt{2}} \gamma(t)(y-a(t))]\subset \mathbb R_+$. Hence, also using $\erf(0) = 0$, the following identity holds true:
    \begin{equation}\label{eq:lemma-erf-H-erf-part}
    \begin{aligned}
         \erf\bigg( \frac{\sqrt[4]{k}\,z(t,y)}{\sqrt{2}}\bigg) 
         &= 
         \frac{2}{\sqrt{\pi}}
         \int_\Gamma e^{-\zeta^2} d\zeta 
         = 
         \frac{2}{\sqrt{\pi}}
         \int_0^{\frac{\sqrt[4]{k}}{\sqrt{2}} \gamma(t) (y-a(t))} 
         e^{-\Gamma(\omega)^2} \Gamma'(\omega) d\omega 
         \\
         &= 
         \frac{2}{\sqrt{\pi}} \int_0^{\frac{\sqrt[4]{k}}{\sqrt{2}} \gamma(t) (y-a(t))} \exp\big( -e^{-\frac{i\pi}{4}}\omega^2\big) d \omega\, e^{-\frac{i\pi}{8}}.
    \end{aligned}
    \end{equation}
    Next, we apply Cauchy’s theorem to shift the half-line Gaussian integral to an integral along the complex half-line at an angle of $e^{-\frac{i\pi}{8}}$ (where the integrand still retains exponential decay):
    \begin{equation}\label{eq:lemma-erf-H-1-is-int-et2}
        1 = \frac{2}{\sqrt{\pi}} \int_0^\infty e^{-l^2} dl   = 
        \frac{2}{\sqrt{\pi}}
        \int_0^\infty \exp\big( -e^{-\frac{i\pi}{4}}\omega^2\big) d \omega\, e^{-\frac{i\pi}{8}}.
    \end{equation}
    Here, the contribution from the circular arc at infinity (required to close the contour) is zero due to the exponential decay of $e^{-\zeta^2}$ in the sector $\arg(\zeta)< \frac{\pi}{4}$. Combining \eqref{eq:lemma-erf-H-erf-part} with \eqref{eq:lemma-erf-H-1-is-int-et2} implies
    \begin{equation}\label{eq:lemma-erf-H-to-cite-at-the-end}
    \begin{aligned}
      \bigg| 
          \,\erf\bigg( \frac{\sqrt[4]{k}\,z(t,y)}{\sqrt{2}}\bigg) - H(y-a(t))
        \, 
       \bigg|
       &=
       \frac{2}{\sqrt{\pi}} 
        \bigg|
           \int_{\frac{\sqrt[4]{k}}{\sqrt{2}} \gamma(t) (y-a(t))}^\infty \exp\big( -e^{-\frac{i\pi}{4}}\omega^2\big) d \omega\, e^{-\frac{i\pi}{8}}
        \bigg|
        \\
        &\leq 
        \frac{2}{\sqrt{\pi}} 
        \int_{\frac{\sqrt[4]{k}}{\sqrt{2}} \gamma(t) (y-a(t))}^\infty e^{-\frac{\omega^2}{\sqrt{2}}} d\omega
        = 
        \sqrt{\frac{2}{\pi}} 
        \int_{\frac{\sqrt[4]{k}}{\sqrt{2}} \gamma(t) (y-a(t))}^\infty 
        \frac{1}{\omega}
        \frac{d}{d\omega}
        \bigg[
         - e^{-\frac{\omega^2}{\sqrt{2}}} 
        \bigg]
        d\omega,
    \end{aligned}
    \end{equation}
    where all terms in the last integrand are positive. Since $\omega \geq \frac{\sqrt[4]{k}}{\sqrt{2}} \gamma(t) (y-a(t))$, the following upper bound holds true: 
    \begin{equation*}
        \frac{1}{\omega} \leq 
        \frac{\sqrt{2}}{\sqrt[4]{k}\gamma(t) (y-a(t))}
        \leq  
        \frac{16}{|U_s''(a_0)|^\frac{1}{4}\sqrt[4]{k} \,d}.
    \end{equation*}
    Remarking that $16\sqrt{2}/\sqrt{\pi} \leq 13 $, we eventually obtain
    \begin{align*}
        \bigg| 
          \,\erf\bigg( \frac{\sqrt[4]{k}\,z(t,y)}{\sqrt{2}}\bigg) - H(y-a(t))
        \, 
       \bigg| 
       &\leq 
       \frac{13}{|U_s''(a_0)|^\frac{1}{4}\sqrt[4]{k} \,d}
       \int_{\frac{\sqrt[4]{k}}{\sqrt{2}} \gamma(t) (y-a(t))}^\infty 
        \frac{d}{d\omega}
        \bigg[
         - e^{-\frac{\omega^2}{\sqrt{2}}} 
        \bigg]
        d\omega
        \\
        &\leq 
        \frac{13}{|U_s''(a_0)|^\frac{1}{4}\sqrt[4]{k} \,d}
        \exp
    \bigg(
        -\sqrt{k} \,\gamma(t)^2 \frac{(y-a(t))^2}{2\sqrt{2}} 
    \bigg).
    \end{align*}
    Hence, when $y \geq a_0+\tfrac{d}{2}$, the estimate \eqref{eq:erf-H-estimate-appx} follows from the fact that $\gamma(t)^2>|U_s''(a_0)|^{1/2}/\sqrt{2}$ and $|y-a(t)|\geq d/8$.

    \smallskip
    \noindent 
    For the case $y < a_0 - \frac{d}{4}$, the argument is analogous. Here, $y-a(t) < 0$ and thus $H(y-a(t)) = -1$. The integral representation becomes 
    \begin{equation*}
        \bigg| 
          \erf\bigg( \frac{\sqrt[4]{k}\,z(t,y)}{\sqrt{2}}\bigg) - H(y-a(t))
       \bigg|
       =
       \frac{2}{\sqrt{\pi}} 
        \bigg|
           \int^{\frac{\sqrt[4]{k}}{\sqrt{2}} \gamma(t) (y-a(t))}_{-\infty} e^{ -e^{-\frac{i\pi}{4}}\omega^2 }d \omega
        \bigg|
        =
       \frac{2}{\sqrt{\pi}} 
        \bigg|
           \int_{\frac{\sqrt[4]{k}}{\sqrt{2}} \gamma(t) (a(t)-y)}^{\infty} e^{ -e^{-\frac{i\pi}{4}}\omega^2 }d \omega
        \bigg|.
    \end{equation*}
    Since this expression has the same form as in \eqref{eq:lemma-erf-H-to-cite-at-the-end}, the same procedure applies, thereby completing the proof.
\end{proof}

\begin{lemma}\label{lemma:appx-est-Upsilon-nth-derivative}
    For any $n \in \mathbb N$ with $n \geq 3$ and any $t\in [0, T_{\rm max}]$, the following inequality holds true:
    \begin{equation*}
        \max_{a_0-\frac{d}{2}\leq y \leq a_0 + \frac{d}{2}} 
        \big|\, \Upsilon^{(n)}(\sqrt[4]{k}\, z(t,y) )\, \big|
        \leq 
        (n-3)! \, 2\big(\, 1 + |\,U_s''(a_0)\,|^\frac{n-3}{4} d^{n-3}\,\big) k^\frac{n-3}{4}.
    \end{equation*}
\end{lemma}
\begin{proof}
   By definition, the third derivative of $\Upsilon$ is given by $\Upsilon^{(3)}(\zeta) = \sqrt{ \frac{2}{\pi}} \, e^{-\frac{\zeta^2}{2}}$. Therefore, for $n\geq 3$ and $\zeta \in \mathbb C$, the $n$-th derivative of $\Upsilon$ can be expressed as:
    \begin{equation}\label{eq:appx-Upsilon^(n)}
         \Upsilon^{(n)}(\zeta ) 
         = 
         \sqrt{ \frac{2}{\pi}} \frac{d^{n-3}}{d \zeta^{n-3}}
         \Big[ e^{-\frac{\zeta^2}{2}} \Big] 
         = 
         \sqrt{ \frac{2}{\pi}}
         \,\,
         \text{He}_{n-3} 
         \big( \zeta \big) 
         \,
         e^{-\frac{\zeta^2}{2}},
    \end{equation}
    where $\text{He}_{n-3}$ is the probabilistic Hermite polynomial of degree $n-3\in \mathbb N$, satisfying
    \begin{equation*}
        \text{He}_{n-3}(\zeta) = (n-3)! \sum_{m=0}^{\lfloor \frac{n-3}{2} \rfloor} \frac{(-1)^m}{m!(n-3-2m)!}\frac{\zeta^{n-3-2m}}{2^m}. 
    \end{equation*}
    We use the following elementary estimate for $\zeta \in \mathbb C$:
    \begin{equation*}
       | \,\text{He}_{n-3}(\zeta) \,|
       \leq (n-3)! \big( 1 + |\zeta|^{n-3} \big)
       \sum_{m = 0}^\infty \frac{1}{2^m}
       = (n-3)!\, 2 \big( 1 + |\zeta|^{n-3} \big)
    \end{equation*}
    We set $\zeta =\sqrt[4]{k}\,z(t,y) = \sqrt[4]{k} \gamma(t) (y-a(t)) e^{-\frac{i\pi}{8}}$ with $\gamma(t) = |\partial_y^2u_s(t,a(t))/2|^{1/4}$ and remark that the exponential term satisfies
    \begin{equation*}
       \Big| \exp \Big( -\sqrt{k} \,\frac{\gamma(t)^2 (y-a(t))^2 e^{-\frac{i\pi}{4}}}{2} \Big) \Big|
       = 
       \exp \Big( -\sqrt{k} \,\frac{\gamma(t)^2 (y-a(t))^2 }{2\sqrt{2}} \Big) 
       \leq 1.
    \end{equation*}
    Additionally, from the assumptions in \eqref{eq:assumptions-for-Tmax}, we have $\gamma(t) \leq |U_s''(a_0)|^{1/4}$ and $a(t) \in [a_0-\tfrac{d}{8},\, a_0+\tfrac{d}{8}]$. Thus, since $|y-a(t)|\leq d$ for any $y \in [a_0-\tfrac{d}{2},\, a_0+\tfrac{d}{2}]$, it follows that
    \begin{align*}
        \max_{a_0-\frac{d}{2}\leq y \leq a_0 + \frac{d}{2}} 
        \big| 
            \,\text{He}_{n-3}\big(\sqrt[4]{k}\, z(t,y) \big)\,
        \big|
        &\leq 
        (n-3)!
        \,2
        \max_{a_0-\frac{d}{2}\leq y \leq a_0 + \frac{d}{2}} 
        \Big( 1 + \gamma(t)^{n-3} |y-a(t)|^{n-3} k^\frac{n-3}{4}  \,\Big)\\
        &\leq 
        (n-3)! \, 2\big(\, 1 + |\,U_s''(a_0)\,|^\frac{n-3}{4} d^{n-3}\,\big) k^\frac{n-3}{4}.
    \end{align*}
    Combining this with the expression for $\Upsilon^{(n)}$ in \eqref{eq:appx-Upsilon^(n)} and noting that $\sqrt{2/\pi}\leq 1$, completes the proof of the lemma.
\end{proof}

\section{Basic estimates for the heat equation}
\noindent 
In this section we provide some elementary estimates  in weighted Sobolev spaces of the heat equation in the half-line. For general initial data $U_s \in L^2(\mathbb R_+)$, the unique solution $u_s \in \mathcal{C}(\mathbb R_+, L^2(\mathbb R_+))\cap L^2(\mathbb R_+, H^1_0(\mathbb R_+))$ of System \eqref{eq:heat-equation} is given by the Dirichlet Green function:
\begin{equation}\label{appx-eq:green-formula}
    u_s(t,y) = 
    \int_0^{\infty}
    G_D(t,y, z) \,U_s (z) \,d z,
    \qquad 
    G_D(t,y,z) = 
    \frac{1}{\sqrt{4\pi t}}\Big( e^{-\frac{|y-z|^2}{4t}} - e^{-\frac{|y+z|^2}{4t}} \Big),
\end{equation}
for any $t,\,y >0$. The following lemma concerns the propagation of the $H^m_\lambda(\mathbb R_+)$-norm.
\begin{lemma}\label{lemma:Green-formula-in-Hnmu}
    Let $m \in \mathbb N_0$, $\lambda > 0$ and $U_s \in H^m_\lambda(\mathbb R_+)$. If $m\geq 1$, assume that $U_s(0) = \dots = U_s^{2\lfloor \frac{m}{2} \rfloor }(0) = 0$. Then $u_s$ defined by \eqref{appx-eq:green-formula} belongs to $\mathcal{C}([0, +\infty[, H^m_\lambda (\mathbb R_+))$ and satisfies
    \begin{equation*}
        \| \,\partial_y^j u_s(t , \cdot ) \,\|_{L^2_\lambda(\mathbb R_+)}
        \leq 
        e^{2 \lambda^2 t} 
        \big\| \,U_{s}^{(j)}\, \|_{L^2_\lambda (\mathbb R_+)},
    \end{equation*}
    for any $t\geq 0$ and $j\in \{0, \dots, m\}$.
\end{lemma}
\begin{proof}
We denote by $\tilde{U}_s$ and $\tilde{u}_s$ the odd extensions of $U_s$ and $u_s$ to the whole line $y \in \mathbb{R}$. Due to the compatibility conditions of $U_s$ at $y=0$, we have $\tilde U_s \in H^m(\mathbb{R})$. Since $\tilde{u}_s$ solves the heat equation on the whole line with initial datum $\tilde{U}_s$, it follows that $\tilde{u}_s \in \mathcal{C}(\mathbb{R}_+, H^m(\mathbb{R}))$. Moreover, $\tilde{u}_s$ and its derivatives can be expressed using the heat kernel:
\begin{equation*}
    \partial_y^j
    \tilde u_s(t,y) = \frac{1}{\sqrt{4\pi t}} \int_{\mathbb R} e^{-\frac{|y-z|^2}{4t}} \tilde U_s^{(j)}(z) dz,
\end{equation*}
for all $t>0$, $y \in  \mathbb R$ and $j \in \{0, \dots,m \}$. We apply the triangular inequality to the following expression:
    \begin{align*}
       \big|\,e^{\lambda |y|}  \partial_y^j \tilde u_s(t,y)\,\big| 
       &= 
       \bigg|\,
       \frac{1}{\sqrt{4\pi t}} \int_{\mathbb R} 
       e^{-\frac{|y-z|^2}{4t} +\lambda |y- z +z |} \tilde U_s^{(j)} (z) d z 
       \,\bigg|
       \leq 
       \int_{\mathbb R} K_\lambda (t, y-z)
       e^{\lambda |z| } |\, \tilde U_s^{(j)}(z) |dz
       ,
    \end{align*}
    where $K_\lambda$ denotes the kernel $K_\mu(t,z ):= \frac{1}{\sqrt{4\pi t}} e^{-\frac{z^2}{4t}+ \lambda |z|}$, for any $t>0$ and $z \in \mathbb R$. The $L^1$-norm of $K_\lambda$ satisfies
    \begin{align*}
        \int_{\mathbb R} | \, K_\lambda (t, z) \, |dz 
        &=
        \frac{1}{\sqrt{\pi t}}
        \int_{0}^\infty 
        e^{-\frac{z^2}{4t}+ \lambda z}dz
        = 
        \frac{e^{\lambda^2 t}}{\sqrt{\pi t}}
        \int_{0}^\infty 
        e^{-\left( \frac{z}{2 \sqrt{t}}  - \lambda \sqrt{t} \right)^2}dz = 
        e^{\lambda^2 t}
        \frac{2}{\sqrt{\pi}}
        \int_{-\lambda \sqrt{t}}^\infty 
        e^{-w^2}dw\\
        &
        =
        e^{\lambda^2 t}
        \bigg( 
            \frac{2}{\sqrt{\pi}}
            \int_{0}^\infty 
            e^{-w^2}dw
            +
            \frac{2}{\sqrt{\pi}}
            \int_{0}^{\lambda \sqrt{t}}
            e^{-w^2}dw
        \bigg)
        =
        e^{\lambda^2 t} \Big( 1 + \erf\big( \lambda \sqrt{t}  \big) \Big)\leq e^{2 \lambda^2 t} .
    \end{align*}
    Thanks to Young's inequality, $\partial_y^j \tilde u_s(t, \cdot)  \in  L^2_\lambda (\mathbb R)$ for any $t>0$ and satisfies 
    \begin{align*}
        \big\| \, \partial_y^j \tilde u_s(t, \cdot)\, \big\|_{L^2_\lambda(\mathbb R)}
        \leq 
        \| K_\lambda (t, \cdot ) \|_{L^1(\mathbb R)} 
        \|  \,\tilde U_s^{(j)} \|_{L^2_\lambda (\mathbb R)}
        \leq  
        2 \, e^{2 \lambda^2 t} \big\| \,U_s^{(j)} \|_{L^2_\lambda (\mathbb R_+)}.
    \end{align*}
     It remains to prove that $\tilde{u}_s$ in $\mathcal{C}([0, \infty[, H^m_\lambda(\mathbb{R}))$. Thanks to Scheffé's lemma, $K_\lambda \in \mathcal{C}( \, ]0, \infty[, L^1(\mathbb R))$ which implies that $\partial_y^j \tilde{u}_s$ in $\mathcal{C}(]0, \infty[, L^2_\lambda(\mathbb{R}))$ for any $j \in \{0, \dots, m\}$. As $t\to 0+$, $K_\lambda$ behaves like the heat kernel and converges distributionally to the Dirac-delta:
    \begin{equation*}
        \lim_{t\to 0+ } K_\lambda(t, \cdot ) = \delta_0,\quad \text{in} \quad \mathcal{D}'(\mathbb R).
    \end{equation*}
    The continuity at $t = 0$ is hence achieved by the density of $\mathcal{D}(\mathbb R)$ in $H^m_\lambda(\mathbb R)$ and the boundness of $\| K_\lambda(t, \cdot) \|_{L^1(\mathbb R)}$ around $t= 0$. Details are left to the reader.
\end{proof}

\section{A comparison with the regular and shear-layer velocities}\label{sec:comparison-with-shear-layer}
    
\noindent 
In this final section, we compare our matched asymptotic profile $\phi_{\text{inn}, k}+ \phi_{\text{out}, k}$ introduced in \Cref{def:inner-outer-profiles} with the ``regular'' and ``shear-layer'' profiles $v_\ee^{reg}+ v_\ee^{sl}$ proposed by G\'erard-Varet and Dormy in \cite{MR2601044}. 

\smallskip 
\noindent 
We first introduce some notation from Section 4 of \cite{MR2601044}, supported by explicit formulas derived in \cite{DeAnnaKortum2025}. Let $\ee = 1/k >0$ and and let $\tau \in \mathbb C$ be as in (1.7) of \cite{MR2601044}:  ${\rm Im}(\tau)<0$ and there exists a function $W \in \mathcal{C}^\infty(\mathbb R, \mathbb C)$ satisfying 
\begin{equation*}
    (\tau-z^2)^2 \frac{d}{dz} W(z) + i \frac{d^3}{dz^2} \big( (\tau-z^2) W(z) ) = 0,\qquad \lim_{z \to -\infty} W(z) = 0,\qquad \lim_{z \to +\infty} W(z) = 1.
\end{equation*}
In Corollary 1.2 of \cite{DeAnnaKortum2025}, we showed that there is a unique solution $(\tau, W)$ given by $\tau = -e^{\frac{i \pi}{4}}$ and
\begin{equation*}
    W(z) = \frac{\Upsilon\big( \zeta \big) }{1 +\zeta^2  } = 
     \frac{1}{\sqrt{2\pi}}\frac{\zeta}{1 +\zeta^2}\, e^{-\frac{\zeta^2}{2}} +
    \frac{1}{2}  
    \left(1 + \erf\left( \frac{\zeta}{\sqrt{2}} \right) \right)
    ,\qquad \text{with }\zeta = e^{-\frac{i \pi}{8}} z \in \mathbb C.
\end{equation*}
The function $\sigma_k(t)$ of \Cref{def:inner-outer-profiles} is in \cite{MR2601044} denoted by $i\, \ee^{-1} \omega(\ee, t)$, as $\omega(\ee, t)$ is defined as
\begin{equation*}
    \omega(\ee, t) :=
    -u_s(t,a(t)) + 
    \frac{\ee^{1/2}}{\sqrt{2}}
    |\partial_y^2 u_s(t,a(t))|^\frac{1}{2}\tau 
    = 
    -u_s(t,a(t)) + 
    e^{i\frac{5 \pi}{4}}
    \sqrt{\frac{|\partial_y^2 u_s(t,a(t))|}{2k}}
    = 
    \frac{\sigma_k(t)}{i k}.
\end{equation*}
The regular and shear layer velocities are defined in Section 4 of \cite{MR2601044} as follows:
\begin{equation}\label{eq:shear-and-regular-layers}
\begin{alignedat}{4}
    &v_\ee^{reg}(t,y) 
    &&= 
    H(y-a(t))
    \bigg(  
        u_s(t,y) 
        -u_s(t,a(t)) + 
        \frac{\ee^{1/2}}{\sqrt{2}}
        |\partial_y^2 u_s(t,a(t))|^\frac{1}{2}\tau 
    \bigg),\\
    &v_\ee^{sl}(t,y) 
    &&= 
    \frac{\ee^{1/2}}{\sqrt{2}}
    \varphi(y-a(t)) 
    |\partial_y^2 u_s(t,a(t))|^{1/2}
    V
    \left(  
      |\partial_y^2 u_s(t,a(t))|^{1/4} \frac{y-a(t)}{(2\ee)^{1/4}}
    \right),
\end{alignedat}    
\end{equation}
where $\varphi $ is a smooth truncation function near $0$ and $V(z) =(\tau-\zeta^2)(W(z)-H(z))$ for any $z \in \mathbb R$. 

\noindent 
We compare these profiles with our matched asymptotic construction $\phi_{{\rm inn},k} + \phi_{{\rm out},k}$ via the following lemma.

\begin{lemma}
    Replace $\varphi(y-a(t))$ in \eqref{eq:shear-and-regular-layers} with  $\chi(y)$ of \Cref{def:uink-into}. Define $v_\ee^{new} : [0, T_{\rm max}] \times \mathbb R_+ \to \mathbb C$
    \begin{equation*}
         v_\ee^{new}(t,y) 
         :=
         \chi(y)
         H(y-a(t))
         \bigg( 
            u_s(t,a(t)) + 
            \frac{\partial_y^2 u_s(t,y(a))}{2}(y-a(t))^2
            -
            u_s(t,y)
         \bigg).
\end{equation*}
Then, for any $(t,y) \in [0, T_{\rm max}] \times \mathbb R_+$, the following identity holds true:
\begin{equation}\label{eq:relation-our-profiles-GV-profiles}
    \phi_{{\rm inn}, k}(t,y)+\phi_{{\rm out}, k}(t,y) 
    = 
    \Big( 
        v_\ee^{reg}(t,y)+ 
        v_\ee^{sl}(t,y) + 
        v_\ee^{new}(t,y)
    \Big) e^{i \ee^{-1}\int_0^t \omega(\ee, s) ds}.
\end{equation}
\end{lemma}
\begin{proof}
We first rewrite the outer layer $\phi_{{\rm out}, k}$ of \Cref{def:inner-outer-profiles} as
\begin{equation}\label{eq:final-sec-phi-out-v-reg}
\begin{aligned}
    \phi_{{\rm out}, k}(t,y) 
    &= 
    \big(1 - \chi(y)\big)
        H\big(y-a(t)\big)
        \bigg( 
            u_s(t,y)
            +
            \frac{\sigma_k(t)}{\im k}
         \bigg)
        e^{\int_0^t \sigma_k(\tau) d\tau }
        \\
        &=v_{\ee}^{reg}(t,y) e^{i \ee^{-1}\int_0^t \omega(\ee, s) ds} 
        -
        \chi(y) H\big(y-a(t)\big)
        \bigg( 
            u_s(t,y)
            +
            \frac{\sigma_k(t)}{\im k}
         \bigg)
        e^{\int_0^t \sigma_k(\tau) d\tau }.
\end{aligned}
\end{equation}
Next, we recast $V(z)$ in terms of $\Upsilon(\zeta)$ with respect to the relation $\zeta = e^{-\frac{i \pi}{8}} z\in \mathbb C$. Recalling that $\tau = -e^{\frac{i \pi}{4}} = e^{i \frac{5 \pi}{4}}$:
\begin{equation*}
\begin{aligned}
        V(z) 
        &= 
        \big( \tau -z^2 \big) \Big( W(z) - H(z) \Big) = 
        e^{i\frac{5 \pi}{4}} \big( 1+e^{ - i \frac{\pi}{4}} z^2  \big) 
        \big( W(z) - H(z) \big) 
        \\
        &
        =
        e^{i\frac{5\pi}{4}}
        \big( 1 + \zeta^2 \big) 
        \bigg( 
        \frac{\Upsilon \big( \zeta \big) }{1+ \zeta^2}
           - H(z)  
        \bigg)
        = 
        e^{i\frac{5\pi}{4}} \bigg( \Upsilon \big( \zeta \big) - (1+\zeta^2) H(z) \bigg).
\end{aligned}
\end{equation*}
Hence, using the short notation $\beta(t) = \partial_y^2 u_s(t,a(t))/2<0$, we deduce from \eqref{eq:shear-and-regular-layers} (with $\chi$ instead of $\varphi$) that 
\begin{align*}
    v^{\rm sl}_\ee (t,y) 
    &
    e^{i \ee^{-1}\int_0^t \omega(\ee, s) ds} 
    =
    \frac{\ee^{1/2}}{\sqrt{2}}
    \chi(y) 
    |\beta(t)|^{1/2}
    V
    \left(  
      |\beta(t)|^{1/4} \frac{y-a(t)}{\ee^{1/4}}
    \right)
     e^{i \ee^{-1}\int_0^t \omega(\ee, s) ds}
     \\
    &= 
    \sqrt{\frac{|\beta(t)|}{2k}}
    \chi(y) 
    e^{i\frac{5\pi}{4}}
    \left(  
     \Upsilon \left( \sqrt[4]{|\beta(t)|k} (y-a(t))  e^{-\frac{i\pi}{8}} \right) 
     -
     \Big( 1 + e^{-i \frac{\pi}{4}}\sqrt{k |\beta(t)|} (y-a(t))^2 \Big) 
    \right)
     e^{\int_0^t \sigma_k(\tau) d\tau }.
\end{align*}
From the definition of $\phi_{\rm inn, k}$ in \Cref{def:inner-outer-profiles} we finally derive:
\begin{align*}
    \phi_{\text{inn}, k} (t,y)
    = 
    \,
    &v^{\rm sl}_\ee (t,y) e^{i \ee^{-1}\int_0^t \omega(\ee, s) ds} 
    +
    \\
    &+\chi(y) H(y-a(t)) 
    \Bigg(
              \frac{\sigma_k(t)}{ik} + u_s(t,a(t))
              +
              \frac{\partial_y^2 u_s(t,a(t))}{2}
              (y-a(t))^2
    \Bigg)e^{\int_0^t \sigma_k(s) ds }.
\end{align*}
By summation, this last identity with \eqref{eq:final-sec-phi-out-v-reg}, we establish \eqref{eq:relation-our-profiles-GV-profiles}. This concludes the proof of the lemma.
\end{proof} 
\begin{remark}
We conclude with a comment on the role of the additional term $v_\ee^{\mathrm{new}}$ in \eqref{eq:relation-our-profiles-GV-profiles}. Since $\phi_{{\rm inn}, k}$ and $\phi_{{\rm out}, k}$ are smooth in $y=a(t)$ for $t>0$, if one considers just the approximate profile $v^{sl}_\ee  + v^{reg}_\ee $ of \cite{MR2601044} then 
\begin{equation*}
\partial_y^4 \bigl(v^{sl}_\ee + v^{reg}_\ee\bigr)\big|_{y = a(t)}
 = \partial_y^4  v^{new}_\ee\big|_{y = a(t)} + \partial_y^4 
 \bigl(\phi_{{\rm inn}, k}  + \phi_{{\rm out}, k} \bigr)\big|_{y = a(t)}
 \sim   \partial_y^3 u_s(t,a(t))  \, \delta_{a(t)} .
\end{equation*}
As a consequence, whenever $\partial_y^3 u_s(t,a(t)) \neq 0$ for some $t>0$, $ v^{sl}_\ee + v^{reg}_\ee$ is at most in $W^{3,\infty}$ near $y = a(t)$.  In other words, if $\partial_y^3 u_s(t,a(t)) \neq 0$, one cannot rely solely on the profile of \cite{MR2601044}  to establish ill-posedness in Sobolev spaces with $y$-regularity higher than $W^{3,\infty}$ for $v$.

\noindent 
Nevertheless, we remark that, in the autonomous case $u_s(t,y) = U_s(y) $ with $U_s$ quadratic around $y = a_0$, this singularity never occurs ($v_\ee^{new} \equiv 0$). By contrast, if $u_s$ satisfies the heat equation, the singularity does occur.

\smallskip 
\noindent 
On the other hand, this issue is resolved by our work, introducing the correction $v^{\mathrm{new}}_\ee$. The mechanism becomes much more transparent when working with our profiles $\phi_{{\rm inn}, k}$ and $\phi_{{\rm out}, k}$ in \Cref{def:inner-outer-profiles}, which are smooth by construction. This is one of the reasons motivating our use of matched asymptotics.
\end{remark}

\bigskip
\noindent 
\textbf{Acknowledgements:} We would like to thank David Gérard-Varet for his invitation and warm hospitality at the UFR de Mathématiques and IMJ-PRG, Université Paris Cité. We are also grateful to David Gérard-Varet for several fruitful discussions on the subject of this work. This research did not receive any specific grant from funding agencies in the public, commercial, or not-for-profit sectors.

\bigskip
\noindent 
\textbf{Author contributions:}\\  
Francesco De Anna: conceptualization, methodology, formal analysis, writing – original draft.\\
Joshua Kortum: conceptualization, formal analysis, writing - review \& editing.

\bigskip
\noindent 
\textbf{Declaration of competing interests:} \\
No known competing interests.

\bigskip
\noindent 
\textbf{Declaration of generative AI and AI-assisted technologies in the manuscript preparation process:}\\
During the preparation of this work the authors used M365 Copilot (based on the GPT-5 reasoning model) to correct orthographic, grammatical, or typographical inaccuracies. After using this tool, the authors reviewed and edited the content as needed and take full responsibility for the content of the published article.

\bigskip
\noindent
\textbf{Data availability:}\\
No data was used for the research described in the article.

\bibliographystyle{abbrv}
\bibliography{literature}

@article{DeAnnaKortum2025,
  author        = {Francesco De Anna and Joshua Kortum},
  title         = {Linear Instability of the Prandtl Equations via Hypergeometric Functions and the Harmonic Oscillator},
  journal       = {arXiv preprint arXiv:2503.02417},
  year          = {2025},
  doi           = {10.48550/arXiv.2503.02417},
  primaryClass  = {math.AP},
  eprint        = {2503.02417},
  archivePrefix = {arXiv},
  url           = {https://doi.org/10.48550/arXiv.2503.02417}
}

@book{holmes2012introduction,
  title={Introduction to perturbation methods},
  author={Holmes, Mark H},
  year={2012},
  publisher={Springer Science \& Business Media}
}

@article {MR2601044,
    AUTHOR = {G\'erard-Varet, David and Dormy, Emmanuel},
     TITLE = {On the ill-posedness of the {P}randtl equation},
   JOURNAL = {J. Amer. Math. Soc.},
  FJOURNAL = {Journal of the American Mathematical Society},
    VOLUME = {23},
      YEAR = {2010},
    NUMBER = {2},
     PAGES = {591--609},
      ISSN = {0894-0347,1088-6834},
   MRCLASS = {35Q30 (35B65 76D05 76D10)},
  MRNUMBER = {2601044},
MRREVIEWER = {Ben\ W.\ Schweizer},
       DOI = {10.1090/S0894-0347-09-00652-3},
       URL = {https://doi.org/10.1090/S0894-0347-09-00652-3},
}

@article {MR3864769,
    AUTHOR = {Liu, Cheng-Jie and Xie, Feng and Yang, Tong},
     TITLE = {A note on the ill-posedness of shear flow for the {MHD}
              boundary layer equations},
   JOURNAL = {Sci. China Math.},
  FJOURNAL = {Science China. Mathematics},
    VOLUME = {61},
      YEAR = {2018},
    NUMBER = {11},
     PAGES = {2065--2078},
      ISSN = {1674-7283,1869-1862},
   MRCLASS = {35Q35 (35B35 76D10 76W05)},
  MRNUMBER = {3864769},
       DOI = {10.1007/s11425-017-9306-0},
       URL = {https://doi.org/10.1007/s11425-017-9306-0},
}

@article {MR3925144,
    AUTHOR = {Dietert, Helge and G\'erard-Varet, David},
     TITLE = {Well-posedness of the {P}randtl equations without any
              structural assumption},
   JOURNAL = {Ann. PDE},
  FJOURNAL = {Annals of PDE. Journal Dedicated to the Analysis of Problems
              from Physical Sciences},
    VOLUME = {5},
      YEAR = {2019},
    NUMBER = {1},
     PAGES = {Paper No. 8, 51},
      ISSN = {2524-5317,2199-2576},
   MRCLASS = {76D03 (35B30 35Q35 76D10)},
  MRNUMBER = {3925144},
MRREVIEWER = {S\'ebastien\ J.\ Boyaval},
       DOI = {10.1007/s40818-019-0063-6},
       URL = {https://doi.org/10.1007/s40818-019-0063-6},
}

@article {MR3670620,
    AUTHOR = {Liu, Cheng-Jie and Yang, Tong},
     TITLE = {Ill-posedness of the {P}randtl equations in {S}obolev spaces
              around a shear flow with general decay},
   JOURNAL = {J. Math. Pures Appl. (9)},
  FJOURNAL = {Journal de Math\'ematiques Pures et Appliqu\'ees. Neuvi\`eme
              S\'erie},
    VOLUME = {108},
      YEAR = {2017},
    NUMBER = {2},
     PAGES = {150--162},
      ISSN = {0021-7824,1776-3371},
   MRCLASS = {35Q35 (35M13 76D03 76D10)},
  MRNUMBER = {3670620},
MRREVIEWER = {Beno\^it\ P.\ Desjardins},
       DOI = {10.1016/j.matpur.2016.10.014},
       URL = {https://doi.org/10.1016/j.matpur.2016.10.014},
}

@article {MR3458159,
    AUTHOR = {Liu, Cheng-Jie and Wang, Ya-Guang and Yang, Tong},
     TITLE = {On the ill-posedness of the {P}randtl equations in
              three-dimensional space},
   JOURNAL = {Arch. Ration. Mech. Anal.},
  FJOURNAL = {Archive for Rational Mechanics and Analysis},
    VOLUME = {220},
      YEAR = {2016},
    NUMBER = {1},
     PAGES = {83--108},
      ISSN = {0003-9527,1432-0673},
   MRCLASS = {35Q30 (35B30 35R25 76D09)},
  MRNUMBER = {3458159},
MRREVIEWER = {Amin\ Esfahani},
       DOI = {10.1007/s00205-015-0927-1},
       URL = {https://doi.org/10.1007/s00205-015-0927-1},
}

@article {MR2952715,
    AUTHOR = {G\'erard-Varet, D. and Nguyen, T.},
     TITLE = {Remarks on the ill-posedness of the {P}randtl equation},
   JOURNAL = {Asymptot. Anal.},
  FJOURNAL = {Asymptotic Analysis},
    VOLUME = {77},
      YEAR = {2012},
    NUMBER = {1-2},
     PAGES = {71--88},
      ISSN = {0921-7134,1875-8576},
   MRCLASS = {35Q53 (35R25)},
  MRNUMBER = {2952715},
       DOI = {10.3233/asy-2011-1075},
       URL = {https://doi.org/10.3233/asy-2011-1075},
}

@incollection {MR3916787,
    AUTHOR = {Maekawa, Yasunori and Mazzucato, Anna},
     TITLE = {The inviscid limit and boundary layers for {N}avier-{S}tokes
              flows},
 BOOKTITLE = {Handbook of mathematical analysis in mechanics of viscous
              fluids},
     PAGES = {781--828},
 PUBLISHER = {Springer, Cham},
      YEAR = {2018},
      ISBN = {978-3-319-13344-7; 978-3-319-13343-0; 978-3-319-13345-4},
   MRCLASS = {76D05 (35Q30 76D09 76D10)},
  MRNUMBER = {3916787},
       DOI = {10.1007/978-3-319-13344-7\_1},
       URL = {https://doi.org/10.1007/978-3-319-13344-7_1},
}

@article {MR4465902,
    AUTHOR = {Li, Wei-Xi and Masmoudi, Nader and Yang, Tong},
     TITLE = {Well-posedness in {G}evrey function space for 3{D} {P}randtl
              equations without structural assumption},
   JOURNAL = {Comm. Pure Appl. Math.},
  FJOURNAL = {Communications on Pure and Applied Mathematics},
    VOLUME = {75},
      YEAR = {2022},
    NUMBER = {8},
     PAGES = {1755--1797},
      ISSN = {0010-3640,1097-0312},
   MRCLASS = {76D10},
  MRNUMBER = {4465902},
MRREVIEWER = {Xinhui\ Si},
       DOI = {10.1002/cpa.21989},
       URL = {https://doi.org/10.1002/cpa.21989},
}

@article {MR3855356,
    AUTHOR = {G\'erard-Varet, David and Maekawa, Yasunori and Masmoudi,
              Nader},
     TITLE = {Gevrey stability of {P}randtl expansions for 2-dimensional
              {N}avier-{S}tokes flows},
   JOURNAL = {Duke Math. J.},
  FJOURNAL = {Duke Mathematical Journal},
    VOLUME = {167},
      YEAR = {2018},
    NUMBER = {13},
     PAGES = {2531--2631},
      ISSN = {0012-7094,1547-7398},
   MRCLASS = {35Q30 (35Q35 76D05)},
  MRNUMBER = {3855356},
MRREVIEWER = {Luc\ Paquet},
       DOI = {10.1215/00127094-2018-0020},
       URL = {https://doi.org/10.1215/00127094-2018-0020},
}

@article {MR3429469,
    AUTHOR = {Gerard-Varet, David and Masmoudi, Nader},
     TITLE = {Well-posedness for the {P}randtl system without analyticity or
              monotonicity},
   JOURNAL = {Ann. Sci. \'Ec. Norm. Sup\'er. (4)},
  FJOURNAL = {Annales Scientifiques de l'\'Ecole Normale Sup\'erieure.
              Quatri\`eme S\'erie},
    VOLUME = {48},
      YEAR = {2015},
    NUMBER = {6},
     PAGES = {1273--1325},
      ISSN = {0012-9593,1873-2151},
   MRCLASS = {35Q35 (35B30 76D05)},
  MRNUMBER = {3429469},
       DOI = {10.24033/asens.2270},
       URL = {https://doi.org/10.24033/asens.2270},
}

@article {MR2849481,
    AUTHOR = {Guo, Yan and Nguyen, Toan},
     TITLE = {A note on {P}randtl boundary layers},
   JOURNAL = {Comm. Pure Appl. Math.},
  FJOURNAL = {Communications on Pure and Applied Mathematics},
    VOLUME = {64},
      YEAR = {2011},
    NUMBER = {10},
     PAGES = {1416--1438},
      ISSN = {0010-3640,1097-0312},
   MRCLASS = {35Q30 (35B30 35C20 76D05 76D10)},
  MRNUMBER = {2849481},
MRREVIEWER = {Pavel\ I.\ Naumkin},
       DOI = {10.1002/cpa.20377},
       URL = {https://doi.org/10.1002/cpa.20377},
}

@article {MR3765768,
    AUTHOR = {Chen, Dongxiang and Wang, Yuxi and Zhang, Zhifei},
     TITLE = {Well-posedness of the {P}randtl equation with monotonicity in
              {S}obolev spaces},
   JOURNAL = {J. Differential Equations},
  FJOURNAL = {Journal of Differential Equations},
    VOLUME = {264},
      YEAR = {2018},
    NUMBER = {9},
     PAGES = {5870--5893},
      ISSN = {0022-0396,1090-2732},
   MRCLASS = {35Q35 (35B30)},
  MRNUMBER = {3765768},
       DOI = {10.1016/j.jde.2018.01.024},
       URL = {https://doi.org/10.1016/j.jde.2018.01.024},
}

@article{lombardo2003well,
  title={Well-posedness of the boundary layer equations},
  author={Lombardo, M.~C. and Cannone, M. and Sammartino, M.},
  journal={SIAM J. Math. Anal.},
  volume={35},
  number={4},
  pages={987--1004},
  year={2003},
  publisher={SIAM}
}

@article{oleinik1963prandtl,
  title={On the mathematical theory of boundary layer for an unsteady flow of incompressible fluid},
  author={Oleinik, Olga Arsen’evna},
  journal={Journal of Applied Mathematics and Mechanics},
  volume={30},
  number={5},
  pages={951--974},
  year={1966},
  publisher={Elsevier}
}

@article {MW2015,
    AUTHOR = {Masmoudi, N. and Wong, T.~K.},
     TITLE = {Local-in-time existence and uniqueness of solutions to the
              {P}randtl equations by energy methods},
   JOURNAL = {Comm. Pure Appl. Math.},
  FJOURNAL = {Communications on Pure and Applied Mathematics},
    VOLUME = {68},
      YEAR = {2015},
    NUMBER = {10},
     PAGES = {1683--1741},
}

@article{AWXY2015,
    AUTHOR = {Alexandre, R. and Wang, Y.-G. and Xu, C.-J. and Yang, T.},
     TITLE = {Well-posedness of the {P}randtl equation in {S}obolev spaces},
   JOURNAL = {J. Amer. Math. Soc.},
  FJOURNAL = {Journal of the American Mathematical Society},
    VOLUME = {28},
      YEAR = {2015},
    NUMBER = {3},
     PAGES = {745--784},
}

@article {LY2020,
    AUTHOR = {Li, W.-X. and   Yang, T.},
     TITLE = {Well-posedness in {G}evrey function spaces for the {P}randtl equations with non-degenerate critical points},
   JOURNAL = {J. Eur. Math. Soc. },
  FJOURNAL = {JEMS},
    VOLUME = {22},
      YEAR = {2020},
    NUMBER = {3},
     PAGES = {717--775},
}

@article {MR4773381,
    AUTHOR = {Pan, Xinghong and Xu, Chao-Jiang},
     TITLE = {Long-time existence of {G}evrey-2 solutions to the 3{D}
              {P}randtl boundary layer equations},
   JOURNAL = {Commun. Math. Sci.},
  FJOURNAL = {Communications in Mathematical Sciences},
    VOLUME = {22},
      YEAR = {2024},
    NUMBER = {5},
     PAGES = {1203--1250},
      ISSN = {1539-6746,1945-0796},
   MRCLASS = {35Q35 (76D03 76D10)},
  MRNUMBER = {4773381},
MRREVIEWER = {Jingyang\ Shu},
       DOI = {10.4310/cms.2024.v22.n5.a3},
       URL = {https://doi.org/10.4310/cms.2024.v22.n5.a3},
}

@article {MR3634071,
    AUTHOR = {Guo, Yan and Nguyen, Toan T.},
     TITLE = {Prandtl boundary layer expansions of steady {N}avier-{S}tokes
              flows over a moving plate},
   JOURNAL = {Ann. PDE},
  FJOURNAL = {Annals of PDE. Journal Dedicated to the Analysis of Problems
              from Physical Sciences},
    VOLUME = {3},
      YEAR = {2017},
    NUMBER = {1},
     PAGES = {Paper No. 10, 58},
      ISSN = {2524-5317,2199-2576},
   MRCLASS = {35Q30 (35B35 35C20)},
  MRNUMBER = {3634071},
MRREVIEWER = {Gudrun\ Th\"ater},
       DOI = {10.1007/s40818-016-0020-6},
       URL = {https://doi.org/10.1007/s40818-016-0020-6},
}

@article {MR3961300,
    AUTHOR = {Gerard-Varet, David and Maekawa, Yasunori},
     TITLE = {Sobolev stability of {P}randtl expansions for the steady
              {N}avier-{S}tokes equations},
   JOURNAL = {Arch. Ration. Mech. Anal.},
  FJOURNAL = {Archive for Rational Mechanics and Analysis},
    VOLUME = {233},
      YEAR = {2019},
    NUMBER = {3},
     PAGES = {1319--1382},
      ISSN = {0003-9527,1432-0673},
   MRCLASS = {76D05 (35B35 35Q35)},
  MRNUMBER = {3961300},
MRREVIEWER = {Matteo\ Caggio},
       DOI = {10.1007/s00205-019-01380-x},
       URL = {https://doi.org/10.1007/s00205-019-01380-x},
}

@article {MR4592099,
    AUTHOR = {Guo, Yan and Iyer, Sameer},
     TITLE = {Steady {P}randtl layer expansions with external forcing},
   JOURNAL = {Quart. Appl. Math.},
  FJOURNAL = {Quarterly of Applied Mathematics},
    VOLUME = {81},
      YEAR = {2023},
    NUMBER = {2},
     PAGES = {375--411},
      ISSN = {0033-569X,1552-4485},
   MRCLASS = {76D10 (35Q35)},
  MRNUMBER = {4592099},
MRREVIEWER = {Mohammed\ Guedda},
       DOI = {10.1090/qam/1655},
       URL = {https://doi.org/10.1090/qam/1655},
}

@article {MR4715239,
    AUTHOR = {Grenier, Emmanuel and Nguyen, Toan T.},
     TITLE = {On nonlinear instability of {P}randtl's boundary layers: the
              case of {R}ayleigh's stable shear flows},
   JOURNAL = {J. Math. Pures Appl. (9)},
  FJOURNAL = {Journal de Math\'ematiques Pures et Appliqu\'ees. Neuvi\`eme
              S\'erie},
    VOLUME = {184},
      YEAR = {2024},
     PAGES = {71--90},
      ISSN = {0021-7824,1776-3371},
   MRCLASS = {35Q30 (35B25 35Q35 76D10 76E05)},
  MRNUMBER = {4715239},
MRREVIEWER = {Yoshikazu\ Giga},
       DOI = {10.1016/j.matpur.2024.02.001},
       URL = {https://doi.org/10.1016/j.matpur.2024.02.001},
}

@article {MR1761409,
    AUTHOR = {Grenier, Emmanuel},
     TITLE = {On the nonlinear instability of {E}uler and {P}randtl
              equations},
   JOURNAL = {Comm. Pure Appl. Math.},
  FJOURNAL = {Communications on Pure and Applied Mathematics},
    VOLUME = {53},
      YEAR = {2000},
    NUMBER = {9},
     PAGES = {1067--1091},
      ISSN = {0010-3640,1097-0312},
   MRCLASS = {76E30 (35B35 35Q35 76B03 76D10)},
  MRNUMBER = {1761409},
MRREVIEWER = {Marcel\ Oliver},
       DOI = {10.1002/1097-0312(200009)53:9<1067::aid-cpa1>3.0.co;2-q},
       URL =
              {https://doi.org/10.1002/1097-0312(200009)53:9<1067::aid-cpa1>3.0.co;2-q},
}

@article {MR3566199,
    AUTHOR = {Grenier, Emmanuel and Guo, Yan and Nguyen, Toan T.},
     TITLE = {Spectral instability of characteristic boundary layer flows},
   JOURNAL = {Duke Math. J.},
  FJOURNAL = {Duke Mathematical Journal},
    VOLUME = {165},
      YEAR = {2016},
    NUMBER = {16},
     PAGES = {3085--3146},
      ISSN = {0012-7094,1547-7398},
   MRCLASS = {35Q30 (35B25 35B35)},
  MRNUMBER = {3566199},
MRREVIEWER = {Joel\ David\ Avrin},
       DOI = {10.1215/00127094-3645437},
       URL = {https://doi.org/10.1215/00127094-3645437},
}
\nocite{*}

\end{document}